\documentclass[11pt,reqno]{amsart}
\usepackage{amsmath}
\usepackage{amsthm}
\usepackage{graphicx}
\usepackage{amsfonts}
\usepackage{amssymb}
\usepackage{hyperref}
\usepackage{bm}
\usepackage[margin=1.1in]{geometry}
\usepackage{enumerate}
\numberwithin{equation}{section}

\newtheorem{theorem}{Theorem}[section]
\newtheorem{lemma}[theorem]{Lemma}
\newtheorem{proposition}[theorem]{Proposition}
\newtheorem{corollary}[theorem]{Corollary}

\theoremstyle{definition}

\newtheorem{definition}[theorem]{Definition}
\newtheorem{remark}[theorem]{Remark}

\newtheorem{innercustomthm}{Assumption}
\newenvironment{assumption}[1]
  {\renewcommand\theinnercustomthm{#1}\innercustomthm}
  {\endinnercustomthm}

\def\E{{\mathbb E}}
\def\R{{\mathbb R}}
\def\N{{\mathbb N}}
\def\FF{{\mathbb F}}
\def\PP{{\mathbb P}}

\def\QQ{{\mathbb Q}}
\def\P{{\mathcal P}}

\def\V{{\mathcal V}}

\def\X{{\mathcal X}}

\def\L{{\mathcal L}}

\def\W{{\mathcal W}}
\def\AB{{\mathbb A}}
\def\A{{\mathcal A}}
\def\AM{\mathcal{AM}}
\def\RC{\mathcal{R}}
\def\RCM{\mathcal{RM}}
\def\F{{\mathcal F}}

\def\C{{\mathcal C}}
\def\CP{{C([0,T];\P(\R^d))}}
\def\CE{{C([0,T];E)}}

\title[Closed-loop mean field game convergence]{On the convergence of closed-loop Nash equilibria to the mean field game limit}

\author{Daniel Lacker}
\address{Department of Industrial Engineering \& Operations Research, Columbia University}
\email{daniel.lacker@columbia.edu}

\begin{document}

\begin{abstract}
This paper continues the study of the mean field game (MFG) convergence problem: In what sense do the Nash equilibria of $n$-player stochastic differential games converge to the mean field game as $n\rightarrow\infty$? Previous work on this problem took two forms. First, when the $n$-player equilibria are open-loop, compactness arguments permit a characterization of all limit points of $n$-player equilibria as weak MFG equilibria, which contain additional randomness compared to the standard (strong) equilibrium concept. On the other hand, when the $n$-player equilibria are closed-loop, the convergence to the MFG equilibrium is known only when the MFG equilibrium is unique and the associated ``master equation" is solvable and sufficiently smooth. This paper adapts the compactness arguments to the closed-loop case, proving a convergence theorem that holds even when the MFG equilibrium is non-unique. Every limit point of $n$-player equilibria is shown to be the same kind of weak MFG equilibrium as in the open-loop case. Some partial results and examples are discussed for the converse question, regarding which of the weak MFG equilibria can arise as the limit of $n$-player (approximate) equilibria.
\end{abstract}

\maketitle

\tableofcontents

\section{Introduction}

The goal of this paper is to deepen the study of the $n\rightarrow\infty$ limit theory for $n$-player stochastic differential games of mean field type. To briefly summarize the problem, specified in full detail in Section \ref{se:mainresults}, suppose $n$ players have private state processes $\bm{X}=(X^1,\ldots,X^n)$ governed by the stochastic differential equation (SDE) system
\begin{align*}
dX^i_t &= b(t,X^i_t,\mu^n_t,\alpha^i(t,\bm{X}_t))dt + dW^i_t, \quad\quad \mu^n_t = \frac{1}{n}\sum_{k=1}^n\delta_{X^k_t},
\end{align*}
where $W^1,\ldots,W^n$ are independent Brownian motions, and $X^1_0,\ldots,X^n_0$ are i.i.d.
Note that the drift function $b$ is the same for each player, but the dynamics of player $i$'s state process depend only on $X^i$ itself, the empirical probability measure $\mu^n_t$ of all players' states, and the control $\alpha^i$ of player $i$. Each player chooses $\alpha^i$ from the set $\A$ of measurable functions from $[0,T] \times (\R^d)^n$ to the set $A$ of admissible actions. That is, each player's control is chosen as a (deterministic) function of time and the current states of all players. The goal of player $i$ is to maximize the expected payoff
\[
J^n_i(\alpha^1,\ldots,\alpha^n) = \E\left[\int_0^Tf(t,X^i_t,\mu^n_t,\alpha^i(t,\bm{X}_t))dt + g(X^i_T,\mu^n_T)\right],
\]
which takes the same symmetric form as the drift. The primary object of study is a \emph{closed-loop Markovian Nash equilibrium}, defined as any vector $(\alpha^1,\ldots,\alpha^n) \in \A^n$ such that
\[
J^n_i(\alpha^1,\ldots,\alpha^n) \ge \sup_{\beta \in \A}J^n_i(\alpha^1,\ldots,\alpha^{i-1},\beta,\alpha^{i+1},\ldots,\alpha^n).
\]
As is well known, Markovian Nash equilibria can be constructed by solving a parabolic PDE system, representing the value functions of each of the $n$ players, under suitable assumptions on the coefficients $(b,f,g)$; see \cite[Section 2.1.4]{CarmonaDelarue_book_I}.
For a more thorough introduction to stochastic differential games and mean field games, refer to the recent books \cite{CarmonaDelarue_book_I,CarmonaDelarue_book_II}.

A fundamental problem in mean field game (MFG) theory is to characterize the limiting behavior of Nash equilibrium as $n\rightarrow\infty$. More specifically, if $(\alpha^{n,1},\ldots,\alpha^{n,n})$ is a Nash equilibrium for each $n$, how does the associated empirical measure process $\mu^n=(\mu^n_t)_{t \in [0,T]}$ behave as $n\rightarrow\infty$? The heuristic put forth in the foundational work of \cite{lasrylions,lasry2006jeux1,lasry2006jeux2} and \cite{huang2006large,huang2007large} suggests that the limiting behavior should be captured by what we call in this paper the \emph{strong mean field equilibria}. A strong mean field equilibrium (or \emph{strong MFE}, defined precisely in Definition \ref{def:strongMFE-strict}) is a flow of probability measures $m=(m_t)_{t \in [0,T]}$ such that $m_t=\mathrm{Law}(X^*_t)$ for all $t \in [0,T]$, where $X^*$ is the optimal state process for the following stochastic control problem, in which $m$ is treated as fixed:
\begin{align}
\begin{split}
&\sup_{\alpha}\E\left[\int_0^Tf(t,X_t,m_t,\alpha(t,X_t))dt + g(X_T,m_T)\right], \\
&dX_t = b(t,X_t,m_t,\alpha(t,X_t))dt + dW_t.
\end{split}  \label{intro:MFE}
\end{align}
The majority of the MFG literature focuses on questions of existence and uniqueness of equilibria, though there is by now a decent understanding of this convergence problem. Early results \cite{lasrylions,lasry2006jeux1,feleqi-MFGderivation} confirmed the MFE as the relevant limiting concept but imposed strong restrictions  on the controls $\alpha^{n,i}$, requiring them to be of the distributed form $\alpha^{n,i}(t,x_1,\ldots,x_n) = \widehat\alpha^{n,i}(t,x_i)$.

The first comprehensive results came in \cite{lacker2016general,fischer2017connection}, but notably working with the distinct (and typically simpler to analyze) concept of \emph{open-loop} controls. In an open-loop equilibrium, each player specifies a control as a function of the noises $(W^1,\ldots,W^n)$ rather than the states $(X^1,\ldots,X^n)$, and this results in completely different equilibria. See \cite[pp. 72-76]{CarmonaDelarue_book_I} for a careful discussion of the differences between the open-loop and closed-loop regimes, which we will review briefly in Section \ref{se:closedloopvsopenloop}. For open-loop equilibria, the results of \cite{lacker2016general} give a rather complete picture of the $n\rightarrow\infty$ behavior: Even when the MFG equilibrium is non-unique, we can still characterize all subsequential limits of $\mu^n$ as MFG equilibria, as long as we work with a suitable \emph{weak equilibrium} concept. Conversely, each of these weak equilibria can arise as the limit of $\mu^n$, for a suitable choice of approximate $n$-player open-loop equilibria.

Our understanding of the $n\rightarrow\infty$ behavior of closed-loop equilibria is much less complete in the closed-loop regime, but a major breakthrough came with the work of Cardaliaguet et al. \cite{cardaliaguet-delarue-lasry-lions} on the \emph{master equation}, an infinite-dimensional PDE that describes the value function of the mean field game.
It was shown in \cite[Section 6]{cardaliaguet-delarue-lasry-lions} how to use a smooth solution of the \emph{master equation} to prove that $\mu^n$ converges to the (unique, in their setting) MFG equilibrium $\mu=(\mu_t)_{t \in [0,T]}$; see also \cite[Section 6.3]{CarmonaDelarue_book_II}.
More recently, these ideas were refined in \cite{DelarueLackerRamananCLT,DelarueLackerRamananLDP} to derive a central limit theorem and a large deviation principle for $\mu^n$, as well as nonasymptotic bounds on various distances between $\mu^n$ and its limit $\mu$. The idea of using the master equation to prove limit theorems has proven to be powerful and fairly versatile, with \cite{bayraktar2017analysis,cecchin-pelino} adapting the idea to models with finite state space. 

The master equation approach, however, is limited in several ways. The most fundamental shortcoming is that it requires the MFG equilibrium to be unique. In game theory, uniqueness is of course the exception, not the rule, and the aforementioned papers leave open the intriguing question of how to describe the limiting behavior of $\mu^n$ when there are multiple MFG equilibria.
Furthermore, it is very challenging to produce a classical solution of the master equation, and this has been accomplished so far only in quite restricted settings \cite{cardaliaguet-delarue-lasry-lions,chassagneux2014probabilistic,gangbo2015existence}.

This paper fills the gap between the open-loop and closed-loop regimes by proving (in Theorem \ref{th:mainlimit})
a limit theorem for closed-loop equilibrium which is general enough to accommodate non-unique MFG equilibrium. Under suitable assumptions, we show that the sequence of empirical measure flows $(\mu^n)$ is tight in a suitable space, and every limit in distribution is what we call a \emph{weak semi-Markov mean field equilibrium}, or \emph{weak MFE} for short. 
This equilibrium concept (given precisely in Definition \ref{def:MFEsemimarkov-strict}) differs from the standard strong MFE described above in three key respects:
\begin{itemize}
\item The deterministic measure flow $(m_t)_{t \in [0,T]}$ is replaced by a stochastic one $\mu=(\mu_t)_{t \in [0,T]}$.
\item The controls $\alpha$ in \eqref{intro:MFE} are \emph{semi-Markov}, meaning $\alpha=\alpha(t,X_t,\mu)$, where the dependence on the path $\mu=(\mu_t)_{t \in [0,T]}$ is nonanticipative.
\item The consistency condition $m_t=\mathrm{Law}(X_t)$ becomes conditional, $\mu_t = \mathrm{Law}(X_t \, | \, (\mu_s)_{s \le t})$.
\end{itemize}
The philosophy behind the weak MFE is no different from strong MFE: Each individual player treats the mean field $\mu$ as given, describing the distribution of states among an infinite (continuum) population of competing players. Each player reacts optimally to $\mu$, in a way that is consistent with (i.e., reproduces) the mean field $\mu$ when aggregated over the infinity of players. The stochastic measure flow $\mu$ should be thought of as an endogeneous common noise, in the sense that its randomness is felt equally by all of the players. An economist might refer to this as \emph{aggregate uncertainty} \cite{bergin1992anonymous}, as opposed to (exogeneous) \emph{aggregate shocks}, which one might produce by allowing correlations between the driving Brownian motions $W^1,\ldots,W^n$ as in \cite{carmona-delarue-lacker,cardaliaguet-delarue-lasry-lions}.

In fact, our main limit theorem applies also to \emph{closed-loop path-dependent equilibria}, in which each player can choose a control $\alpha^i=\alpha^i(t,(\bm{X}_s)_{s \le t})$ depending on the entire history of the $n$ state processes. Under modest convexity assumptions, we show that every Markovian equilibrium for the $n$-player game is also a path-dependent equilibrium, which allows us to study the two simultaneously. This seems to be the first MFG limit theorem for path-dependent equilibria.

In the special case where the weak MFE is unique, our main limit theorem becomes a proper convergence result: Every sequence of $n$-player (closed-loop) Nash equilibria converges to the unique weak MFE. In particular, we show that the well known monotonicity condition of Lasry-Lions \cite{lasrylions} is sufficient (Corollary \ref{co:uniqueness-propagation}), and this recovers and generalizes the aspects of the limit theorems of \cite{cardaliaguet-delarue-lasry-lions} pertaining to empirical measures.

The proof of our main limit theorem is based on probabilistic weak convergence and compactness arguments, as well as judicious use of Markovian projection arguments which allow one to ``mimick"  the time-$t$ marginal laws of a general It\^o process by a Markovian diffusion (see Theorem \ref{th:markovprojection}, quoted from \cite{gyongy1986mimicking,brunick2013mimicking}). In a sense, the techniques build on those developed in \cite{lacker2016general} for the open-loop regime, but the adaptation to the closed-loop case is highly non-trivial. The central difficulty of the closed-loop regime comes from the fact that a single player's change in strategy can have an outsized impact on the empirical measure due to the feedback through the other controls.
If we heuristically consider such a player to be \emph{influential}, the key idea behind our proof is that, in a certain averaged sense, \emph{not too many players can be simultaneously influential}. 
See Section \ref{se:proofideas} for an informal discussion of the proof.

Our notion of \emph{weak MFE} turns out to be equivalent in a certain sense to the notion of \emph{weak MFG solution} introduced in \cite{carmona-delarue-lacker,lacker2016general}. In particular, we encounter here the same interesting phenomenon explored in \cite[Section 3]{lacker2016general}, which is that \emph{not all weak MFE are mixtures of strong MFE}. 
That is, if we let $S$ denote the set of strong MFE $m = (m_t)_{t \in [0,T]}$ in the usual sense described in \eqref{intro:MFE} above, then there can exist weak MFE $\mu$ such that $\PP(\mu \in S) <  1$.
In particular, in settings with multiple MFG equilibria, the usual strong MFE concept is inadequate for describing the limiting behavior of $n$-player equilibria.
Notably, this phenomenon does not appear in McKean-Vlasov systems (i.e., uncontrolled systems) 
or in static mean field games (i.e., one-shot games with no time component).

The converse to our main limit theorem turns out to be challenging to address, and we have only partial results. We show in Theorem \ref{th:converselimit-strongMFE} that, under reasonable assumptions, every \emph{strong MFE} does indeed arise as the $n\rightarrow\infty$ limit of some sequence of approximate $n$-player (closed-loop) equilibria (see also \cite[Section 6.1]{CarmonaDelarue_book_II}). The question of if or when this is true for \emph{weak MFE} remains open.\footnote{The recent work \cite{nutz2018convergence} provides a remarkably detailed analysis of nearly the same questions in the context of a specific MFG of optimal stopping, showing  that certain MFE arise as the limits of $n$-player equilibria while others do not; while \cite{nutz2018convergence} notably does not consider approximate equilibria for the $n$-player games, it is still a good source of intuition for what can go wrong.} We give some examples of weak MFE which are not strong but which do arise as limits of approximate $n$-player equilibria by exploiting an interesting connection with the \emph{regularization-by-noise} or \emph{Peano phenomenon}, which can be described as follows: Adding $\epsilon dW_t$ can turn a  non-unique ODE into a well-posed SDE, and in certain cases the limit in distribution of the SDE solution as $\epsilon \downarrow 0$ is a particular mixture of the ODE solutions \cite{bafico1982small}.

On the other hand, in the open-loop regime, it is by now well known that strong MFE do arise as the limit of $n$-player equilibria \cite{huang2007large,carmona2013probabilistic}, and it was shown in \cite{lacker2016general} that the same is true for any weak MFE.
In this sense our results support the folklore that open-loop and closed-loop should ``converge together" as $n\rightarrow\infty$.\footnote{Interestingly, however, this connection can break down if the interactions are not sufficiently continuous; the forthcoming \cite{carmona-cerenzia-palmer} studies an explicitly solvable MFG with singular interactions \`a la Dyson Brownian motion in which the $n\rightarrow\infty$ limits of the open-loop and closed-loop equilibria of the $n$-player games are different.} This is reminiscent of the study in discrete time in \cite{fudenberg-levine}, though the precise use of terminology therein is different.

The paper is organized as follows. We begin in Section \ref{se:mainresults} by specifying notation and assumption, precisely defining the $n$-player and mean field games, and stating clearly the main results. Notably, Section \ref{se:proofideas} sketches the key ideas of the proof of the main result, Theorem \ref{th:mainlimit}. Section \ref{se:mainresults-generalized} makes a first step toward proving the main theorems by relaxing the notions of equilibrium, leading to somewhat more general forms of the main theorems which are interesting in their own right. The heart of the paper is Section \ref{se:mainlimitproof}, devoted to the proof of Theorem \ref{th:mainlimit}. The comparison between open-loop and closed-loop equilibria is developed further in Section \ref{se:closedloopvsopenloop-proofs}, with proofs of several statements from Section \ref{se:closedloopvsopenloop}. Lastly, Section \ref{se:examples} contains some proofs and examples surrounding the partial converse to the main limit theorem, discussed in the previous paragraph.

\section{Setup and main results} \label{se:mainresults}

We begin by fixing some commonly used notation.
We are given a time horizon $T > 0$ and a dimension $d \in \N$, and we write $\C^d$ for the space of continuous paths,
\[
\C^d := C([0,T];\R^d),
\]
equipped with the sup-norm. We use boldface for vectors, such as $\bm{x}=(x_1,\ldots,x_n) \in (\R^d)^n$.

For a complete separable metric space $(E,d)$, let $\P(E)$ denote the set of Borel probability measures.
We always endow $\P(E)$ with the topology of weak convergence and its corresponding Borel $\sigma$-field. Although we will not explicitly use it, to fix ideas we suppose throughout the paper that $\P(E)$ is equipped with the Wasserstein metric
\begin{align}
(m,m') \mapsto \inf_\gamma \int_{E \times E} 1 \wedge d(x,y)\, \gamma(dx,dy), \label{def:wasserstein}
\end{align}
where the infimum is over all $\gamma \in \P(E \times E)$ with marginals $m$ and $m'$. This is known to (completely) metrize weak convergence \cite[Theorem 7.12]{villani2003topics}. In particular, we will make frequent use of the space $\CP$, implicitly equipped with the sup-metric.

For any random variable $X$ we write $\L(X)$ for its law, or $\L(X \, | \, Y)$ for a version of the conditional law of $X$ given another random variable $Y$, which is always well-defined up to almost sure equality when the random variables take values in Polish spaces. We write $X \stackrel{d}{=} Y$ when two random variables have the same law, and we write $X \sim \lambda$ to mean that $\L(X)=\lambda$.

We are given a time horizon $T > 0$, a control space $A$, an initial state distribution $\lambda \in \P(\R^d)$, and the following functions:
\begin{align*}
(b,f) &: [0,T] \times \R^d \times \P(\R^d) \times A \rightarrow \R^d \times \R, \\ 
g &: \R^d \times \P(\R^d) \rightarrow \R.
\end{align*}
The following assumption is in force throughout the paper:

\begin{assumption}{\textbf{A}} \label{assumption:A}
{\ }
\begin{enumerate}
\item[(A.1)] $A$ is a compact convex subset of normed vector space.
\item[(A.2)] The functions $b$, $f$, and $g$ are bounded and jointly continuous. 
\end{enumerate}
\end{assumption}

Occasionally, we will also need the following convexity assumption, which dates back to the work of Filippov \cite{filippov-convexity} and Roxin \cite{roxin1962existence}. It holds, for example, if $b=b(t,x,m,a)$ is affine in $a$ and $f=f(t,x,m,a)$ concave in $a$, for each $(t,x,m)$.

\begin{assumption}{\textbf{B}} \label{assumption:B}
For each $(t,x,m) \in [0,T] \times \R^d \times \P(\R^d)$, the following set is convex:
\[
K(t,x,m) = \left\{(b(t,x,m,a),z) : a \in A, \ z \le f(t,x,m,a)\right\} \subset \R^d \times \R.
\]
\end{assumption}

\subsection{The $n$-player games} \label{se:nplayergame}

Let $n \in \N$. 
In the $n$-player game, an \emph{admissible control} is a progressively measurable function $\alpha : [0,T] \times (\C^d)^n \rightarrow A$.\footnote{Here, we may define progressive measurability simply to mean that $\alpha$ is Borel measurable and satisfies $\alpha(t,\bm{x})=\alpha(t,\bm{x}')$ whenever $t\in [0,T]$ and $\bm{x},\bm{x}' \in (\C^d)^n$ satisfy $\bm{x}_s=\bm{x}'_s$ for all $s \le t$.} Let $\A_n$ denote the set of admissible controls. A \emph{Markovian control} is an admissible control $\alpha \in \A_n$ of the form $\alpha(t,x) = \tilde\alpha(t,x_t)$, where $\tilde\alpha : [0,T] \times (\R^d)^n \rightarrow A$ is Borel measurable. Let $\AM_n \subset \A_n$ denote the set of Markovian controls. Accepting a mild abuse of notation, we will identify $\AM_n$ with the set of Borel measurable functions from $[0,T] \times (\R^d)^n$ to $A$.

The state processes in the $n$-player game are described as follows. For any $\bm{\alpha}=(\alpha^1,\ldots,\alpha^n) \in \A^n_n$, by Girsanov's theorem the following SDE system has a unique in law solution $\bm{X}=(X^1,\ldots,X^n)$:
\begin{align*}
dX^i_t &= b(t,X^i_t,\mu^n_t,\alpha^i(t,\bm{X}))dt + dW^i_t, \quad\quad \mu^n_t = \frac{1}{n}\sum_{k=1}^n\delta_{X^k_t},
\end{align*}
where $W^1,\ldots,W^n$ are independent $d$-dimensonal Brownian motions, and $X^1_0,\ldots,X^n_0$ are i.i.d.\ with law $\lambda$, independent of $(W^1,\ldots,W^n)$.\footnote{We could allow a constant invertible volatility coefficient $\sigma \in \R^{d \times d}$, but by redefining the state variables there is no loss of generality in taking $\sigma$ to be the identity matrix.}
We may write $\bm{X}[\bm{\alpha}] = (X^1[\bm{\alpha}],\ldots,X^n[\bm{\alpha}])$ in place of $\bm{X}=(X^1,\ldots,X^n)$ to stress which controls are being applied, and similarly $\mu^n[\bm{\alpha}]=\mu^n$.

When $(\alpha^1,\ldots,\alpha^n)$ are Markovian, the solution of the above SDE is \emph{strong}, thanks to a result of Veretennikov \cite{veretennikov1981strong} (see also Krylov-R\"ockner \cite[Theorem 2.1]{krylov-rockner}). 
In particular, we can in that case assume the solution processes are all defined on the same probability space. In general, however, we work with weak solutions of SDEs, and keep in mind that for a different control we may need to construct the state process $\bm{X}$ on a different probability space.

Player $i \in \{1,\ldots,n\}$ chooses $\alpha^i$ to try to maximize
\begin{align*}
J^n_i(\alpha^1,\ldots,\alpha^n) := \E\left[\int_0^Tf(t,X^i_t,\mu^n_t,\alpha^i(t,\bm{X}))dt + g(X^i_T,\mu^n_T)\right].
\end{align*}

\begin{definition} \label{def:nplayer-eq}
Let $\epsilon \ge 0$. A \emph{closed-loop (path-dependent) $\epsilon$-Nash equilibrium} is a tuple $(\alpha^1,\ldots,\alpha^n) \in \A_n^n$ such that
\begin{align*}
J^n_i(\alpha^1,\ldots,\alpha^n) \ge \sup_{\beta \in \A_n}J^n_i(\alpha^1,\ldots,\alpha^{i-1},\beta,\alpha^{i+1},\ldots,\alpha^n) - \epsilon, \quad \text{for }  i=1,\ldots,n.
\end{align*}
A \emph{Markovian $\epsilon$-Nash equilibrium} is a tuple $(\alpha^1,\ldots,\alpha^n) \in \AM_n^n$ such that\footnote{
Alternative terminology is common in the engineering literature: Instead of ``closed-loop path-dependent" and ``Markovian" one sometimes encounters ``closed loop perfect state" and ``feedback perfect state," respectively.
}
\begin{align*}
J^n_i(\alpha^1,\ldots,\alpha^n) \ge \sup_{\beta \in \AM_n}J^n_i(\alpha^1,\ldots,\alpha^{i-1},\beta,\alpha^{i+1},\ldots,\alpha^n) - \epsilon, \quad \text{for } i=1,\ldots,n.
\end{align*}
\end{definition}

Note that the notion of Markovian Nash equilibrium involves a supremum only over $\AM_n$, so a priori there is no clear relationship between these two equilibrium concepts. Nonetheless, using Assumption \ref{assumption:B} we prove in Section \ref{se:comparisonofequilibria} that Markovian equilibria form a subset of closed-loop equilibria, which allows us to study both types of equilibrium simultaneously:

\begin{proposition} \label{pr:np-eq-inclusion0}
Suppose Assumptions \ref{assumption:A} and \ref{assumption:B} hold, and let $\epsilon \ge 0$. Then any Markovian $\epsilon$-Nash equilibrium is also a closed-loop $\epsilon$-Nash equilibrium. 
\end{proposition}

It is well known that a Markovian Nash equilibrium for the $n$-player game can be constructed from a classical solution (if one exists) of the corresponding \emph{Nash system}, a system of $n$ parabolic PDEs representing the value functions of the $n$ players. This observation, which goes back to \cite{bensoussan1983nonlinear,bensoussan2013regularity}, is discussed in contexts closer to ours in \cite[Section 2.1.4]{CarmonaDelarue_book_I} and \cite[Section 1.1]{cardaliaguet-delarue-lasry-lions}. Note also that closed-loop and Markovian equilibria can be constructed using a form of the stochastic maximum principle \cite[Section 2.2.2]{CarmonaDelarue_book_I}.

\subsection{The mean field game} \label{se:MFG}

We next define the limiting (mean field) game, beginning with the usual notion of equilibrium, which we call a \emph{strong} equilibrium. In the following, for $m \in \CP$ and measurable functions $\alpha : [0,T] \times \R^d \rightarrow A$ we will encounter SDEs which we will write in the form
\[
d X_t = b(t, X_t,m_t,\alpha(t,X_t))dt + dW_t, \quad X_0 \sim\lambda.
\]
When we say ``$X$ is the unique solution" of this SDE, we mean implicitly that $X=(X_t)_{t \in [0,T]}$ and $W=(W_t)_{t \in [0,T]}$ are continuous stochastic processes defined on some common filtered probability space $(\Omega,\F,\FF,\PP)$ on which $X$ is $\FF$-adapted, $W$ is an $\FF$-Brownian motion, the initial state $X_0$ has law $\PP \circ X_0^{-1} = \lambda$ and is independent of $W$, and the above SDE is satisfied. We avoid making explicit mention of the probability space, as we work exlusively with distributional properties of $X$.
Recall in the following that we write $\L(Y)$ for the law of a random variable $Y$. 

\begin{definition} \label{def:strongMFE-strict}
We say that  $m=(m_t)_{t \in [0,T]} \in C([0,T];\P(\R^d))$ is a \emph{strong mean field equilibrium (MFE)} if there exists a measurable function $\alpha^* : [0,T] \times \R^d \rightarrow A$ such that the unique  solution of the SDE
\[
dX^*_t = b(t,X^*_t,m_t,\alpha^*(t,X^*_t))dt + dW_t, \quad X^*_0 \sim \lambda
\]
satisfies the following:
\begin{enumerate}
\item The consistency condition holds: $m_t=\L(X_t)$ for all $t \in [0,T]$.
\item For any measurable function $\alpha : [0,T] \times \R^d \rightarrow A$, we have
\begin{align*}
\E&\left[\int_0^T f(t,X^*_t,m_t,\alpha^*(t,X^*_t)) dt + g(X^*_T,m_T)\right] \\
	&\ge \E\left[\int_0^T f(t,X_t,m_t,\alpha(t,X_t)) dt + g(X_T,m_T)\right],
\end{align*}
where $  X$ is the unique solution of 
\[
d X_t = b(t, X_t,m_t,\alpha(t,X_t))dt + dW_t, \quad X_0 \sim\lambda.
\]
\end{enumerate}
\end{definition}

It was shown in \cite[Theorem 6.2]{lacker2015mean} that a strong MFE exists under Assumptions \ref{assumption:A} and \ref{assumption:B}, but we will not make use of this fact.
We next define our weak equilibrium concept, after first introducing a useful terminology:

\begin{definition} \label{def:semiMarkovFunction}
For a Polish space $E$, we say a function $F : [0,T] \times \R^d \times \CP \rightarrow E$ is semi-Markov if it is Borel measurable and satisfies $F(t,x,m)=F(t,x,m')$ whenever $(t,x) \in [0,T] \in \R^d$ and $m,m' \in \CP$ satisfy $m_s=m'_s$ for all $s \le t$.
\end{definition}

We use the term \emph{semi-Markov} because the control $\alpha^*(t,x,m)$ depends on the state process only at its current time (Markovian) but on the entire history of the measure flow (non-Markovian). It is important to notice that the dependence on $m$ is \emph{nonanticipative}.

\begin{definition} \label{def:MFEsemimarkov-strict}
A \emph{weak semi-Markov mean field equilibrium} (or simply a \emph{weak MFE}) is a tuple $(\Omega,\F,\FF,\PP,W,\alpha^*,X^*,\mu)$, where $(\Omega,\F,\FF,\PP)$ is a complete filtered probability space and:
\begin{enumerate}
\item $\mu$ is a continuous $\FF$-adapted $\P(\R^d)$-valued process, $W$ is a $\FF$-Brownian motion, and $X^*$ is a continuous $\R^d$-valued $\FF$-adapted process with $\PP \circ (X^*_0)^{-1} = \lambda$.
\item $\alpha^* : [0,T] \times \R^d \times \CP  \rightarrow A$ is semi-Markov.
\item $X^*_0$, $\mu$, and $W$ are independent.
\item The state equation holds:
\begin{align*}
dX^*_t = b(t,X^*_t,\mu_t,\alpha^*(t,X^*_t,\mu))dt + dW_t.
\end{align*}
\item For every alternative semi-Markov $\alpha : [0,T] \times \R^d \times \CP \rightarrow A$  we have
\begin{align*}
\E&\left[\int_0^T f(t,X^*_t,\mu_t,\alpha^*(t,X^*_t,\mu) dt + g(X^*_T,\mu_T)\right] \\
	&\ge \E\left[\int_0^T f(t, X_t,\mu_t,\alpha(t, X_t,\mu)) dt + g( X_T,\mu_T)\right],
\end{align*}
where $ X$ is the solution (see Remark \ref{re:strongsolutions} below) of
\begin{align}
d X_t = b(t, X_t,\mu_t,\alpha(t, X_t,\mu))dt + dW_t, \quad X_0 = X^*_0.  \label{def:SDE-semimarkov-strong}
\end{align}
\item The consistency condition holds: $\mu_t = \PP(X^*_t \in \cdot \, | \, \F^\mu_t)$ a.s.\ for each $t \in [0,T]$, where $\F^\mu_t = \sigma(\mu_s : s \le t)$.
\end{enumerate}
We refer also to the $\P(\R^d)$-valued process $\mu$ itself as a weak (semi-Markov) MFE. In this way, if $\mu$ is deterministic, then it is a weak MFE if and only if  it is a strong MFE. In other words, a strong MFE is always a weak MFE.
\end{definition}

\begin{remark} \label{re:strongsolutions}
The SDE \eqref{def:SDE-semimarkov-strong} admits a unique strong solution (in particular, defined on the same probability space $\Omega$), as we discuss in detail in Appendix \ref{ap:SDErandomcoeff}. 
When $\mu$ is deterministic, this follows immediately from the main results of \cite{veretennikov1981strong,krylov-rockner}. Appendix \ref{ap:SDErandomcoeff} extends this to cover stochastic $\mu$, as long as $X_0$, $W$, and $\mu$ are independent. In particular, the solution $X$ of \eqref{def:SDE-semimarkov-strong} is adapted to the complete filtration generated by the process $(X^*_0,W_s,\mu_s)_{s \le t}$, and so  is $X^*$.
\end{remark}

\subsection{Limit theorems}

The following is the main result of the paper: 

\begin{theorem} \label{th:mainlimit}
Suppose Assumptions \ref{assumption:A} and \ref{assumption:B} hold.
Fix a sequence $\epsilon_n \ge 0$ with $\epsilon_n \rightarrow 0$. For each $n$, suppose $\bm{\alpha}^n=(\alpha^{n,1},\ldots,\alpha^{n,n}) \in \A_n^n$ is a closed-loop $\epsilon_n$-Nash equilibrium. Then the associated empirical measure flow sequence $\mu^n=\mu^n[\bm{\alpha}^n]$ is tight as a family of $\CP$-valued random variables, and every limit in distribution is a weak  MFE.
\end{theorem}

The proof is given between Sections \ref{se:comparisonofequilibria} and \ref{se:mainlimitproof}.
Recalling Proposition \ref{pr:np-eq-inclusion0}, we immediately deduce that Theorem \ref{th:mainlimit} remains true if instead $\bm{\alpha}^n$ is a \emph{Markovian} $\epsilon_n$-Nash equilibrium.
It is well known that a suitable monotonicity condition on the payoff functions ensures that the mean field equilibrium is unique, and we adapt these ideas to our weaker equilibrium concept. The following theorem, inspired by the early uniqueness result of Lasry-Lions \cite{lasrylions}, is proven in Section \ref{se:uniqueness-proof}.

\begin{theorem} \label{th:uniqueness}
Suppose Assumption \ref{assumption:A} holds, along with the following:
\begin{enumerate}[(i)]
\item $b(t,x,m,a)=b(t,x,a)$ has no mean field term.
\item $f(t,x,m,a) = f_1(t,x,m) + f_2(t,x,a)$, for some measurable functions $f_1$ and $f_2$.
\item The action space $A$ is a convex, compact subset of $\R^k$ for some $k$.
\item For each $(t,m) \in [0,T] \times \P(\R^d)$, $b=b(t,x,a)$ is affine in $(x,a)$, $g=g(x,m)$ is concave in $x$, and $f=f(t,x,m,a)$ is strictly concave in $(x,a)$.
\item The monotonicity condition holds: For each $m_1,m_2 \in \P(\R^d)$, we have
\begin{align*}
\int_{\R^d}(f_1(t,x,m_1) - f_1(t,x,m_2))(m_1-m_2)(dx) &\le 0, \\
\int_{\R^d}(g(x,m_1) - g(x,m_2))(m_1-m_2)(dx) &\le 0.
\end{align*}
\end{enumerate}
Then there exists a unique weak MFE, and it is in fact a strong MFE. 
\end{theorem}

Noting that condition (iv) of Theorem \ref{th:uniqueness} implies Assumption \ref{assumption:B}, we may combine the Theorems \ref{th:uniqueness} and \ref{th:mainlimit} to get the following propagation of chaos result:

\begin{corollary} \label{co:uniqueness-propagation}
Suppose  the assumptions of Theorem \ref{th:uniqueness} hold.
For each $n$, suppose $\bm{\alpha}^n=(\alpha^{n,1},\ldots,\alpha^{n,n}) \in \A_n^n$ is a closed-loop $\epsilon_n$-Nash equilibrium. Then $\mu^n=\mu^n[\bm{\alpha}^n]$ converges in probability in $\CP$ to the unique strong MFE.
\end{corollary}

Corollary \ref{co:uniqueness-propagation} is worth comparing to the results of \cite[Section 2.4.4]{cardaliaguet-delarue-lasry-lions}, the only previous limit theorem for closed-loop $n$-player equilibria. Assuming a unique (strong) MFE and a smooth solution of the master equation,
a comparable limit theorem for $\mu^n$ follows from \cite[Theorem 2.15]{cardaliaguet-delarue-lasry-lions}, though it is not stated explicitly. Aside from the fact that they treat common noise, our Corollary \ref{co:uniqueness-propagation} holds under much weaker assumptions. Moreover, our main result, Theorem \ref{th:mainlimit}, holds even when the MFE is non-unique, which seems completely out of reach of the techniques of \cite{cardaliaguet-delarue-lasry-lions}.
Of course, the smooth regime they work with affords a more refined and quantitative description of the limit theorem including convergence of value functions; see also \cite{DelarueLackerRamananCLT,DelarueLackerRamananLDP}.

\begin{remark}
Instead of Definition \ref{def:MFEsemimarkov-strict}, one might propose a more natural \emph{fully-Markov} equilibrium concept, in which the control is of the form $\alpha^*(t,X_t,\mu_t)$, depending only on the present value of the measure flow. It is not clear if this smaller class of equilibria is sufficient to catch all limit points of $n$-player equilibria, and we suspect not.
The issue is likely the mode of convergence, and the method of proof suggests the following conjecture: In the setting of Theorem \ref{th:mainlimit}, every limit point of the pre-compact sequence $(\L(\mu^n_t))_{t \in [0,T]}$ in $C([0,T];\P(\P(\R^d)))$ can be written as $(\L(\mu_t))_{t \in [0,T]}$ for some fully-Markov equilibrium, in the sense just described. To prove this would likely require a Markovian projection argument for measure-valued processes, and such technology does not seem to be available at this time.
\end{remark}

\subsection{A partial converse to the main limit theorem} \label{se:selectionoflimits}

Theorem \ref{th:mainlimit} ensures that all subsequential limits of closed-loop $n$-player approximate equilibria are weak MFE. The natural followup question is: Are all weak MFE subsequential limits of closed-loop $n$-player approximate equilibria? This remains unclear in general, but this section discusses a partial result and a sketch of how to build interesting examples, carried out in more detail in Section \ref{se:examples}. (Notably, if the $n$-player equilibria are open-loop rather than closed-loop, then the results of \cite{lacker2016general} provide an affirmative answer to this question, and we will return to this point in Section \ref{se:closedloopvsopenloop}.)

\begin{assumption}{\textbf{C}} \label{assumption:C}
The drift $b$ is Lipschitz with respect to total variation, in the following sense:
There exists $c > 0$ such that, for each $(t,x,a) \in [0,T] \times \R^d \times A$ and $m,m' \in \P(\R^d)$, we have
\begin{align*}
\frac{1}{c}|b(t,x,m,a) - b(t,x,m',a)| \le  \|m-m'\|_{\mathrm{TV}} := \sup_f \int_{\R^d} f\,d(m-m'),
\end{align*}
where the supremum is over all measurable functions $f : \R^d \rightarrow [-1,1]$.
\end{assumption}

Note that the metric $\|m-m'\|_{\mathrm{TV}}$  dominates the Wasserstein metric defined in \eqref{def:wasserstein}, and thus Assumption \ref{assumption:C} is weaker in a sense than the Wasserstein-Lipschitz assumptions that appear more often in the literature.

We prove the following in Section \ref{se:ex:weakMFEconverse}, which shows that every strong MFE arises as the limit of $n$-player approximate equilibria. 
The only prior result of this nature seems to be the recent \cite[Theorem 6.9]{CarmonaDelarue_book_II}, which operates under different and mostly stronger assumptions. The same conclusion is also implicit in \cite[Proposition 6.3]{cardaliaguet-delarue-lasry-lions}, under even heavier assumptions.

\begin{theorem} \label{th:converselimit-strongMFE}
Suppose Assumptions \ref{assumption:A}, \ref{assumption:B}, and \ref{assumption:C} hold.
Suppose $m \in \CP$ is a strong MFE. Then there exist $\epsilon_n \ge 0$ with $\epsilon_n \rightarrow 0$ and, for each $n$, a Markovian $\epsilon_n$-Nash equilibrium $\bm{\alpha}^n \in \AM_n^n$ such that $\mu^n[\bm{\alpha}^n]$ converges in law to $m$ in $\CP$.
\end{theorem}

The strategy in proving this is standard: Let $\alpha^*(t,x)$ be the corresponding optimal control from Definition \ref{def:strongMFE-strict}. The state process $X$ in Definition \ref{def:strongMFE-strict} is then the solution of
\begin{align}
dX_t = b(t,X_t,m_t,\alpha^*(t,X_t))dt + dW_t, \quad m_t = \L(X_t), \ \forall t \in [0,T]. \label{def:intro:MKVeq1}
\end{align}
Then, we tell each player in the $n$-player game to adopt the control $\alpha^*(t,X^i_t)$. This results in the $n$-particle system
\[
dX^i_t = b(t,X^i_t,\mu^n_t,\alpha^*(t,X^i_t))dt + dW^i_t.
\]
We expect from McKean-Vlasov limit theory that $\mu^n$ converges in law to $m$. The inequality of the optimality condition (2) of Definition \ref{def:strongMFE-strict} should then translate to the approximate Nash property  in the pre-limit. The precise form of Assumption \ref{assumption:C} is inspired from the recent \cite{lacker2018strong}, which proves a strong form of propagation of chaos that allows us to avoid imposing continuity assumptions on the control $\alpha^*$.

It is not clear when we can expect Theorem \ref{th:converselimit-strongMFE} to extend to weak MFE.
To explain what can go wrong, suppose that $(\Omega,\F,\FF,\PP,W,\alpha^*,X^*,\mu)$ is a weak MFE in the sense of Definition \ref{def:MFEsemimarkov-strict}. We then have
\begin{align}
dX^*_t = b(t,X^*_t,\mu_t,\alpha^*(t,X^*_t,\mu))dt + dW_t, \quad \mu_t = \L(X^*_t \, | \, \F^\mu_t), \ \ t \in [0,T]. \label{def:intro:MKVeq2}
\end{align}
Because $X^*_0$, $W$, and $\mu$ are independent, the law of $(X^*_0,W)$ remains unchanged if we condition on $\mu$; it is then intuitively clear (and follows from Lemma \ref{le:ap:SDErandom-uniqueness}) that the $\CP$-valued random variable $\mu$ belongs almost surely to the set $S^*$, consisting of those $m \in \CP$ which solve the McKean-Vlasov equation deterministically,
\begin{align}
dX^m_t = b(t,X^m_t,m_t,\alpha^*(t,X^m_t,m))dt + dW_t, \quad m_t = \L(X^m_t), \ \ t \in [0,T]. \label{def:intro:MKVeq3}
\end{align}
The key point is that if $\mu$ is a weak but not strong MFE, then this McKean-Vlasov equation \eqref{def:intro:MKVeq3} is necessarily non-unique; i.e., $S^*$ is not a singleton. In other words, a weak MFE can always be expressed as a mixture of solutions of a non-unique McKean-Vlasov equation.
As a consequence, we cannot expect propagation of chaos to hold for the corresponding particle system. 
That is, if we proceed as before by letting the players in the $n$-player game use the (path-dependent) controls $\bm{\alpha}^n=(\alpha^{n,1},\ldots,\alpha^{n,n}) \in \A_n^n$ given by
\[
\alpha^{n,i}(t,\bm{x}) = \alpha^*\left(t,x^i_t,\frac{1}{n}\sum_{k=1}^n\delta_{x^k}\right), \quad\quad \bm{x}=(x^1,\ldots,x^n) \in (\C^d)^n,
\]
then there is no way to know if $\mu^n[\bm{\alpha}^n]$ converges to the given $\mu$. For non-unique McKean-Vlasov equations, one can often show that the sequence $\mu^n[\bm{\alpha}^n]$ is tight and that every limit point is supported on $S^*$. But  when $S^*$ is not a singleton, there is no way in general to know which mixture(s) will be ``picked out" by the limit $n\rightarrow\infty$.

We will discuss these ideas further in Section \ref{se:examples}, which includes examples of weak MFE which are not strong MFE but which do arise as the limits of $n$-player (approximate) Nash equilibria. Section  \ref{se:ex:weakMFEconverse}, in particular, gives an example of an interesting kind of weak MFE, discussed also in \cite[Section 3]{lacker2016general}: If $S \subset \CP$ denotes the set of strong MFE, then there can exist weak semi-Markov MFE $\mu$ with $\PP(\mu \in S) < 1$. 
But we do not address an intriguing open problem: Can one construct a weak MFE $\mu$ satisfying $\PP(\mu \in S) < 1$ which arises as the limit of $n$-player approximate equilibria? In the examples we give in Section \ref{se:examples} of weak MFE which arise as the limits of $n$-player approximate equilibria, the weak MFE are always mixtures of strong MFE; that is, they satisfy $\PP(\mu \in S)=1$. Note, on the other hand, that it is known that all weak MFE do indeed arise as limits of \emph{open-loop} $n$-player approximate equilibria; see Theorem \ref{th:openloopconverse} below, essentially quoted from \cite{lacker2016general}.

\subsection{Closed-loop versus open-loop equilibria} \label{se:closedloopvsopenloop}

The parallel limit theory for open-loop $n$-player equilibria is better understood and allows for some interesting comparisons between the two regimes. First, we recall the definition of open-loop equilibrium. In this section, we impose stronger continuity assumptions on $b$ and $f$, so that we may apply the results of \cite{lacker2016general}:

\begin{assumption}{\textbf{D}} \label{assumption:D}
There exist $c > 0$ such that $\int_{\R^d}|x|^2 \, \lambda(dx) < \infty$ and, for each $t \in [0,T]$, $a \in A$,  $x,x' \in \R^d$, and $m,m' \in \P(\R^d)$, we have
\[
|b(t,x,m,a) - b(t,x',m',a)| \le c(|x-x'| + \W_1(m,m')),
\]
where $\W_1$ denotes the Wasserstein metric, defined by $\W_1(m,m') = \inf_\gamma \int_{\R^d \times \R^d}|x-y|\gamma(dx,dy)$, where the infimum is over all probability measures $\gamma$ on $\R^d \times \R^d$ with marginals $m$ and $m'$. Moreover, the objective function $f=f(t,x,m,a)$ satisfies the uniform continuity condition
\begin{align*}
\lim_{(x',m') \rightarrow (x,m)}\sup_{a \in A}|f(t,x',m',a)-f(t,x,m,a)| = 0,
\end{align*}
for all $(t,x,m) \in [0,T] \times \R^d \times \P(\R^d)$.
\end{assumption}

The open-loop $n$-player game is defined on a fixed filtered probability space $(\Omega^n,\F^n,\FF^n,\PP^n)$, supporting independent $\FF^n$-Brownian motions and i.i.d.\ $\F^n_0$-measurable initial states $(X^1_0,\ldots,X^n_0)$ with law $\lambda$.\footnote{The filtration $\FF^n$ does not need to be the minimal one generated by the initial states and Brownian motions.} Let $\AB_n$ denote the set of $\FF^n$-adapted $A$-valued processes. For $\bm{\alpha}=(\alpha^1,\ldots,\alpha^n) \in \AB_n^n$, define the expected payoff
\[
J^n_i(\bm{\alpha}) = \E\left[\int_0^Tf(t,X^i_t,\mu^n_t,\alpha^i_t)dt + g(X^i_T,\mu^n_T)\right],
\]
where  $(X^1,\ldots,X^n)$ is the unique strong solution (recalling Assumption \ref{assumption:D}) of the SDE 
\begin{align*}
dX^i_t = b(t,X^i_t,\mu^n_t,\alpha^i_t)dt + dW^i_t, \quad\quad \mu^n_t = \frac{1}{n}\sum_{k=1}^n\delta_{X^k_t}.
\end{align*}
We may again write $\mu^n=\mu^n[\bm{\alpha}]$ to emphasize the dependence on the choice of control.
For $\epsilon \ge 0$, an \emph{open-loop $\epsilon$-equilibrium} is a tuple $\bm{\alpha}^n=(\alpha^1,\ldots,\alpha^n) \in \AB_n^n$ such that
\begin{align*}
J^n_i(\bm{\alpha}) \ge \sup_{\beta \in \AB_n}J^n_i(\alpha^1,\ldots,\alpha^{i-1},\beta,\alpha^{i+1},\ldots,\alpha^n) - \epsilon.
\end{align*}
It cannot be stressed enough that open-loop and closed-loop equilibria can be very different. See  \cite{carmona-fouque-sun} for an example of an $n$-player game in which the unique (and explicit) open-loop and closed-loop equilibria are distinct, although they converge to the same limit as $n\rightarrow\infty$. Open-loop equilibria are most often found using the stochastic maximum principle \cite[Section 2.2.1]{CarmonaDelarue_book_I}.

We will prove in Section \ref{se:closedloopvsopenloop-proofs} a correspondence between our notion of weak MFE and the equilibrium concept used in \cite{lacker2016general}.
Then,  \cite[Theorems 3.4]{lacker2016general} rewrites as follows:

\begin{theorem} \label{th:openloopconverse}
Suppose Assumptions \ref{assumption:A}, \ref{assumption:B}, and \ref{assumption:D} hold. If, for each $n$, we are given an open-loop $\epsilon_n$-Nash equilibrium  $\bm{\alpha}^n=(\alpha^{n,1},\ldots,\alpha^{n,n}) \in \AB_n^n$ for some $\epsilon_n \ge 0$ with $\epsilon_n \rightarrow 0$, then $\mu^n[\bm{\alpha}^n]$ is tight in $\CP$, and every limit in distribution is a weak MFE.
Conversely, for every weak MFE $\mu$, we may find, for each $n$, $\epsilon_n \ge 0$ and an open-loop $\epsilon_n$-Nash equilibrium $\bm{\alpha}^n=(\alpha^{n,1},\ldots,\alpha^{n,n})$ such that $\epsilon_n \rightarrow 0$ and $\mu^n[\bm{\alpha}^n]$ converges in law to $\mu$ in $\CP$.
\end{theorem}

A proof is given at the end of Section \ref{se:closedloopvsopenloop-proofs}.
Combining Theorems \ref{th:mainlimit} and \ref{th:openloopconverse}, we immediately deduce that closed-loop equilibria can be approximated by open-loop equilibria:

\begin{corollary} \label{co:openloopconverse}
Suppose Assumptions \ref{assumption:A}, \ref{assumption:B}, and \ref{assumption:D} hold.
Let $\epsilon_n \ge 0$ with $\epsilon_n \rightarrow 0$. For each $n$, let $\bm{\alpha}^n=(\alpha^{n,1},\ldots,\alpha^{n,n}) \in \A_n^n$ be a closed-loop $\epsilon_n$-Nash equilibrium. Then there exist $\delta_n \ge 0$ with $\delta_n \rightarrow 0$ and, for each $n$,  an open-loop $\delta_n$-Nash equilibrium $\bm{\beta}^n=(\beta^{n,1},\ldots,\beta^{n,n}) \in \AB_n^n$ such that $\mu^n[\bm{\alpha}^n]$ and $\mu^n[\bm{\beta}^n]$ converge together in law  in $\CP$. Precisely, for each $\varphi \in C_b(\CP)$ we have
\[
\lim_{n \rightarrow\infty} \E[\varphi(\mu^n[\bm{\alpha}^n])] - \E[\varphi(\mu^n[\bm{\beta}^n])] = 0.
\]
\end{corollary}

\subsection{Ideas of the proof of the main limit theorem} \label{se:proofideas}

In this section we informally explain some of the main ideas of the rather lengthy proof of Theorem \ref{th:mainlimit}, which comes in Section \ref{se:mainlimitproof}. Tightness is straightforward here and fairly standard, so we mostly focus on the two bigger challenges of identifying the dynamics at the limit (properties (1-4) and (6) of Definition \ref{def:MFEsemimarkov-strict}) and proving the optimality of the limiting control (property (5) of Definition \ref{def:MFEsemimarkov-strict}).

A key tool in identifying the limiting dynamics is (a special case of) the Markovian projection theorem, due originally to Gy\"ongy \cite[Theorem 4.6]{gyongy1986mimicking} and later generalized in \cite[Corollary 3.7]{brunick2013mimicking}:

\begin{theorem}[Markovian projection] \label{th:markovprojection}
Let $(\Omega,\F,\FF,\PP)$ be a filtered probability space supporting an $\FF$-adapted continuous process $X$ and an $\FF$-Brownian motion $W$. Suppose $b=(b_t)_{t \in [0,T]}$ is a bounded $\FF$-progressively measurable process such that, almost surely,
\[
X_t = X_0 + \int_0^tb_s\,ds + W_t, \quad t \in [0,T].
\]
Then there exists a bounded measurable function $\widehat{b} : [0,T] \times \R^d \rightarrow \R^d$ such that
\begin{align*}
\widehat{b}(t,X_t) = \E[b_t \, | \, X_t], \ \ a.s., \ t \in [0,T],
\end{align*}
and, moreover, the unique strong solution of the SDE
\[
dY_t = \widehat{b}(t,Y_t)dt + dW_t, \quad Y_0=X_0,
\]
satisfies $Y_t \stackrel{d}{=} X_t$ for each $t \in [0,T]$.
\end{theorem}

\subsubsection{Limiting dynamics} \label{se:proofsketch-limitingdynamics}

We want to show that, for any weak limit $\mu$ of $(\mu^n)$, we may construct a tuple $(\Omega,\F,\FF,\PP,W,\alpha^*,X^*,\mu)$ and such that properties (1-4) and (6) of Definition \ref{def:MFEsemimarkov-strict} hold. Much of this argument is an embellishment of a well-established martingale approach for deriving the McKean-Vlasov limit for interacting diffusions, developed for instance in \cite{oelschlager1984martingale,gartner1988mckean}. A first difference is that here we work with the extended empirical measure
\[
\bm{\overline\mu}^n = \frac{1}{n}\sum_{k=1}^n\delta_{(X^{k},\alpha^{n,k})}.
\]
Here we view $X^{k}$ as a $\C^d$-valued random variable and $\alpha^{n,k}=\alpha^{n,k}(t,\bm{X})$ as a random variable taking values in the space $\V$ of \emph{relaxed} or \emph{measure-valued} controls, defined in Section \ref{se:relaxedcontrols}; the space $\V$ is essentially a convenient compactification of the space $L^0([0,T];A)$ of measurable $A$-valued paths. First, we show that every weak limit $\bm{\overline\mu}$ of $(\bm{\overline\mu}^n)$ satisfies
\[
\int_{\R^d}\varphi\,d(\mu_t-\mu_0) = \int_{\C^d\times\V} \int_0^t\int_A \left(b(s,x_s,\mu_s,a) \cdot \nabla\varphi(x_s) + \frac12\Delta\varphi(x_s)\right) q_s(da)ds  \bm{\overline\mu}(dx,dq)
\]
almost surely, for each smooth test function $\varphi$ on $\R^d$, where $\mu_t = \bm{\overline\mu} \circ [(x,q) \mapsto x_t]^{-1}$ is the marginal flow associated to the $x$ variable.

The above integral equation closely resembles the weak or integrated form of a Fokker-Planck equation. Instead of an integral $\int_0^t\int_{\R^d} \, ... \, \mu_s(dx)ds$ appearing on the right-hand side, we have a more complicated expression involving the integral with the respect to $\bm{\overline\mu}$. Drawing intuition from the Markovian Projection Theorem \ref{th:markovprojection}, we would like to condition on the marginal flow $(\mu_t)_{t \in [0,T]}$, in order to ``project away the extra randomness" in some sense. Ultimately, we build (cf. Lemma \ref{le:projection-semiMarkov}) a semi-Markov control $\alpha^* : [0,T] \times \R^d \times \CP \rightarrow A$ such that
\[
\int_{\R^d}\varphi\,d(\mu_t-\mu_0) = \int_0^t \int_{\R^d} \left(b(s,x,\mu_s,\alpha^*(t,x,\mu)) \cdot \nabla\varphi(x) + \frac12\Delta\varphi(x)\right) \mu_s(dx) ds,
\]
almost surely, for each $\varphi$, and such that the expected value of objective function is preserved in a suitable sense. This now says that $(\mu_t)_{t\in [0,T]}$ almost surely solves a Fokker-Planck equation, which we can identify with the solution of an SDE. In fact, this SDE is of McKean-Vlasov type, because $\mu$ itself appears nonlinearly in the coefficients $b$ and $\alpha^*$, and this line of reasoning eventually leads us to properties (1-4) and (6) of Definition \ref{def:MFEsemimarkov-strict}.

\subsubsection{Optimality at the limit} \label{se:proofsketch-optimality} 

Suppose now that we have proven the claimed tightness of Theorem \ref{th:mainlimit} and also that for any limit point $\mu$ of $(\mu^n)$ we may construct a tuple $(\Omega,\F,\FF,\PP,W,\alpha^*,X^*,\mu)$ such that properties (1-4) and (6) of Definition \ref{def:MFEsemimarkov-strict} hold. The final and most difficult step is to check that this tuple satisfies the optimality property (5) of Definition \ref{def:MFEsemimarkov-strict}. In the following, we work with a relabeled convergent subsequence and assume $\mu^n$ converges in law to $\mu$.

The general strategy, reminiscent of Gamma-convergence arguments, is to choose an arbitrary alternative control $\alpha : [0,T] \times \R^d \times \CP \rightarrow A$, and give it to each of the players in the $n$-player game.
Precisely, for each $n$ and each $k=1,\ldots,n$, define the $n$ state processes $\bm{Y}^{k} = (Y^{k,1},\ldots,Y^{k,n})$ by
\begin{align*}
dY^{k,k}_t &= b(t,Y^{k,k}_t,\mu^{n,k}_t,\alpha(t,Y^{k,k}_t,\mu^{n,k}))dt + dW^k_t, \quad\quad \mu^{n,k} = \frac{1}{n}\sum_{j=1}^n\delta_{Y^{n,k,j}} \\
dY^{k,i}_t &= b(t,Y^{k,i}_t,\mu^{n,k}_t,\alpha^{n,i}(t,\bm{Y}^{k}))dt + dW^i_t, \quad i \neq k, 
\end{align*}
with initial states $Y^{k,i}_0=X^{i}_0$. The state process $\bm{Y}^{k}$ differes from the equilibrium state process $\bm{X}[(\alpha^{n,1},\ldots,\alpha^{n,n})]$ only in that we switched player $k$'s control from $\alpha^{n,k}$ to $\alpha$.

The assumed $\epsilon_n$-Nash equilibrium property of $(\alpha^{n,1},\ldots,\alpha^{n,n})$ then implies that
\begin{align}
\frac{1}{n}&\sum_{k=1}^n\E\left[\int_0^Tf(t,X^{k}_t,\mu^n_t,\alpha^{n,k}(t,\bm{X}^n))dt + g(X^{k}_T,\mu^n_T)\right] \nonumber \\
	&\ge -\epsilon_n + \frac{1}{n}\sum_{k=1}^n\E\left[\int_0^Tf(t,Y^{k,k}_t,\mu^{n,k}_t,\alpha(t,Y^{k,k}_t,\mu^{n,k}))dt + g(Y^{k,k}_T,\mu^{n,k}_T)\right]. \label{pf:sketch-limitoptimality1}
\end{align}
We then wish to take limits on both sides. First, the arguments of Section \ref{se:proofsketch-limitingdynamics} allow us to identify the limit of the left-hand side of \eqref{pf:sketch-limitoptimality1} as precisely the left-hand side of the inequality in (5) of Definition \ref{def:MFEsemimarkov-strict}.
What remains is to show that the right-hand side of \eqref{pf:sketch-limitoptimality1} along the same subsequence converges to the right-hand side of the inequality in (5) of Definition \ref{def:MFEsemimarkov-strict}.

This last point is the technical crux of the argument. It is not obvious at first how to approach this, because we know very little about the controls $\alpha^{n,1},\ldots,\alpha^{n,n}$. Intuitively, one is tempted claim that, because we have only switched one single agent's control, $\mu^{n,k}$ should be  close in some sense to $\mu^n$, for each $k$. The challenge comes from the closed-loop nature of the controls; if one player switches controls, then all of the other players controls react to the change in the state process. It could be the case that all of the controls  $\alpha^{n,1},\ldots,\alpha^{n,n}$ depend very heavily on, say, player $1$'s state process, in which case a change in control from this player $1$ would have a strong influence on the empirical measure.

While we cannot show that $\mu^{n,k}$ and $\mu^n$ have the same limiting behavior for each $k$, we are able to show that $\L(\mu^n)$ and $\frac{1}{n}\sum_{k=1}^n\L(\mu^{n,k})$ have the same limiting behavior, in the sense that the total variation distance between these two measures converges to zero as $n\rightarrow\infty$.
Indeed, supposing the state process $\bm{X}$ is defined on the probability space $(\Omega^n,\F^n,\FF^n,\PP^n)$, we may define an equivalent probability measure $\QQ^{n,k}$ by setting $d\QQ^{n,k}/d\PP^n := \zeta^{n,k}_T$, where the positive martingale $(\zeta^{n,k}_t)_{t \in [0,T]}$ is given as the unique solution of the SDE
\[
d\zeta^{n,k}_t = \zeta^{n,k}_t \Big(b(t,X^{k}_t,\mu^n_t,\alpha(t,X^{k}_t,\mu^n)) - b(t,X^{k}_t,\mu^n_t,\alpha^{n,k}(t,\bm{X})) \Big) \cdot dW^k_t, \quad \zeta^{n,k}_0=1.
\]
By Girsanov's theorem and uniqueness of the SDEs, we have $\L(\bm{Y}^{k})=\QQ^{n,k} \circ \bm{X}^{-1}$. Hence, for any bounded measurable function $h$,
\begin{align}
\frac{1}{n}\sum_{k=1}^n\E[h(\mu^{n,k})] &= \frac{1}{n}\sum_{k=1}^n\E[\zeta^{n,k}_Th(\mu^n)]. \label{pf:sketch-limitoptimality2}
\end{align}
Because the Brownian motions $W^k$ are independent, the process $\frac{1}{n}\sum_{k=1}^n\zeta^{n,k}_t$ is a martingale with quadratic variation up to time $s$ given by
\begin{align*}
\frac{1}{n^2}\sum_{k=1}^n\int_0^s\Big|b(t,X^{k}_t,\mu^n_t,\alpha(t,X^{k}_t,\mu^n)) - b(t,X^{n,k}_t,\mu^n_t,\alpha^{n,k}(t,\bm{X})) \Big|^2dt,
\end{align*}
which is of order $1/n$ because $b$ is bounded.
Hence, $\frac{1}{n}\sum_{k=1}^n\zeta^{n,k}_T \rightarrow 1$ in probability, and from \eqref{pf:sketch-limitoptimality2} we deduce that
\begin{align}
\lim_{n\rightarrow\infty} \frac{1}{n}\sum_{k=1}^n\E[h(\mu^{n,k})] - \E[h(\mu^{n})] = 0. \label{pf:sketch-limitoptimality3}
\end{align}

Most of the intuition behind this proof is contained in this argument that $\L(\mu^n)$ and $\frac{1}{n}\sum_{k=1}^n\L(\mu^{n,k})$ have the same limiting behavior, but one important additional point is worth mentioning: The right-hand side of \eqref{pf:sketch-limitoptimality1} can be written as the integral of a fixed ($n$-independent) function with respect to the measure $\frac{1}{n}\sum_{k=1}^n\L(Y^{k,k},\mu^{n,k})$, and it is this measure whose limiting behavior we should identify, not just $\frac{1}{n}\sum_{k=1}^n\L(\mu^{n,k})$. To this end, for any bounded measurable function $h$, write
\begin{align*}
\frac{1}{n}\sum_{k=1}^n\E[h(Y^{k,k},\mu^{n,k})] &= \frac{1}{n}\sum_{k=1}^n\E[\zeta^{n,k}_Th(X^{k},\mu^n)].
\end{align*}
The limiting behavior of this expression can be identified by studying the $(d+1)$-dimensional particle system $(X^{k},\zeta^{n,k})_{k=1}^n$, following the classical martingale approach for McKean-Vlasov systems mentioned in Section \ref{se:proofsketch-limitingdynamics}.

\section{Relaxed equilibria} \label{se:mainresults-generalized}

Our proofs will make heavy use of \emph{relaxed} or \emph{randomized controls}, essentially replacing $A$-valued controls with $\P(A)$-valued controls, which by now have a long history in stochastic optimal control theory \cite{fleming1976generalized,elkaroui-compactification} for their useful compactness properties. Relaxed controls were employed in an MFG context \cite{lacker2015mean,carmona-delarue-lacker,lacker2016general}, and we will use them in the same way. It is worth noting, however, that while they are certainly mathematically convenient, relaxed controls also admit a natural interpretation in a game-theoretic context as \emph{mixed strategies}.

\subsection{Relaxed $n$-player games} \label{se:nplayergame-relaxed}

We begin by extending the equilibrium concepts for $n$-player games of Section \ref{se:nplayergame}.
Write $\RC_n$ for the set of progressively measurable functions $\Lambda : [0,T] \times (\C^d)^n \rightarrow \P(A)$, and let $\RCM_n$ denote the subset of functions of the form $\Lambda(t,x) = \tilde\Lambda(t,x_t)$ for some measurable function $\tilde\Lambda : [0,T] \times (\R^d)^n \rightarrow \P(A)$. Via the embedding $A \ni a \mapsto \delta_a \in \P(A)$, we may view $\A_n$ and $\AM_n$ as subsets of $\RC_n$, and we have the following natural inclusions:
\[
\AM_n \subset \A_n \subset \RC_n, \quad\quad \AM_n \subset \RCM_n \subset \RC_n.
\]
The state process and objective functions are defined for relaxed controls $\bm{\Lambda}=(\Lambda^1,\ldots,\Lambda^n) \in \RC_n^n$ as follows:
\begin{align*}
dX^i_t &= \int_{A}b(t,X^i_t,\mu^n_t,a)\Lambda^i(t,\bm{X})(da)dt + dW^i_t, \\
J^n_i(\Lambda^1,\ldots,\Lambda^n) &= \E\left[\int_0^T\int_Af(t,X^i_t,\mu^n_t,a)\Lambda^i(t,\bm{X})(da)dt + g(X^i_T,\mu^n_T)\right].
\end{align*}
We may write $\bm{X}[\bm{\Lambda}] = (X^1[\bm{\Lambda}],\ldots,X^n[\bm{\Lambda}])$ in place of $\bm{X}=(X^1,\ldots,X^n)$ to stress which controls are being applied, and similarly we may write $\mu^n[\bm{\Lambda}]$ in place of $\mu^n$.

\begin{definition} \label{def:nplayer-eq-relaxed}
Let $\epsilon \ge 0$. A \emph{relaxed closed-loop $\epsilon$-Nash equilibrium} is a tuple $(\Lambda^1,\ldots,\Lambda^n) \in \RC_n^n$ such that
\begin{align*}
J^n_i(\Lambda^1,\ldots,\Lambda^n) \ge \sup_{\beta \in \RC_n}J^n_i(\Lambda^1,\ldots,\Lambda^{i-1},\beta,\Lambda^{i+1},\ldots,\Lambda^n) - \epsilon, \quad \text{for } i=1,\ldots,n.
\end{align*}
A \emph{relaxed Markovian $\epsilon$-Nash equilibrium} is a tuple $(\Lambda^1,\ldots,\Lambda^n) \in \RCM_n^n$ such that 
\begin{align*}
J^n_i(\Lambda^1,\ldots,\Lambda^n) \ge \sup_{\beta \in \RCM_n}J^n_i(\Lambda^1,\ldots,\Lambda^{i-1},\beta,\Lambda^{i+1},\ldots,\Lambda^n) - \epsilon, \quad \text{for } i=1,\ldots,n.
\end{align*}
\end{definition}

The following trio of propositions, along with Proposition \ref{pr:np-eq-inclusion0}, will show that the four equilibrium concepts described in Definitions \ref{def:nplayer-eq} and \ref{def:nplayer-eq-relaxed} are roughly equivalent, if we accept both assumptions \ref{assumption:A} and \ref{assumption:B}. The proofs are given in Section \ref{se:comparisonofequilibria}.

\begin{proposition} \label{pr:np-eq-inclusion1}
Suppose Assumption \ref{assumption:A} holds, and let $\epsilon \ge 0$. Then any relaxed Markovian $\epsilon$-Nash equilibrium is also a relaxed closed-loop $\epsilon$-Nash equilibrium.  
\end{proposition}

\begin{proposition} \label{pr:np-eq-inclusion2}
Suppose Assumptions \ref{assumption:A} and \ref{assumption:B} hold, and let $\epsilon \ge 0$. Then:
\begin{enumerate}[(a)]
\item Any Markovian $\epsilon$-Nash equilibrium is also a relaxed Markovian $\epsilon$-Nash equilibrium.
\item Any closed-loop $\epsilon$-Nash equilibrium is also a relaxed closed-loop $\epsilon$-Nash equilibrium.
\end{enumerate}
\end{proposition}

\begin{proposition} \label{pr:np-eq-inclusion3}
Suppose Assumptions \ref{assumption:A} and \ref{assumption:B} hold, and let $\epsilon \ge 0$. Then:
\begin{enumerate}[(a)]
\item For any relaxed Markovian $\epsilon$-Nash equilibrium $\bm{\Lambda}=(\Lambda^1,\ldots,\Lambda^n) \in \RCM_n^n$, there exists a Markovian $\epsilon$-Nash equilibrium $\bm{\alpha}=(\alpha^1,\ldots,\alpha^n) \in \AM_n^n$ such that $\bm{X}[\bm{\Lambda}] \stackrel{d}{=} \bm{X}[\bm{\alpha}]$.
\item For any relaxed closed-loop $\epsilon$-Nash equilibrium $\bm{\Lambda}=(\Lambda^1,\ldots,\Lambda^n) \in \RC_n^n$, there exists a closed-loop $\epsilon$-Nash equilibrium $\bm{\alpha}=(\alpha^1,\ldots,\alpha^n) \in \A_n^n$ such that $\bm{X}[\bm{\Lambda}] \stackrel{d}{=} \bm{X}[\bm{\alpha}]$.
\end{enumerate}
\end{proposition}

Some notation helps to summarize the above propositions. Fix $\epsilon \ge 0$, let $\A^{*,\epsilon}_n \subset \A_n^n$ denote the set of closed-loop $\epsilon$-Nash equilibria. Similarly, define $\AM^{*,\epsilon}_n$, $\RC^{*,\epsilon}_n$, and $\RCM^{*,\epsilon}_n$ respectively as the sets of Markovian, relaxed closed-loop, and relaxed Markovian $\epsilon$-Nash equilibria. We may summarize the relations of Propositions \ref{pr:np-eq-inclusion0}, \ref{pr:np-eq-inclusion1}, and \ref{pr:np-eq-inclusion2} by writing
\[
\AM^{*,\epsilon}_n \subset \A^{*,\epsilon}_n \subset \RC^{*,\epsilon}_n, \quad\quad \AM^{*,\epsilon}_n \subset \RCM^{*,\epsilon}_n  \subset \RC^{*,\epsilon}_n.
\]
Moreover, we can think of Proposition \ref{pr:np-eq-inclusion3}(a) (resp. (b)) as reducing $\AM^{*,\epsilon}_n \subset \RCM^{*,\epsilon}_n$ (resp. $\A_n^{*,\epsilon} \subset \RC_n^{*,\epsilon}$) to equality, if we are content to focus only on the law of the state process $\bm{X}$.
Precisely, under Assumptions \ref{assumption:A} and \ref{assumption:B}, we have the following relationships between subsets of $\P((\C^d)^n)$:
\begin{align}
\begin{split}
\{\L(\bm{X}[\bm{\Lambda}]) : \bm{\Lambda} \in \AM^{*,\epsilon}_n\} &= \{\L(\bm{X}[\bm{\Lambda}]) : \bm{\Lambda} \in \RCM^{*,\epsilon}_n\} \\
	&\subset \{\L(\bm{X}[\bm{\Lambda}]) : \bm{\Lambda} \in \RC^{*,\epsilon}_n\} \\
	&= \{\L(\bm{X}[\bm{\Lambda}]) : \bm{\Lambda} \in \A^{*,\epsilon}_n\}.
\end{split} \label{eq:lawrelations-differentcontrols}
\end{align}
Recall that our main result, Theorem \ref{th:mainlimit}, involves only the law of the state process $\bm{X}$. Thanks to the above propositions, we may simultaneously cover all four of these possibilities by focusing solely on the laws of closed-loop Markovian equilibria, i.e., $\A_n^{*,\epsilon}$.

We will make no claims throughout the paper regarding existence of equilibria for $n$-player games, but we provide some references.  As we have mentioned, Markovian Nash equilibria (the set $\AM_n^{*,0}$, in the notation of the previous paragraph) are by the most commonly studied in the literature can be found by solving a system of $n$ Hamilton-Jacobi-Bellman (HJB) equations. 
Relaxed Markovian equilibria are far less common, but the notable paper of Borkar and Ghosh \cite{borkar-ghosh} has several theorems on existence (i.e., $\RCM_n^{*,0} \neq \emptyset$). While their discussion of finite horizon problems is limited to the final sentence of the paper, it is clear that the techniques they develop for infinite-horizon problems can be easily adapted.
Closed-loop path-dependent equilibria have appeared with some frequency in the literature on two-player stochastic differential games \cite{cardaliaguet2009stochastic,hamadene1995zero}. They arise quite naturally in the BSDE-based weak formulation of Hamadene-Lepeltier \cite{hamadene1995zero}, which reduces the existence of Nash equilibria to the solution of a BSDE (which is nothing but the stochastic representation of the corresponding HJB equation). The extension to the $n$-player setting is written in the lecture notes \cite[Section 5.3.2]{carmona2016lectures}, but be careful that our notion of closed-loop equilibrium is called ``open-loop" therein. Lastly, we are unaware of any discussion of \emph{relaxed closed-loop} equilibria in prior literature, but it is useful at the very least as an intermediary in establishing the relations in \eqref{eq:lawrelations-differentcontrols}.

\subsection{Relaxed mean field equilibria} \label{se:MFG-relaxed}

We next extend the MFG equilibrium concepts (Definitions \ref{def:strongMFE-strict} and \ref{def:MFEsemimarkov-strict}) of Section \ref{se:MFG} to the relaxed setting:

\begin{definition} \label{def:strongMFE}
We say that  $m=(m_t)_{t \in [0,T]} \in C([0,T];\P(\R^d))$ is a \emph{strong relaxed mean field equilibrium} (or simply a \emph{strong RMFE}) if there exists a measurable function $\Lambda^* : [0,T] \times \R^d \rightarrow \P(A)$ such that the unique solution of the SDE
\[
dX^*_t = \int_Ab(t,X^*_t,m_t,a)\Lambda^*(t,X^*_t)(da)dt + dW_t, \quad X^*_0 \sim \lambda
\]
satisfies the following:
\begin{enumerate}
\item The consistency condition holds: $m_t=\L(X_t)$ for all $t \in [0,T]$.
\item For any measurable function $\Lambda : [0,T] \times \R^d \rightarrow \P(A)$, we have
\begin{align*}
\E&\left[\int_0^T\int_A f(t,X^*_t,m_t,a)\Lambda^*(t,X^*_t)(da) dt + g(X^*_T,m_T)\right] \\
	&\ge \E\left[\int_0^T\int_A f(t,  X_t,m_t,a) \Lambda(t,  X_t)(da) dt + g(  X_T,m_T)\right],
\end{align*}
where $X$ is the unique solution of 
\[
d X_t = \int_A b(t,  X_t,m_t,a) \Lambda(t,X_t)(da)dt + dW_t, \quad X_0 \sim\lambda.
\]
\end{enumerate}
\end{definition}

It was shown in \cite[Theorem 6.2]{lacker2015mean} that a strong MFE exists under Assumption \ref{assumption:A}, though we will not need this fact. Recall from Definition \ref{def:semiMarkovFunction} the notion of a \emph{semi-Markov function}.

\begin{definition} \label{def:MFEsemimarkov}
A \emph{weak semi-Markov relaxed mean field equilibrium} (or simply a \emph{weak RMFE}) is a tuple $(\Omega,\F,\FF,\PP,W,\Lambda^*,X^*,\mu)$, where $(\Omega,\F,\FF,\PP)$ is a complete filtered probability space and:
\begin{enumerate}
\item $\mu$ is a continuous $\FF$-adapted $\P(\R^d)$-valued process, $W$ is a $\FF$-Brownian motion, and $X^*$ is a continuous $\R^d$-valued $\FF$-adapted process with $\PP \circ (X_0^*)^{-1} = \lambda$.
\item $\Lambda^* : [0,T] \times \R^d \times \CP  \rightarrow \P(A)$ is semi-Markov.
\item $X^*_0$, $\mu$, and $W$ are independent.
\item The state equation holds:
\begin{align}
dX^*_t = \int_A b(t,X^*_t,\mu_t,a)\Lambda^*(t,X^*_t,\mu)(da)dt + dW_t. \label{def:SDE-MFEsemimarkov}
\end{align}
\item For every alternative $\Lambda : [0,T] \times \R^d \times \CP \rightarrow \P(A)$ satisfying (2), we have
\begin{align*}
\E&\left[\int_0^T\int_A f(t,X^*_t,\mu_t,a)\Lambda^*(t,X^*_t,\mu)(da) dt + g(X^*_T,\mu_T)\right] \\
	&\ge \E\left[\int_0^T\int_A f(t, X_t,\mu_t,a)\Lambda(t, X_t,\mu)(da) dt + g( X_T,\mu_T)\right],
\end{align*}
where $ X$ is the solution (recall Remark \ref{re:strongsolutions}) of
\[
d X_t = \int_A b(t, X_t,\mu_t,a) \Lambda(t, X_t,\mu)(da)dt + dW_t, \quad X_0 \sim \lambda.
\]
\item The consistency condition holds: $\mu_t = \PP(X^*_t \in \cdot \, | \, \F^\mu_t)$ a.s.\ for each $t \in [0,T]$, where $\F^\mu_t = \sigma(\mu_s : s \le t)$.
\end{enumerate}
We refer also to the $\P(\R^d)$-valued process $\mu$ itself as a weak RMFE.
\end{definition}

Similar to the relationships of Section \ref{se:nplayergame-relaxed}, under Assumptions \ref{assumption:A} and \ref{assumption:B} we prove in Section \ref{se:comparisonofequilibria} that MFE and relaxed MFE induce the same measure flows:

\begin{proposition} \label{pr:RMFE-to-MFE}
Suppose Assumptions \ref{assumption:A} and \ref{assumption:B} hold. Then every strong RMFE is a strong MFE, and every strong MFE is a strong RFME. Similarly, on the level of the measure flow $\mu$, every weak RMFE is a weak MFE, and every weak MFE is a RMFE.
\end{proposition}

\begin{remark} \label{re:Xcompatiblewithmu}
Recall from Remark \ref{re:strongsolutions} that the SDEs in \eqref{def:SDE-MFEsemimarkov} admit unique strong solutions, and in particular $X^*$ is necessarily adapted to the complete filtration generated by the process $(X^*_0,W_s,\mu_s)_{s \le t}$. With this and property (3) of Definition \ref{def:MFEsemimarkov}, we easily deduce that $\L(X^*_t \, | \, \F^\mu_t) = \L(X^*_t \, | \, \mu)$ a.s., for each $t \in [0,T]$, which will be useful later.
\end{remark}

\subsection{Extensions of the limit theorems} \label{se:extensions-relaxed}

This section collects some generalizations of the main results announced in Section \ref{se:mainresults}, which do not require the convexity Assumption \ref{assumption:B}. The results of the previous two subsections show how the various equilibrium concepts related to each other if we impose Assumption \ref{assumption:B}, and this is how we will deduce the results of Section \ref{se:mainresults} from those announced here.

\begin{theorem} \label{th:mainlimit-relaxed}
Suppose Assumption \ref{assumption:A} holds.
Fix a sequence $\epsilon_n \ge 0$ with $\epsilon_n \rightarrow 0$. For each $n$, suppose $\bm{\alpha}^n=(\alpha^{n,1},\ldots,\alpha^{n,n}) \in \A_n^n$ is a closed-loop $\epsilon_n$-Nash equilibrium. Then the associated empirical measure flow sequence $\mu^n=\mu^n[\bm{\alpha}^n]$ is tight as a family of $\CP$-valued random variables, and every limit in distribution is a weak RMFE.
\end{theorem}

If we impose both Assumption \ref{assumption:A} and \ref{assumption:B}, then Proposition \ref{pr:RMFE-to-MFE} tells us that weak RMFE and weak MFE are one and the same. Thus, our main result, Theorem \ref{th:mainlimit}, follows from Theorem \ref{th:mainlimit-relaxed}.
Recall also the relations summarized in \eqref{eq:lawrelations-differentcontrols}. 
Under Assumptions \ref{assumption:A} and \ref{assumption:B}, we deduce that Theorem \ref{th:mainlimit-relaxed} remains valid when $\bm{\alpha}^n$ is instead assumed to be any of the four types of equilibrium described in Definitions \ref{def:nplayer-eq} and \ref{def:nplayer-eq-relaxed}.

Similarly, we may deduce the converse Theorem \ref{th:converselimit-strongMFE} from Proposition \ref{pr:np-eq-inclusion3}(a) the following generalization to relaxed equilibria:

\begin{theorem} \label{th:converselimit-strongMFE-relaxed}
Suppose Assumptions \ref{assumption:A} and \ref{assumption:C} hold.
Suppose $m \in \CP$ is a strong RMFE. Then there exist $\epsilon_n \ge 0$ with $\epsilon_n \rightarrow 0$ and, for each $n$, a relaxed Markovian $\epsilon_n$-Nash equilibrium $\bm{\Lambda}^n \in \RCM_n^n$ such that $\mu^n[\bm{\Lambda}^n]$ converges in law to $m$ in $\CP$.
\end{theorem}

Sections \ref{se:mainlimitproof} and \ref{se:converse-strongMFE-proof} are devoted to the proofs of Theorems \ref{th:mainlimit-relaxed} and \ref{th:converselimit-strongMFE-relaxed}, respectively.
We might lastly state a form of Theorem \ref{th:openloopconverse} without Assumption \ref{assumption:B}, as long as we use weak RMFE instead of weak MFE, but we opt not to write this out explicitly.

\section{Relating the various equilibrium concepts} \label{se:comparisonofequilibria}

This section proves the various relationships between different equilibrium concepts of Definitions \ref{def:nplayer-eq} and \ref{def:nplayer-eq-relaxed}, announced in Propositions \ref{pr:np-eq-inclusion0}, \ref{pr:np-eq-inclusion1}, \ref{pr:np-eq-inclusion2}, and \ref{pr:np-eq-inclusion3}. We also prove Proposition \ref{pr:RMFE-to-MFE}, which relates MFE to RMFE.

\subsection{Proof of Proposition \ref{pr:np-eq-inclusion1}}

Let $\epsilon \ge 0$, and fix a relaxed Markovian $\epsilon$-Nash equilibrium $\bm{\Lambda}=(\Lambda^1,\ldots,\Lambda^n) \in \RCM_n^n$. The goal is to show that $\bm{\Lambda}$ is also a relaxed closed-loop $\epsilon$-Nash equilibrium. The state processes $\bm{X}=\bm{X}[\bm{\Lambda}]$ solve the SDE system
\begin{align*}
dX^i_t &= \int_{A}b(t,X^i_t,\mu^n_t,a)\Lambda^i(t,\bm{X}_t)(da)dt + dW^i_t, \quad \mu^n_t = \frac{1}{n}\sum_{k=1}^n\delta_{X^k_t}.
\end{align*}
Let $\beta \in \RC_n$ be an alternative relaxed closed-loop control. We will focus on player $1$, showing that
\begin{align}
J^n_1(\Lambda^1,\ldots,\Lambda^n) \ge J^n_1(\beta,\Lambda^2,\ldots,\Lambda^n) - \epsilon. \label{pf:np-inclusion1-1}
\end{align}
The argument for other players $i \neq 1$ is identical. To proceed, let $\bm{Y}=(Y^1,\ldots,Y^n) := \bm{X}[\beta,\Lambda^2,\ldots,\Lambda^n]$ be the state processes in which players $i\neq 2$ still use $\Lambda^i$:
\begin{align*}
dY^1_t &= \int_A b(t,Y^1_t,\nu^n_t,a)\beta(t,\bm{Y})(da)dt + dW^1_t, \\
dY^i_t &= \int_A b(t,Y^i_t,\nu^n_t,a)\Lambda^i(t,\bm{Y}_t)(da)dt + dW^i_t, \quad \ i \neq 1 \\
\nu^n_t &= \frac{1}{n}\sum_{k=1}^n\delta_{Y^k_t}.
\end{align*}
Define $\widetilde\beta \in \RCM_n$ by setting $\widetilde\beta(t,\bm{x}) = \E[ \beta(t,\bm{Y}) \, | \, \bm{Y}_t=\bm{x}]$, noting that the existence of a jointly measurable version of this conditional mean measure is demonstrated by Lemma \ref{le:conditional-marginal-meanmeasure}. More precisely, this defines a Borel measurable function $\widetilde\beta : [0,T] \times (\R^d)^n \rightarrow \P(A)$ such that
\begin{align}
\int_A \varphi(t,\bm{Y}_t,a)\widetilde\beta(t,\bm{Y}_t)(da) = \E\left[ \int_A \varphi(t,\bm{Y}_t,a) \beta(t,\bm{Y})(da) \, \Big| \, \bm{Y}_t\right], \ \ a.s., \ a.e. \ t \in [0,T]. \label{pf:np-inclusion1-2}
\end{align}
By Theorem \ref{th:markovprojection}, the unique solution $\bm{\widetilde{Y}}=(\widetilde{Y}^1,\ldots,\widetilde{Y}^n)$ of the SDE system
\begin{align*}
d\widetilde{Y}^1_t &= \int_A b(t,\widetilde{Y}^1_t,\widetilde{\nu}^n_t,a)\widetilde\beta(t,\bm{\widetilde{Y}}_t)(da)dt + d{W}^1_t, \\
d\widetilde{Y}^i_t &= \int_A b(t,\widetilde{Y}^i_t,\widetilde{\nu}^n_t,a)\Lambda^i(t,\bm{\widetilde{Y}}_t)(da)dt + d{W}^i_t, \quad \ i \neq 1 \\
\widetilde{\nu}^n_t &= \frac{1}{n}\sum_{k=1}^n\delta_{\widetilde{Y}^k_t}.
\end{align*}
satisfies $\bm{\widetilde{Y}}_t \stackrel{d}{=} \bm{Y}_t$ for all $t \in [0,T]$. Using Fubini's theorem and \eqref{pf:np-inclusion1-2}, we find
\begin{align*}
J^n_1(\beta,\Lambda^2,\ldots,\Lambda^n) &= \E\left[\int_0^T\int_Af(t,Y^1_t,\nu^n_t,a)\beta(t,\bm{Y})(da)dt + g(Y^1_T,\nu^n_T)\right] \\
	&= \E\left[\int_0^T\int_Af(t,Y^1_t,\nu^n_t,a)\widetilde\beta(t,\bm{Y}_t)(da)dt + g(Y^1_T,\nu^n_T)\right] \\
	&= \E\left[\int_0^T \int_A f(t,\widetilde{Y}^1_t,\widetilde{\nu}^n_t, a)\widetilde\beta(t,\bm{\widetilde{Y}}_t)(da)dt + g(\widetilde{Y}^1_T,\widetilde{\nu}^n_T)\right] \\
	&= J^n_1(\widetilde\beta,\Lambda^2,\ldots,\Lambda^n) \\
	&\le J^n_1(\Lambda^1,\ldots,\Lambda^n)  + \epsilon.
\end{align*}
Indeed, the last inequality follows from the assumption that $(\Lambda^1,\ldots,\Lambda^n)$ is a relaxed Markovian $\epsilon$-Nash equilibrium.
\hfill\qedsymbol

\subsection{Proof of Proposition \ref{pr:np-eq-inclusion2}}

{\ } \\
\noindent\textit{Proof of (a):} 
Let $\epsilon \ge 0$, and fix a Markovian $\epsilon$-Nash equilibrium $\bm{\alpha}=(\alpha^1,\ldots,\alpha^n) \in \AM_n^n$. The goal is to show that $\bm{\alpha}$ is also a relaxed Markovian $\epsilon$-Nash equilibrium. The state processes $\bm{X}=\bm{X}[\bm{\alpha}]$ solve the SDE system
\begin{align*}
dX^i_t &= b(t,X^i_t,\mu^n_t,\alpha^i(t,\bm{X}_t))dt + dW^i_t, \quad \mu^n_t = \frac{1}{n}\sum_{k=1}^n\delta_{X^k_t}.
\end{align*}
Let $\beta \in \RCM_n$ be an alternative relaxed control. We will focus on player $1$, showing that
\begin{align}
J^n_1(\alpha^1,\ldots,\alpha^n) \ge J^n_1(\beta,\alpha^2,\ldots,\alpha^n) - \epsilon. \label{pf:np-inclusion2-1}
\end{align}
The argument for other players $i \neq 1$ is identical. To proceed, let $\bm{Y}=(Y^1,\ldots,Y^n) = \bm{X}[\beta,\alpha^2,\ldots,\alpha^n]$ be the state processes in which players $i\neq 2$ still use $\alpha^i$:
\begin{align*}
dY^1_t &= \int_A b(t,Y^1_t,\nu^n_t,a)\beta(t,\bm{Y}_t)(da)dt + dW^1_t, \\
dY^i_t &= b(t,Y^i_t,\nu^n_t,\alpha^i(t,\bm{Y}_t))dt + dW^i_t, \quad \ i \neq 1 \\
\nu^n_t &= \frac{1}{n}\sum_{k=1}^n\delta_{Y^k_t}.
\end{align*}
Recalling the definition of the convex set $K(t,x,m)$ from Assumption \ref{assumption:B}, we have
\begin{align*}
\int_A \big(b(t,Y^1_t,\nu^n_t,a),f(t,Y^1_t,\nu^n_t,a)\big)\beta(t,\bm{Y}_t)(da) \in K(t,Y^1_t,\nu^n_t).
\end{align*}
Let $L_n : (\R^d)^n \rightarrow \P(\R^d)$ denote the empirical measure map, $L_n(\bm{x}) = \frac{1}{n}\sum_{k=1}^n\delta_{x_k}$.
Using a measurable selection theorem \cite[Theorem A.9]{haussmannlepeltier-existence}, we may find a measurable function $\widetilde\alpha : [0,T] \times (\R^d)^n \rightarrow A$ such that, for each $\bm{x}=(x_1,\ldots,x_n) \in (\R^d)^n$,
\begin{align}
b(t,x_1,L_n(\bm{x}),\widetilde\alpha(t,\bm{x})) &= \int_A  b(t,x_1,L_n(\bm{x}),a) \beta(t,\bm{x})(da) \label{pf:np-inclusion2-2}
\end{align}
and
\begin{align*}
\int_A  f(t,x_1,L_n(\bm{x}),a) \beta(t,\bm{x})(da)  &\le f(t,x_1,L_n(\bm{x}),\widetilde\alpha(t,\bm{x})).
\end{align*}
The first of these identities implies that in fact
\begin{align*}
dY^1_t &= b(t,Y^1_t,\nu^n_t,\widetilde\alpha(t,\bm{Y}_t))dt + dW^1_t,
\end{align*}
and in particular  $\bm{Y} \stackrel{d}{=} \bm{X}[\widetilde\alpha,\alpha^2,\ldots,\alpha^n]$,
while the second implies
\begin{align*}
J^n_1(\beta,\alpha^2,\ldots,\alpha^n) &= \E\left[\int_0^T \int_A f(t,Y^1_t,\nu^n_t,a)\beta(t,\bm{Y}_t)(da)dt + g(Y^1_T,\nu^n_T)\right] \\
	&\le \E\left[\int_0^T f(t,Y^1_t,\nu^n_t,\widetilde\alpha(t,\bm{Y}_t))dt + g(Y^1_T,\nu^n_T)\right] \\
	&= J^n_1(\widetilde\alpha,\alpha^2,\ldots,\alpha^n) \\
	&\le J^n_1(\alpha^1,\ldots,\alpha^n) + \epsilon.
\end{align*}
Indeed, the last inequality follows from the assumption that $(\alpha^1,\ldots,\alpha^n)$ is a Markovian $\epsilon$-Nash equilibrium.

{\ } \\
\noindent\textit{Proof of (b):} This proof is identical to that of part (a), except that all of the controls involved (namely, $\alpha^i$ and $\beta$) are closed-loop (path-dependent) instead of Markovian.

\subsection{Proof of Proposition \ref{pr:np-eq-inclusion0}}
{\ } \\
This combines ideas of both of the previous proofs.
Let $\epsilon \ge 0$, and fix a Markovian $\epsilon$-Nash equilibrium $\bm{\alpha}=(\alpha^1,\ldots,\alpha^n) \in \AM_n^n$. The goal is to show $\bm{\alpha}$ is a closed-loop $\epsilon$-Nash equilibrium. The state processes $\bm{X}=\bm{X}[\bm{\alpha}]$ solve the SDE system
\begin{align*}
dX^i_t &= b(t,X^i_t,\mu^n_t,\alpha^i(t,\bm{X}_t))dt + dW^i_t, \quad \mu^n_t = \frac{1}{n}\sum_{k=1}^n\delta_{X^k_t}.
\end{align*}
Let $\beta \in \A_n$ be an alternative closed-loop control. We will focus on player $1$, showing that
\begin{align}
J^n_1(\alpha^1,\ldots,\alpha^n) \ge J^n_1(\beta,\alpha^2,\ldots,\alpha^n) - \epsilon. \label{pf:np-inclusion2-3}
\end{align}
Let $\bm{Y}=(Y^1,\ldots,Y^n)= \bm{X}[\beta,\alpha^2,\ldots,\alpha^n]$ be the state processes in which players $i\neq 2$ still use $\alpha^i$:
\begin{align*}
dY^1_t &= b(t,Y^1_t,\nu^n_t,\beta(t,\bm{Y}))dt + dW^1_t, \\
dY^i_t &= b(t,Y^i_t,\nu^n_t,\alpha^i(t,\bm{Y}_t))dt + dW^i_t, \quad \ i \neq 1 \\
\nu^n_t &= \frac{1}{n}\sum_{k=1}^n\delta_{Y^k_t}.
\end{align*}
Recalling the definition of the convex set $K(t,x,m)$ from Assumption \ref{assumption:B}, notice that
\begin{align*}
\E\left[b(t,Y^1_t,\nu^n_t,\beta(t,\bm{Y})) \, | \, \bm{Y}_t\right] \in K(t,Y^1_t,\nu^n_t), \ \ a.s.,
\end{align*}
for each $t \in [0,T]$. 
Using a measurable selection theorem \cite[Theorem A.9]{haussmannlepeltier-existence}, we may find a measurable function $\widetilde\beta : [0,T] \times (\R^d)^n \rightarrow A$ such that
\begin{align}
b(t,Y^1_t,\nu^n_t,\widetilde\beta(t,\bm{Y}_t)) &= \E\left[b(t,Y^1_t,\nu^n_t,\beta(t,\bm{Y})) \, | \, \bm{Y}_t\right], \ \ a.s.  \label{pf:np-inclusion2-4} \\
f(t,Y^1_t,\nu^n_t,\widetilde\beta(t,\bm{Y}_t)) &\ge \E\left[f(t,Y^1_t,\nu^n_t,\beta(t,\bm{Y})) \, | \, \bm{Y}_t\right], \ \ a.s., \label{pf:np-inclusion2-5} 
\end{align}
for each $t \in [0,T]$.
By Theorem \ref{th:markovprojection}, the unique solution $\bm{\widetilde{Y}}=(\widetilde{Y}^1,\ldots,\widetilde{Y}^n)$ of the SDE system
\begin{align*}
d\widetilde{Y}^1_t &= b(t,\widetilde{Y}^1_t,\widetilde{\nu}^n_t,\widetilde\beta(t,\bm{\widetilde{Y}}_t))dt + d{W}^1_t, \\
d\widetilde{Y}^i_t &= b(t,\widetilde{Y}^i_t,\widetilde{\nu}^n_t,\alpha^i(t,\bm{\widetilde{Y}}_t))dt + d{W}^i_t, \quad \ i \neq 1 \\
\widetilde{\nu}^n_t &= \frac{1}{n}\sum_{k=1}^n\delta_{\widetilde{Y}^k_t}.
\end{align*}
satisfies $\bm{\widetilde{Y}}_t \stackrel{d}{=} \bm{Y}_t$ for all $t \in [0,T]$. Note that $\bm{\widetilde{Y}} \stackrel{d}{=} \bm{X}[\widetilde\beta,\alpha^2,\ldots,\alpha^n]$. Use Fubini's theorem and \eqref{pf:np-inclusion2-5} to get
\begin{align*}
J^n_1(\beta,\alpha^2,\ldots,\alpha^n) &= \E\left[\int_0^T f(t,Y^1_t,\nu^n_t,\beta(t,\bm{Y}))dt + g(Y^1_T,\nu^n_T)\right] \\
	&\le \E\left[\int_0^T f(t,Y^1_t,\nu^n_t, \widetilde\beta(t,\bm{Y}_t))dt + g(Y^1_T,\nu^n_T)\right] \\
	&= \E\left[\int_0^T f(t,\widetilde{Y}^1_t,\widetilde{\nu}^n_t,\widetilde\beta(t,\bm{\widetilde{Y}}_t))dt + g(\widetilde{Y}^1_T,\widetilde{\nu}^n_T)\right] \\
	&= J^n_1(\widetilde\beta,\alpha^2,\ldots,\alpha^n) \\
	&\le J^n_1(\alpha^1,\ldots,\alpha^n)  + \epsilon.
\end{align*}
Indeed, the last inequality follows from the assumption that $(\alpha^1,\ldots,\alpha^n)$ is a Markovian $\epsilon$-Nash equilibrium.

\hfill\qedsymbol

\subsection{Proof of Proposition \ref{pr:np-eq-inclusion3}}

{\ } \\
\noindent\textit{Proof of (a):} 
Let $\epsilon \ge 0$, and fix a relaxed Markovian $\epsilon$-Nash equilibrium $\bm{\Lambda}=(\Lambda^1,\ldots,\Lambda^n) \in \RCM_n^n$. The state processes $\bm{X}=\bm{X}[\bm{\Lambda}]$ solve the SDE system
\begin{align*}
dX^i_t &= \int_{A}b(t,X^i_t,\mu^n_t,a)\Lambda^i(t,\bm{X}_t)(da)dt + dW^i_t, \quad \mu^n_t = \frac{1}{n}\sum_{k=1}^n\delta_{X^k_t}.
\end{align*}
Let $L_n : (\R^d)^n \rightarrow \P(\R^d)$ denote the empirical measure map, $L_n(\bm{x}) = \frac{1}{n}\sum_{k=1}^n\delta_{x_k}$.
Recalling the definition of $K(t,x,m)$ from Assumption \ref{assumption:B}, it holds for each $t \in [0,T]$, $\bm{x}=(x_1,\ldots,x_n) \in (\R^d)^n$, and $i \in \{1,\ldots,n\}$ that
\begin{align*}
\int_A \big(b(t,x_i,L_n(\bm{x}),a),f(t,x_i,L_n(\bm{x}),a)\big)\Lambda^i(t,\bm{x})(da) \in K(t,x_i,L_n(\bm{x})).
\end{align*}
Using a measurable selection theorem \cite[Theorem A.9]{haussmannlepeltier-existence}, we may find a measurable function $\alpha^i : [0,T] \times (\R^d)^n \rightarrow A$  such that, for each $\bm{x}=(x_1,\ldots,x_n) \in (\R^d)^n$ and $t \in [0,T]$,
\begin{align}
b(t,x_i,L_n(\bm{x}),\alpha^i(t,\bm{x})) &= \int_A  b(t,x_i,L_n(\bm{x}),a) \Lambda^i(t,\bm{x})(da) \label{pf:np-inclusion3-1}
\end{align}
and
\begin{align*}
\int_A  f(t,x_i,L_n(\bm{x}),a) \Lambda^i(t,\bm{x})(da)  &\le f(t,x_i,L_n(\bm{x}),\alpha^i(t,\bm{x})).
\end{align*}
The first of these identities implies that in fact $\bm{X}$ solves the SDE system
\begin{align*}
dX^i_t &= b(t,X^i_t,\mu^n_t,\alpha^i(t,\bm{X}_t))dt + dW^i_t,
\end{align*}
i.e., $\bm{X} = \bm{X}[\alpha^1,\ldots,\alpha^n]$, while the second implies
\begin{align}
J^n_i(\Lambda^1,\ldots,\Lambda^n) &= \E\left[\int_0^T\int_Af(t,X^i_t,\mu^n_t,a)\Lambda^i(t,\bm{X}_t)(da)dt + g(X^i_T,\mu^n_T)\right] \nonumber \\
	&\le \E\left[\int_0^T f(t,X^i_t,\mu^n_t, \alpha^i(t,\bm{X}_t))dt + g(X^i_T,\mu^n_T)\right] \nonumber \\
	&= J^n_i(\alpha^1,\ldots,\alpha^n).  \label{pf:np-inclusion3-3}
\end{align}

Now, let us show that $(\alpha^1,\ldots,\alpha^n)$ is an $\epsilon$-Nash equilibrium. Fix an alternative Markovian control $\beta \in \AM_n$. We will focus on player $1$, showing that
\begin{align}
J^n_1(\alpha^1,\ldots,\alpha^n) \ge J^n_1(\beta,\alpha^2,\ldots,\alpha^n) - \epsilon. \label{pf:np-inclusion3-2}
\end{align}
To proceed, let $\bm{Y}=(Y^1,\ldots,Y^n)=\bm{X}[\beta,\Lambda^2,\ldots,\Lambda^n]$ be the state processes in which players $i\neq 2$ still use the relaxed controls $\Lambda^i$:
\begin{align*}
dY^1_t &= b(t,Y^1_t,\nu^n_t,\beta(t,\bm{Y}_t))dt + dW^1_t, \\
dY^i_t &= \int_{A}b(t,Y^i_t,\nu^n_t,a)\Lambda^i(t,\bm{Y}_t)(da)dt + dW^i_t, \quad \ i \neq 1 \\
\nu^n_t &= \frac{1}{n}\sum_{k=1}^n\delta_{Y^k_t}.
\end{align*}
Using \eqref{pf:np-inclusion3-1} we may write
\begin{align*}
dY^i_t &= b(t,Y^i_t,\nu^n_t,\alpha^i(t,\bm{Y}_t))dt + dW^i_t, \quad \ i \neq 1,
\end{align*}
i.e., $\bm{Y} \stackrel{d}{=} \bm{X}[\beta,\alpha^2,\ldots,\alpha^n]$.
Hence, since $(\Lambda^1,\ldots,\Lambda^n)$ is $\epsilon$-Nash, we may use \eqref{pf:np-inclusion3-3} to get
\begin{align*}
J^n_1(\beta,\alpha^2,\ldots,\alpha^n) &= J^n_1(\beta,\Lambda^2,\ldots,\Lambda^n) \le J^n_1(\Lambda^1,\ldots,\Lambda^n) + \epsilon \\
	&\le J^n_1(\alpha^1,\ldots,\alpha^n) + \epsilon.
\end{align*}

{\ } \\
\noindent\textit{Proof of (b):} This proof is identical to that of part (a), except that all of the controls involved (namely, $\Lambda^i$, $\alpha^i$, and $\beta$) are closed-loop (path-dependent) instead of Markovian.
\hfill\qedsymbol

\subsection{Proof of Proposition \ref{pr:RMFE-to-MFE}}
{\ } \\
We prove the claims only for weak MFE, as the strong MFE is a special case of a deterministic weak MFE.
We begin with a preparatory argument. Let $\Lambda : [0,T] \times \R^d \times \CP \rightarrow \P(A)$ be any semi-Markov function (recall Definition \ref{def:semiMarkovFunction}).
 Recalling the definition of the convex set $K(t,x,m)$ from Assumption \ref{assumption:B}, note that for $(t,x,m) \in [0,T] \times \R^d \times \CP$ we have
\begin{align*}
\int_A \big(b(t,x,m_t,a),f(t,x,m_t,a)\big)\Lambda(t,x,m)(da) \in K(t,x,m_t).
\end{align*}
Using a measurable selection theorem \cite[Theorem A.9]{haussmannlepeltier-existence}, we may find a semi-Markov function $\alpha^\Lambda : [0,T] \times \R^d \times \CP \rightarrow A$  such that, for each $(t,x,m)$,
\begin{align*}
b(t,x,m_t,\alpha^\Lambda(t,x,m)) &= \int_A  b(t,x,m,a) \Lambda (t,x,m)(da)
\end{align*}
and
\begin{align*}
\int_A  f(t,x,m,a) \Lambda (t,x,m)(da)  &\le f(t,x,m,\alpha^\Lambda(t,x,m)).
\end{align*}
In particular, if $X$ solves the SDE
\begin{align}
dX_t = \int_A b(t,X_t,\mu_t,a)\Lambda(t,X_t,\mu)(da)dt + dW_t, \label{pf:RMFE-to-MFE-2}
\end{align}
where $W$ is a Brownian motion, $\mu$ is a continuous $\P(\R^d)$-valued process, and $(X_0,\mu,W)$ are independent, then $X$ also solves the SDE
\[
dX_t = b(t,X_t,\mu_t,\alpha^\Lambda(t,X_t,\mu)) dt + dW_t,
\]
and we have the inequality
\begin{align}
\E &\left[\int_0^T\int_A f(t,X_t,\mu_t,a)\Lambda(t,X_t,\mu)(da)dt + g(X_T,\mu_T)\right] \nonumber \\
	&\le \E\left[\int_0^T f(t,X_t,\mu_t,\alpha^\Lambda(t,X_t,\mu))dt + g(X_T,\mu_T)\right].  \label{pf:RMFE-to-MFE-3}
\end{align}

With this construction the proof is straightforward. We first show that a weak RMFE is a  weak MFE. Let $(\Omega,\F,\FF,\PP,W,\Lambda^*,X^*,\mu)$ be a weak RMFE. It is then easy to check using the above facts that $(\Omega,\F,\FF,\PP,W,\alpha^{\Lambda^*},X^*,\mu)$ is a weak MFE. Conversely, let $(\Omega,\F,\FF,\PP,W,\alpha^*,X^*,\mu)$ be a weak MFE. Define $\Lambda^*(t,x,m) := \delta_{\alpha^*(t,x,m)}$. It is clear that $(\Omega,\F,\FF,\PP,W,\Lambda^*,X^*,\mu)$ satisfies properties (1-4) and (6) of Definition \ref{def:MFEsemimarkov}. To prove (5), let $\Lambda : [0,T] \times \R^d \times \CP \rightarrow \P(A)$ denote any semi-Markov function, and let $X$ solve the corresponding SDE \eqref{pf:RMFE-to-MFE-2}. Combine property (5) of the definition of weak MFE (Definition \ref{def:MFEsemimarkov-strict}) with \eqref{pf:RMFE-to-MFE-3} to get
\begin{align*}
\E &\left[\int_0^T\int_A f(t,X^*_t,\mu_t,a)\Lambda^*(t,X^*_t,\mu)(da)dt + g(X^*_T,\mu_T)\right] \\
	&= \E\left[\int_0^T f(t,X^*_t,\mu_t,\alpha^*(t,X^*_t,\mu))dt + g(X^*_T,\mu_T)\right] \\
	&\ge \E\left[\int_0^T f(t,X_t,\mu_t,\alpha^\Lambda(t,X_t,\mu))dt + g(X_T,\mu_T)\right] \\
	&\ge \E\left[\int_0^T\int_A f(t,X_t,\mu_t,a)\Lambda(t,X_t,\mu)(da)dt + g(X_T,\mu_T)\right].
\end{align*}
This completes the proof.

\section{Proof of the main limit theorem} \label{se:mainlimitproof}

This section is devoted to the proof of Theorem \ref{th:mainlimit-relaxed}, from which Theorem \ref{th:mainlimit} follows (see Section \ref{se:extensions-relaxed}). We break this up into three major steps. First, we show tightness, which is straightforward in the present context. Next, we identify the limiting dynamics, in the sense that we prove that properties (1-4) and (6) of Definition \ref{def:MFEsemimarkov} hold at the limit. Lastly, we address the optimality condition (5).

In fact, before we prove Theorem \ref{th:mainlimit-relaxed}, we will carry out the bulk of the analysis is without using the fact that the $n$-player controls are given as $\epsilon_n$-Nash equilibria. That is, much of the work of characterizing the limiting behavior can and should be done independently of the Nash property. Only at the end will we use the Nash property to produce an inequality, which is then passed to the limit to obtain the desired optimality condition.

In the following, we work with an arbitrary sequence of controls $(\alpha^{n,1},\ldots,\alpha^{n,n}) \in \A_n^n$. We write $\bm{X}^n=(X^{n,1},\ldots,X^{n,n})$ to denote the corresponding state process in the $n$-player game, which we now index by $n$ for clarity, and which is determined as the unique in law solution of the SDE
\begin{align}
dX^{n,i}_t &= b(t,X^{n,i}_t,\mu^n_t,\alpha^{n,i}(t,\bm{X}^n))dt + dW^i_t, \quad\quad \mu^n_t = \frac{1}{n}\sum_{k=1}^n\delta_{X^{n,k}_t}. \label{def:stateprocesses-limitsection}
\end{align}
As usual, $X^{n,1}_0,\ldots,X^{n,n}_0$ are i.i.d.\ with law $\lambda$.

It is convenient in the following to work with probability measures on the path space rather the flows of probability measures on $\R^d$. To distinguish between the two, we will reserve bold font for the former. For $\bm{m} \in \P(\C^d)$, define for each $t \in [0,T]$ the marginal law $m_t = \bm{m} \circ [x \mapsto x_t]^{-1}$, and note that the map
\[
\P(\C^d) \ni \bm{m} \mapsto m:=(m_t)_{t \in [0,T]} \in \CP 
\]
is continuous. Given $\bm{m} \in \P(\C^d)$, we will refer to this $m=(m_t)_{t \in [0,T]}$ as the \emph{induced} or \emph{corresponding measure flow}. To keep track of notation, we stick to the following rules:
\begin{itemize}
\item We use the Latin $m$ for a deterministic measure and the Greek $\mu$ for a random measure.
\item We use boldface for a measure on path space, $\bm{m} \in \P(\C^d)$, to distinguish it from a measure flow, written as $m=(m_t)_{t \in [0,T]} \in \CP$.
\item Starting in Section \ref{se:mainproof-limitingdynamics}, we will encounter probability measures on the extended path space $\C^d \times \V$, with $\V$ defined in Section \ref{se:relaxedcontrols}. We denote such measures as $\bm{\overline{m}} \in \P(\C^d \times \V)$.
\end{itemize}
Define $\bm{\mu}^n$, a random element of $\P(\C^d)$, by
\[
\bm{\mu}^n = \frac{1}{n}\sum_{k=1}^n\delta_{X^{n,k}}.
\]
In light of the previous discussion, if $\bm{\mu}^n$ converges in law in $\P(\C^d)$ to some $\bm{\mu}$, then the marginal flow $\mu^n=(\mu^n_t)_{t \in [0,T]}$ converges in law in $\CP $ to the corresponding marginal flow $\mu=(\mu_t)_{t \in [0,T]}$.

Let $C^\infty_c(\R^d)$ denote the set of smooth functions of compact support. We define the infinitesimal generator of the controlled process as follows: For $\varphi \in C^\infty_c(\R^d)$, define
\begin{align}
L\varphi(t,x,m,a) := b(t,x,m,a) \cdot \nabla\varphi(x) + \frac12\Delta\varphi(x), 
\end{align}
for $(t,x,m,a) \in [0,T] \times \R^d \times \P(\R^d) \times A$.

\subsection{Tightness}

We first prove that $(\bm{\mu}^n)$ is tight.

\begin{lemma} \label{le:tightness-main}
The sequence $(\bm{\mu}^n)$ is a tight family of $\P(\C^d)$-valued random variables.
\end{lemma}
\begin{proof}
According to \cite[(2.5)]{sznitman1991topics}, it suffices to show that the sequence of mean measures $(\bm{m}^n) \subset \P(\C^d)$ is tight, where we define $\bm{m}^n$ for Borel sets $B \subset \C^d$ by
\[
\bm{m}^n(B) = \E[\bm{\mu}^n(B)] = \frac{1}{n}\sum_{k=1}^n\PP(X^{n,k} \in B).
\]
Letting $\|b\|_\infty$ denote the minimal uniform bound on $|b|$, note that $|L\varphi|$ is pointwise bounded by the constant
\[
C_\varphi := \|b\|_\infty\|\nabla\varphi\|_\infty + \frac12\|\Delta\varphi\|_\infty.
\]
By It\^o's formula, for every $\varphi \in C^\infty_c(\R^d)$ the process $(\varphi(X^k_t) + C_\varphi t)_{t \in [0,T]}$ is a submartingale. It follows from \cite[Theorem 1.4.6]{stroock-varadhan} that $\{X^{n,k} : n \in \N, \, k =1,\ldots,n\}$ is a tight family of $\C^d$-valued random variables. Hence, $(\bm{m}^n)$ is tight.
\end{proof}

\subsection{Relaxed controls} \label{se:relaxedcontrols}
Before we proceed to identify the dynamics of limit points of $(\mu^n)$, we must first discuss a convenient topological space in which to view the controls.
Let $\V$ denote the set of measures $q$ on $[0,T] \times A$ with first marginal equal to Lebesgue measure. Equip $\V$ with the topology of weak convergence, and note that $\V$ is a compact metric space because $A$ is.See \cite[Appendix A]{lacker2016general} for a summary of basic facts about this space and references.

Each $q \in \V$ may be identified with a measurable function $[0,T] \ni t \mapsto q_t \in \P(A)$, determined uniquely (up to a.e. equality) by $dtq_t(da) = q(dt,da)$.
Similarly, a measurable $\P(A)$-valued process $(\Lambda_t)_{t \in [0,T]}$ can be identified with the random element $\Lambda = dt\Lambda_t(da)$ of $\V$. 
It is known that one can construct a measurable version of the \emph{canonical process} on $\V$. More precisely, suppose $\FF^\V=(\F^\V_t)_{t \in [0,T]}$ denotes the natural filtration, where for each $t \in [0,T]$ we define $\F^\V_t$ as the $\sigma$-field generated by the functions $\V \ni q \mapsto q(B) \in \R$, for Borel sets $B \subset [0,t] \times A$. Then there exists (see \cite[Lemma 3.2]{lacker2015mean}) an $\FF^\V$-predictable process
\begin{align}
\widehat{q} : [0,T] \times \V \rightarrow \P(A), \quad \text{such that} \quad \widehat{q}(t,q)=q_t, \ a.e. \ t, \ \forall q \in \V. \label{def:predictablecanonical-V}
\end{align}
In particular, the filtration generated by the process $(\widehat{q}(t,\cdot))_{t \in [0,T]}$ is precisely $\FF^\V$.
With this in mind, we are free to identify $\P(A)$-valued processes and $\V$-valued random variables.

\subsection{Projection lemmas} \label{se:mainproof-projections}

As a preparation for the next step of identifying the dynamics of the limiting measure flows, we begin with two projection arguments that will be useful again in later sections. The first is straightforward but worth summarizing, while the second hides some delicate measurability questions which are largely outsourced to the appendix. 
In the following, it is convenient to use the usual duality notation for integration:
\[
\langle m,\varphi\rangle  = \int \varphi\,dm.
\]

\begin{lemma} \label{le:conditional-exp-derivatives}
Suppose $(Y_t)_{t \in [0,T]}$ is a continuous stochastic process taking values in a Polish space $E$ and defined on some probability space $(\Omega,\F,\PP)$. Suppose $h : E \rightarrow \R$ is continuous, and suppose it holds almost surely that
\[
h(Y_t) = h(Y_0) + \int_0^ta_sds, \ \ a.s., \ \text{ for a.e. } t \in [0,T],
\]
where $(a_t)_{t \in [0,T]}$ is some bounded measurable real-valued process. Suppose $\hat{a} : [0,T] \times \CE \rightarrow \R$ is a progressively measurable function satisfying
\[
\hat{a}(t,Y) = \E[a_t \, | \, \F^Y_t], \ \ a.s., \ \text{ for a.e. } t \in [0,T],
\]
where $\F^Y_t = \sigma(Y_s : s \le t)$.
Then 
\[
h(Y_t) = h(Y_0) + \int_0^t\hat{a}(s,Y)ds, \ \ \text{ for all } t \in [0,T], \ \ a.s.
\]
\end{lemma}
\begin{proof}
By continuity, we have
\[
h(Y_t) = h(Y_0) + \int_0^ta_sds, \ \ a.s., \ \ \text{ for all } t \in [0,T], \ \ a.s.
\]
Hence, it holds a.s.\ that for almost every $t \in (0,T)$ we have
\begin{align*}
a_t &= \lim_{\delta \downarrow 0}\frac{1}{\delta}\left(h(Y_t)-h(Y_{t-\delta})\right).
\end{align*}
In particular, $a_t$ is $\F^Y_t$-measurable, and so $a_t = \E[a_t \, | \, \F^Y_t] = \hat{a}(t,Y)$, a.s., for each $t \in [0,T]$. Complete the proof by integrating this identity and using continuity of $h$ and $Y$ to interchange the order of quantifiers as needed.
\end{proof}

In the following lemma, we show that a solution of a certain kind of randomized Fokker-Planck equation can be realized as the conditional law of the state process under a semi-Markov control. Recall the notion of \emph{semi-Markov function} from Definition \ref{def:semiMarkovFunction}.

\begin{lemma} \label{le:projection-semiMarkov}
Suppose $\bm{\overline{\mu}}$ is a $\P(\C^d \times \V)$-valued random variable, and let $\mu=(\mu_t)_{t\in [0,T]}$ denote the corresponding measure flow.\footnote{That is, the $\CP$-valued random variable $\mu$ is defined by $\mu_t := \bm{\overline\mu} \circ [(x,q) \mapsto x_t]^{-1}$, for $t \in [0,T]$.} Suppose it holds with probability $1$ that for every $t \in [0,T]$ and $\varphi \in C_c^\infty(\R^d)$ we have
\begin{align}
\langle\mu_t,\varphi\rangle &= \langle\mu_0,\varphi\rangle + \int_{\C^d\times\V}\left[\int_0^t\int_AL\varphi(s,x_s,\mu_s,a)q_s(da)ds \right] \bm{\overline\mu}(dx,dq). \label{ass:le:projection-semiMarkov}
\end{align}
Then there exists a semi-Markov function $\Lambda^* : [0,T] \times \R^d \times \CP \rightarrow \P(A)$ such that the following hold:
\begin{enumerate}[(a)]
\item It holds with probability $1$ that for every $t \in [0,T]$ and $\varphi \in C_c^\infty(\R^d)$ we have
\begin{align*}
\langle \mu_t,\varphi\rangle = \langle \mu_0,\varphi\rangle + \int_0^t\left\langle \mu_s, \, \int_A L\varphi(s,\cdot,\mu_s,a) \Lambda^*(t,\cdot,\mu)(da)\right\rangle ds.
\end{align*}
\item For each bounded measurable function $\psi$ on $[0,T] \times \R^d \times \P(\R^d) \times A$, we have
\begin{align*}
\E&\left[\int_{\C^d \times \V}\int_0^T\int_A \psi(t,x_t,\mu_t,a)q_t(da)dt\,\bm{\overline\mu}(dx,dq)\right] \\
 &= \E\left[\int_0^T\int_{\R^d}\int_A \psi(t,x,\mu_t,a)\,\Lambda^*(t,x,\mu)(da)\,\mu_t(dx)\,dt\right].
\end{align*}
\item By enlarging the probability space, we may construct continuous $d$-dimensional processes $X$ and $W$ such that:
\begin{enumerate}[(i)]
\item $X_0$, $W$, and $\mu$ are independent.
\item $W$ is a Brownian motion with respect to the complete filtration $\FF =(\F_t)_{t \in [0,T]}$ generated by the process $(X_0,\mu_t,W_t)_{t \in [0,T]}$.
\item $X$ is a continuous process with $X_0 \sim \lambda$, adapted to the completion of $\FF$.
\item The state equation holds,
\begin{align*}
dX_t = \int_A b(t,X_t,\mu_t,a)\Lambda^*(t,X_t,\mu)(da)dt + dW_t.
\end{align*}
\item For each $t$, it holds a.s.\ that $\mu_t = \L(X_t \, | \, \F^{\mu}_t)$, where $\F^{\mu}_t = \sigma(\mu_s : s \le t)$.
\end{enumerate}
\end{enumerate}
\end{lemma}
\begin{proof}
We first justify (c), assuming we have already found $\Lambda^*$ such that (a) and (b) hold. In fact, the claimed processes $X$ and $W$ come from the observation that property (a) is simply a randomized version of a Fokker-Planck equation. Corollary \ref{co:forwardeq-random} works out the details and shows that we can construct $X$ and $W$ satisfying properties (i-v).

To construct $\Lambda^*$ satisfying (a) and (b), we note first that \eqref{ass:le:projection-semiMarkov} rewrites as
\begin{align*}
\langle\mu_t,\varphi\rangle &= \langle\mu_0,\varphi\rangle + \int_{\C^d\times\V}\left[\int_0^t\int_AL\varphi(s,x_s,\mu_s,a)\widehat{q}(s,q)(da)ds \right] \bm{\overline\mu}(dx,dq),
\end{align*}
where $\widehat{q}$ is the ``nice version" of the process $[0,T] \times \V \ni (t,q) \mapsto q_t \in \P(A)$ described in \eqref{def:predictablecanonical-V}.

Suppose for concreteness that the random variable $\bm{\overline\mu}$ is defined on a probability space $(\Omega,\F,\PP)$, and we may assume without loss of generality that $\Omega$ is a Polish space and $\F$ its Borel $\sigma$-field. We now use Lemma \ref{le:measurableversion-marginalprojection} to construct a jointly measurable version of the regular condition law of $(x_t,q)$ given $x_t$ under the random probability measure $\bm{\overline\mu}(dx,dq)$; precisely, there exists a jointly measurable map $[0,T] \times \R^d \times \Omega \ni (t,x,\omega) \mapsto \bm{\overline\mu}_{t,x}(\omega) \in \P(\C^d \times \V)$ such that it holds a.s.\ that for every bounded measurable function $h : [0,T] \times \R^d \rightarrow \R$ and $F : [0,T] \times \C^d \times \V \rightarrow \R$ we have
\begin{align*}
\int_{\C^d \times \V}\int_0^T h(t,x_t) F(t,x,q) dt\bm{\overline\mu}(dx,dq) &= \int_{\C^d \times \V}\int_0^T h(t,x_t)\langle \bm{\overline\mu}_{t,x_t}, F(t,\cdot)\rangle dt\bm{\overline\mu}(dx,dq), \ \ a.s.
\end{align*}
Using Fubini's theorem and a change of variables, we may rewrite this as
\begin{align}
\int_{\C^d \times \V}\int_0^T h(t,x_t) F(t,x,q) dt\bm{\overline\mu}(dx,dq) &= \int_0^T \int_{\R^d} h(t,x)\langle \bm{\overline\mu}_{t,x}, F(t,\cdot)\rangle \mu_t(dx)dt, \ \ a.s. \label{pf:projection-semiMarkov1}
\end{align}
Applying this with $F(t,x,q) = \int_A L\varphi(t,x,\mu_t,a) \widehat{q}(t,q)(da)$, we may write \eqref{ass:le:projection-semiMarkov} as
\begin{align}
\langle\mu_t,\varphi\rangle = \langle\mu_0,\varphi\rangle + \int_0^t\int_{\R^d}\int_{\C^d\times\V} \int_AL\varphi(s,x,\mu_s,a)\,\widehat{q}(s,q)(da)\, \bm{\overline\mu}_{t,x}(d\tilde{x},dq)\,\mu_s(dx)ds, \ a.s. \label{pf:projection-semiMarkov2}
\end{align}

Next, recall that  $\FF^{\mu}=(\F^{\mu}_t)_{t \in [0,T]}$ is the filtration generated by the process $(\mu_t)_{t \in [0,T]}$. We may find (using Corollary \ref{le:conditional-marginal-meanmeasure-2}) a semi-Markov function $\Lambda^* : [0,T] \times \R^d \times \CP \rightarrow \P(A)$ 
such that, for every bounded measurable function $\psi : [0,T] \times \R^d \times \P(\R^d) \times A \rightarrow \R$ and every $(t,x) \in [0,T] \times \R^d$, we have
\begin{align}
\int_A \psi(t,x,\mu_t,\cdot)\,d\Lambda^*(t,x,\mu) &= \E\left[\left. \int_{\C^d\times\V}\int_A \psi(t,x,\mu_t,a)\widehat{q}(s,q)(da) \, \bm{\overline\mu}_{t,x}(d\tilde{x},dq) \right| \F^{\mu}_t\right]. \label{pf:projection-semiMarkov3}
\end{align}
Applying \eqref{pf:projection-semiMarkov3} with $\psi=L\varphi$, and using \eqref{pf:projection-semiMarkov2} and Lemma \ref{le:conditional-exp-derivatives}, we get
\begin{align*}
\langle\mu_t,\varphi\rangle = \langle\mu_0,\varphi\rangle + \int_0^t\int_{\R^d} \int_AL\varphi(s,x,\mu_s,a)\,\Lambda^*(t,x,\mu)
(da)\, \mu_s(dx)ds,
\end{align*}
for all $t \in [0,T]$, almost surely, for each $\varphi \in C_c^\infty(\R^d)$. 
This is exactly (a), once we interchange the order of the quantifiers ``almost surely" and ``for each $\varphi \in C_c^\infty(\R^d)$." This is easily justified by working with a countable dense family of such $\varphi$.

Finally, to prove (b), fix $\psi$, and simply use \eqref{pf:projection-semiMarkov1} and \eqref{pf:projection-semiMarkov3} along with Fubini's theorem:
\begin{align*}
\E&\left[\int_{\C^d \times \V}\int_0^T\int_A \psi(t,x_t,\mu_t,a)q_t(da)dt\,\bm{\overline\mu}(dx,dq)\right] \\
 &= \E\left[\int_0^T\int_{\R^d}\int_{\C^d\times\V}\int_A \psi(t,x,\mu_t,a)\,\widehat{q}(t,q)(da)\,\bm{\overline\mu}_{t,x}(d\tilde{x},dq)\,\mu_t(dx)\,dt\right] \\
 &= \E\left[\int_0^T\int_{\R^d}\int_A \psi(t,x,\mu_t,a)\,\Lambda^*(t,x,\mu)(da)\,\mu_t(dx)\,dt\right].
\end{align*}
\end{proof}

\subsection{Identification of limiting dynamics} \label{se:mainproof-limitingdynamics}
We next provide a first description of the dynamics of subsequential limit points of $\mu^{n_k}=(\mu^{n_k}_t)_{t \in [0,T]}$.

\begin{theorem} \label{th:limit-dynamics}
Suppose a subsequence $(\mu^{n_k}_t)_{t \in [0,T]}$ converges  in law in $\CP $ to $(\mu_t)_{t \in [0,T]}$. Then there exists a semi-Markov function $\Lambda^* : [0,T] \times \R^d \times \CP \rightarrow \P(A)$ such that, by extending the probability space if needed, we may construct continuous $d$-dimensional processes $X$ and $W$ such that:
\begin{enumerate}[(i)]
\item $X_0$, $W$, and $\mu$ are independent.
\item $W$ is a Brownian motion with respect to the complete filtration $\FF =(\F_t)_{t \in [0,T]}$ generated by the process $(X_0,\mu_t,W_t)_{t \in [0,T]}$.
\item $X$ is adapted with respect to the completion of $\FF$, with $X_0 \sim \lambda$.
\item The following SDE holds:
\[
dX_t = \int_A b(t,X_t,\mu_t,a)\Lambda^*(t,X_t,\mu)(da)dt + dW_t.
\]
\item For each $t$, it holds a.s.\ that $\mu_t = \L(X_t \, | \, \F^{\mu}_t)$, where $\F^{\mu}_t = \sigma(\mu_s : s \le t)$.
\end{enumerate}
Moreover,
\begin{align}
\lim_{k\rightarrow\infty}&\frac{1}{n_k}\sum_{i=1}^{n_k}J^{n_k}_i(\alpha^{n_k,1},\ldots,\alpha^{n_k,n_k}) \nonumber \\
	&= \E\left[ \int_0^T \int_Af(t,X_t,\mu_t,a)\Lambda^*(t,X_t,\mu)(da) dt  + g(X_T,\mu_T) \right]. \label{pf:limitvalue1}
\end{align}
\end{theorem}
\begin{proof} 
Let us view each control as a random element of $\V$, by defining
\[
\Lambda^{n,i}(dt,da) = dt\delta_{\alpha^{n,i}(t,\bm{X}^n)}(da).
\]
and define the extended empirical measure $\bm{\overline\mu}^n$, a $\P(\C^d \times \V)$-valued random variable, by
\[
\bm{\overline\mu}^n = \frac{1}{n}\sum_{k=1}^n\delta_{(X^{n,k},\Lambda^{n,k})}.
\]
Because the $\C^d$-marginal $\bm{\mu}^n$ is tight by Lemma \ref{le:tightness-main} and $\V$ is compact, the sequence of random measures $\bm{\overline\mu}^n$ is tight. We may then pass to a further subsequence and assume that $\bm{\overline\mu}^n$ converges in law to some random element $\bm{\overline\mu}$ of $\P(\C^d\times\V)$ whose $\C^d$-marginal is $\mu$.

{\ } \\
\noindent\textbf{Step 1:}
We first show that $\bm{\overline\mu}$ must satisfy the hypothesis \eqref{ass:le:projection-semiMarkov} of Lemma \ref{le:projection-semiMarkov}. Recall that $\bm{X}^n=(X^{n,1},\ldots,X^{n,n})$ is the vector of state processes; see \eqref{def:stateprocesses-limitsection}.
Begin by applying It\^o's formula to $\varphi(X^{n,k}_t)$ and averaging over $k=1,\ldots,n$ to get
\begin{align*}
d\langle \mu^n_t,\varphi\rangle &= \frac{1}{n}\sum_{k=1}^n L\varphi(t,X^{n,k}_t,\mu^n_t,\alpha^{n,k}(t,\bm{X}^n))dt + dM^{n,\varphi}_t \\
	&= \frac{1}{n}\sum_{k=1}^n \int_A L\varphi(t,X^{n,k}_t,\mu^n_t,a)\Lambda^{n,k}_t(da)dt + dM^{n,\varphi}_t,
\end{align*}
where we define the martingale
\[
M^{n,\varphi}_t = \frac{1}{n}\sum_{k=1}^n\nabla\varphi(X^{n,k}_s)\cdot dW^k_s.
\]
Notice that the quadratic variation of this martingale is
\begin{align}
[M^{n,\varphi}]_t = \frac{1}{n^2}\sum_{k=1}^n\int_0^t|\nabla\varphi(X^{n,k}_s)|^2ds \le \frac{\|\nabla\varphi\|_\infty^2}{n}. \label{pf:QVbound1}
\end{align}

For a measure $\bm{\overline{m}} \in \P(\C^d \times \V)$, let $m=(m_t)_{t \in [0,T]} \in \CP$ denote the associated measure flow, and define $F_t : \P(\C^d \times\V) \rightarrow \R$ by
\begin{align*}
F^\varphi_t(\bm{\overline{m}}) = \int_{\C^d\times\V}\left[\varphi(x_t)-\varphi(x_0) - \int_0^t\int_AL\varphi(s,x_s,m_s,a)q_s(da)ds \right] \bm{\overline{m}}(dx,dq).
\end{align*}
We may then write
\[
M^{n,\varphi}_t = F^\varphi_t(\bm{\overline\mu}^n),
\]
It can be shown that $F^\varphi_t$ is a bounded continuous function (see \cite[Appendix A]{lacker2015mean} for details). Because $\bm{\overline\mu}^n$ converges in law to $\bm{\overline\mu}$, we conclude from the continuous mapping theorem that $F^\varphi_t(\bm{\overline\mu^n})$ converges in law to $F^\varphi_t(\bm{\overline\mu})$  (with convergence understood in both cases to be along the same subsequence as before). But \eqref{pf:QVbound1} implies that $F^\varphi_t(\bm{\overline\mu}^n) = M^{n,\varphi}_t$ converges in probability to zero. Hence,
\begin{align*}
F^\varphi_t(\bm{\overline\mu}) = 0, \ \text{almost surely, for each } t \in [0,T], \ \varphi \in C^\infty_c(\R^d).
\end{align*}
For each $\bm{\overline{m}} \in \P(\C^d\times\V)$, it is clear that $\lim_n F^{\varphi_n}_{t_n}(\bm{\overline{m}}) = F^{\varphi}_{t}(\bm{\overline{m}})$ whenever $t_n \rightarrow t$ and $(\varphi_n,\nabla\varphi_n,\Delta\varphi_n) \rightarrow (\varphi,\nabla\varphi,\Delta\varphi)$ uniformly. Hence, working with a countable dense family, we may interchange the order of quantifiers and conclude that
\begin{align*}
F^\varphi_t(\bm{\overline{\mu}}) = 0, \text{ for each } t \in [0,T], \ \varphi \in C^\infty_c(\R^d), \ \text{almost surely}.
\end{align*}
This shows that $\bm{\overline\mu}$ satisfies \eqref{ass:le:projection-semiMarkov}.

{\ } \\
\noindent\textbf{Step 2.} We now construct the processes $X$ and $W$. Thanks to Step 1, we may apply Lemma \ref{le:projection-semiMarkov} to find a semi-Markov function $\Lambda^* : [0,T] \times \R^d \times \CP \rightarrow A$ such that (a), (b), and (c) of Lemma \ref{le:projection-semiMarkov} hold.

{\ } \\
\noindent\textbf{Step 3.} 
To complete the proof, we address the final claim about convergence of value. Notice that
\begin{align*}
J_n &:= \frac{1}{n}\sum_{i=1}^nJ^n_i(\alpha^{n,1},\ldots,\alpha^{n,n}) \\
	&= \frac{1}{n}\sum_{i=1}^n\E\left[\int_0^T f(t,X^{n,i}_t,\mu^n_t,\alpha^{n,i}(t,\bm{X}^n))dt + g(X^i_T,\mu^n_T)\right] \\
	&= \E\left[\int_{\C^d \times \V}\left(\int_0^T\int_A f(t,x_t,\mu^n_t,a)q_t(da)dt + g(x_T,\mu^n_T)\right)\bm{\overline\mu}^n(dx,dq)\right].
\end{align*}
Recall that any subsequence contains a further subsequence along which $\bm{\overline\mu}^n$ converges to some $\bm{\overline\mu}$. Along such a subsequence, by boundedness and continuity of $f$ and $g$, we find that $J_n$ converges to
\begin{align*}
\E\left[\int_{\C^d \times \V}\left(\int_0^T\int_A f(t,x_t,\mu_t,a)q_t(da)dt + g(x_T,\mu_T)\right)\bm{\overline\mu}(dx,dq)\right].
\end{align*}
We claim that this is equal to the right-hand side of \eqref{pf:limitvalue1}.
Recalling that $(\mu_t)_{t \in [0,T]}$ is the marginal flow associated with $\bm{\overline\mu}$ and also that $\mu_t = \L(X_t \, | \, \F^\mu_t)$ for each $t$, we may write the second term as
\begin{align*}
\E\left[\int_{\C^d \times \V}g(x_T,\mu_T)\bm{\overline\mu}(dx,dq)\right] &= \E\left[\int_{\R^d}g(x,\mu_T)\,\mu_T(dx)\right] = \E[g(X_T,\mu_T)].
\end{align*}
To handle the first term, we use part (b) of Lemma \ref{le:projection-semiMarkov} along with Fubini's theorem and the identity $\mu_t = \L(X_t \, | \, \F^\mu_t)$ to write
\begin{align*}
\E&\left[\int_{\C^d \times \V}\int_0^T\int_A f(t,x_t,\mu_t,a)q_t(da)dt\,\bm{\overline\mu}(dx,dq)\right] \\
	&= \E\left[\int_0^T\int_{\R^d}\int_A f(t,x,\mu_t,a)\,\Lambda^*(t,x,\mu)(da)\,\mu_t(dx)\,dt\right] \\
	&= \E\left[\int_0^T\int_{\R^d}\int_A f(t,X_t,\mu_t,a)\,\Lambda^*(t,X_t,\mu)(da)\,dt\right].
\end{align*}
This completes the proof.
\end{proof}

\subsection{Optimality} \label{se:1playerdeviations}

The analysis carried out so far will allow us to check all of the properties of Definition \ref{def:MFEsemimarkov} at the limit except for the optimality condition (5), and this section will complete this last task.
Using Theorem \ref{th:limit-dynamics}, we work with a fixed weak limit $\mu=(\mu_t)_{t \in [0,T]}$, and we abuse notation by relabeling the subsequence with the same notation, so that $\mu^n=(\mu^n_t)_{t \in [0,T]} \rightarrow \mu$ weakly in $\CP$. It is crucial to keep in mind that \textbf{for the rest of this section we are working with this particlar limit point and this particular convergent subsequence}.

By Theorem \ref{th:limit-dynamics}, we may assume that $\mu$ is defined on a complete filtered probability space $(\Omega,\F,\FF,\PP)$, which supports $d$-dimensional processes $X$ and $W$ which satisfy properties (i-v) of Theorem \ref{th:limit-dynamics} and equation \eqref{pf:limitvalue1}, for some semi-Markov function $\Lambda^* : [0,T] \times \R^d \times \CP \rightarrow \P(A)$. 
Throughout this section, the notation $(\Omega,\F,\FF,\PP)$ of this paragraph will stand.

Relative to this fixed random measure flow $\mu$, we define on $(\Omega,\F,\FF,\PP)$ the family of all possible alternative strategy choices. Let us write $\RC_{\mathrm{semi}}$ for the set of semi-Markov functions from $[0,T] \times \R^d \times \CP$ to $\P(A)$. For any $\Lambda \in \RC_{\mathrm{semi}}$, let $X[\Lambda]=(X_t[\Lambda])_{t \in [0,T]}$ denote the unique strong solution (see Lemmas \ref{le:ap:SDErandom-uniqueness} and \ref{le:ap:SDErandom-existence}) of the SDE\footnote{It is not important here that we are working with strong solutions, but it is notationally convenient to construct everything on the same probability space $(\Omega,\F,\FF,\PP)$.}
\begin{align}
dX_t[\Lambda] = \int_Ab(t,X_t[\Lambda],\mu_t,a)\Lambda(t,X_t[\Lambda],\mu)(da)dt + dW_t, \quad X_0[\Lambda] = X_0. \label{pf:optimality-SDE}
\end{align}
In this notation, note that $X[\Lambda^*]=X$.
Define
\begin{align*}
J(\Lambda) &:= \E\left[\int_0^T\int_Af(t,X_t[\Lambda],\mu_t,a)\Lambda(t,X_t[\Lambda],\mu)(da)dt + g(X_T[\Lambda],\mu_T)\right].
\end{align*}

The proof of Theorem \ref{th:mainlimit-relaxed} will be complete if we can show that
\begin{align}
\sup_{\Lambda \in \RC_{\mathrm{semi}}}J(\Lambda) = J(\Lambda^*). \label{pf:optimality-goal1}
\end{align}
We accomplish this in two steps. The first and more straightforward step is to reduce the supremum to a nicer subset of $\RC_{\mathrm{semi}}$. Precisely, we will show
\begin{align}
\sup_{\beta \in \A^c_{\mathrm{semi}}}J(\beta) = \sup_{\Lambda \in \RC_{\mathrm{semi}}}J(\Lambda), \label{pf:optimality-goal1-2}
\end{align}
where we define $\A^c_{\mathrm{semi}}$ to be the set of \emph{continuous} semi-Markov functions $\beta : [0,T] \times \R^d \times \CP \rightarrow A$, which we view as a subset of $\RC_{\mathrm{semi}}$ by means of the usual embedding $A \ni a \mapsto \delta_a \in \P(A)$.  Indeed, \eqref{pf:optimality-goal1-2} follows from:

\begin{lemma} \label{le:approx-optimizer-nice}
For any $\Lambda \in \RC_{\mathrm{semi}}$, there exists a sequence $\beta^n \in \A^c_{\mathrm{semi}}$ such that $(\mu,X[\Lambda^n])$ converges in law in to $(\mu,X[\Lambda])$ and $J(\beta^n) \rightarrow J(\Lambda)$.
\end{lemma}

Lastly, for each ``nice" alternative control $\beta \in \A^c_{\mathrm{semi}}$, we show that $J(\beta)$ is the limit of the average value of some sequence of admissible $n$-player controls, which is accomplished using the following crucial proposition:

\begin{proposition} \label{pr:1playerdeviation}
Let $\beta \in \A^c_{\mathrm{semi}}$.
For each $n$ and each $k=1,\ldots,n$, define $\beta^{n,k} \in \A_n$ by
\begin{align}
\beta^{n,k}(t,\bm{x}) = \beta\Big(t,x^k_t,\frac{1}{n}\sum_{j=1}^n\delta_{x^j}\Big), \quad \text{for} \ \ \ t \in [0,T], \ \bm{x}=(x^1,\ldots,x^n) \in (\C^d)^n. \label{def:pr:1playerdev-specialformbeta}
\end{align}
Then (taking limits along the same subsequence described above)
\begin{align}
\lim_n \frac{1}{n}\sum_{k=1}^n J^{n}_k(\alpha^{n,1},\ldots,\alpha^{n,k-1},\beta^{n,k},\alpha^{n,k+1},\ldots,\alpha^{n,n}) = J(\beta). \label{goal:1playerdeviation}
\end{align}
\end{proposition}

With Proposition \ref{pr:1playerdeviation} in hand, let us see how to complete the proof of Theorem \ref{th:mainlimit-relaxed}: \\

\noindent\textbf{Proof of Theorem \ref{th:mainlimit-relaxed}.}
Using \eqref{pf:optimality-goal1-2}, for an arbitrary $\delta > 0$ we may find $\beta \in \A^c_{\mathrm{semi}}$ such that
\[
\sup_{\Lambda \in \RC_{\mathrm{semi}}}J(\Lambda)  \le J(\beta) + \delta.
\]
To prove \eqref{pf:optimality-goal1} it now suffices to show that $J(\beta) \le J(\Lambda^*)$.
Recall \eqref{pf:limitvalue1} from Theorem \ref{th:limit-dynamics}, which says
\begin{align*}
J(\Lambda^*) &= \lim_n\frac{1}{n}\sum_{i=1}^{n}J^{n}_i(\alpha^{n,1},\ldots,\alpha^{n,n}),
\end{align*}
where the limit is taken along the appropriate subsequence. 
On the other hand, defining $\beta^{n,k}$ as in Proposition \ref{pr:1playerdeviation}, we have \eqref{goal:1playerdeviation}. Finally using the fact that $(\alpha^{n,1},\ldots,\alpha^{n,n})$ is a closed-loop $\epsilon_n$-Nash equilibrium with $\epsilon_n\rightarrow 0$, we conclude that, along the same convergent subsequence,
\begin{align*}
J(\beta) &= \lim_n \frac{1}{n}\sum_{k=1}^n J^{n}_k(\alpha^{n,1},\ldots,\alpha^{n,k-1},\beta^{n,k},\alpha^{n,k+1},\ldots,\alpha^{n,n}) \\
	&\le \lim_n\frac{1}{n}\sum_{k=1}^{n}J^{n}_k(\alpha^{n,1},\ldots,\alpha^{n,n}) + \epsilon_n \\
	&= J(\Lambda^*).
\end{align*}
The proof of Theorem \ref{th:mainlimit-relaxed} is thus complete.
\hfill\qedsymbol \\

\begin{remark}
It is clear from the proof that we do not need the full strength of the $\epsilon_n$-Nash equilibrium property. In fact, it suffices to assume merely that $\bm{\alpha}^n=(\alpha^{n,1},\ldots,\alpha^{n,n})$ satisfies the much weaker inequality
\[
\frac{1}{n}\sum_{k=1}^{n}J^{n}_k(\alpha^{n,1},\ldots,\alpha^{n,n}) + \epsilon_n \ge \sup_{\beta^1,\ldots,\beta^n \in \A_n}\frac{1}{n}\sum_{k=1}^n J^{n}_k(\alpha^{n,1},\ldots,\alpha^{n,k-1},\beta^k,\alpha^{n,k+1},\ldots,\alpha^{n,n}).
\]
\end{remark}

\noindent\textbf{Proof of Lemma \ref{le:approx-optimizer-nice}.}
{\ } \\
\noindent\textbf{Step 1.} Before constructing the approximations, we show how to derive the claimed limits. 
Suppose $\Lambda^n \in \RC_{\mathrm{semi}}$, and assume that it holds for almost every $x \in \R^d$ and $\L(\mu)$-almost every $m \in \CP$ that, for every bounded continuous function $\varphi : [0,T] \times \R^d \times A \rightarrow \R$,
\[
\lim_{n\rightarrow\infty}\int_0^T\int_A\varphi(t,x,a) \Lambda^n(t,x,m)(da)dxdt = \int_0^T\int_A\varphi(t,x,a) \Lambda^n(t,x,m)(da)dt.
\]
Consider the coefficients
\[
B_n(t,x,m) := \int_Ab(t,x,m_t,a)\Lambda^n(t,x,m)(da), \quad B(t,x,m) = \int_A b(t,x,m_t,a)\Lambda(t,x,m)(da).
\]
For continuous $\varphi : [0,T] \times \R^d \rightarrow \R^d$ with compact support, we have
\begin{align*}
\lim_{n\rightarrow\infty}\int_0^T\int_{\R^d}\bigl(B_n(t,x,m) - B(t,x,m)\bigr) \cdot\varphi(t,x)dxdt = 0,
\end{align*}
for $\L(\mu)$-almost every $m \in \CP$.
Using Lemma \ref{le:ap:approximation}, we conclude that $(\mu,X[\Lambda^n])$ converges in law to $(\mu,X[\Lambda])$. To conclude that $J(\Lambda^n) \rightarrow J(\Lambda)$ we would like to simply use the fact that $f$ and $g$ are bounded and continuous, but we must be careful about the fact that $\Lambda$ and $\Lambda^n$ may be discontinuous.
Begin by writing 
\begin{align*}
J(\Lambda^n) &= \E\left[F_n(\mu,X[\Lambda^n])\right],
\end{align*}
where we define $F_n : \CP \times \C^d \rightarrow \R$ by
\[
F_n(m,x) = \int_0^T\int_Af(t,x_t,m_t,a)\Lambda^n(t,x_t,m)(da)dt + g(x_T,m_T).
\]
Define $F(x,m)$ similarly, with $\Lambda$ in place of $\Lambda^n$, so that $J(\Lambda) = \E[F(\mu,X[\Lambda])]$. We know from Lemma \ref{le:ap:approximation} that $\E[h(\mu,X[\Lambda^n])] \rightarrow \E[h(\mu,X[\Lambda])]$ for every bounded measurable function $h : \CP \times \C^d \rightarrow \R$. On the other hand, we know by assumption that $F_n \rightarrow F$ pointwise. We may use a form of the dominated convergence theorem \cite[Proposition 11.4.18]{royden-analysis} to conclude that  $\E\left[F_n(\mu,X[\Lambda^n])\right] \rightarrow \E[F(\mu,X[\Lambda])]$.
{\ } \\

\noindent\textbf{Step 2.} Next, we construct the desired approximations.
Apply the well known ``chattering lemma" (see, e.g., \cite[Theorem 2.2]{elkaroui-partialobservations} or \cite[Theorem 4]{fleming1984stochastic}) to find a sequence  of semi-Markov functions $\beta^n : [0,T] \times \R^d \times \CP \rightarrow A$ such that 
\[
dt\delta_{\beta^n(t,x,m)}(da) \rightarrow dt\Lambda(t,x,m)(da)
\]
weakly (i.e., in $\V$) for each $(x,m)$. Hence, we may assume $\Lambda$ is already of the form $\Lambda(t,x,m) = \delta_{\beta(t,x,m)}$ for some semi-Markov function $\beta$.

To complete the proof we use the fact that, since $A$ is compact and convex, any measurable function from a Polish probability space into $A$ is the a.e.\ limit of continuous functions (see, e.g., \cite[Proposition C.1]{carmona-delarue-lacker}). By ``Polish probability space" we mean a Polish space $E$ equipped with a Borel probability measure. The only hurdle is that the Borel $\sigma$-field of the space $\Theta := [0,T] \times \R^d \times \CP$ is strictly larger than the one generated by semi-Markov functions, but this is not difficult to work around. Equip $\Theta$ with the probability measure $Q$ defined for Borel sets $S \subset \Theta$ by
\[
Q(S) = \frac{1}{T}\E\left[\int_0^T\int_{\R^d} 1_S(t,x,\mu_t)\Phi_d(x)\,dx\,dt \right],
\]
where $\Phi_d$ is the density of a standard $d$-dimensional Gaussian random variable. Define the map $\Pi : \Theta \rightarrow \Theta$ by
\begin{align*}
\Pi(t,x,m) = (t,x,m_{\cdot \wedge t}),
\end{align*}
where $m_{\cdot \wedge t}$ denotes the path which follows $m$ up to time $t$ and is constant thereafter.
Then $\Pi$ is continuous, and the image $\Pi(\Theta)$ is closed. Moreover, the $\sigma$-field generated by $\Theta$ is precisely the one generated by the semi-Markov functions, and so any semi-Markov function $F : \Theta \rightarrow A$ factorizes through $\Pi$, in the sense that $F = F \circ \Pi$. The space $\Pi(\Theta)$ is a Polish space with the induced topology. Hence, as mentioend above,  $\beta=\beta \circ \Pi$ is the $Q \circ \Pi^{-1}$-a.e.\ limit of a sequence of continuous functions $\tilde\beta^n : \Pi(\Theta) \rightarrow A$. Define $\beta^n : \Theta \rightarrow A$ by $\beta^n = \tilde\beta^n \circ \Pi$. Then, $\beta^n$ is continuous for each $n$, and $\beta^n \rightarrow \beta$ holds $Q$-a.e.
\hfill\qedsymbol

{\ } \\

\noindent\textbf{Proof of Proposition \ref{pr:1playerdeviation}.}
Recall that $\bm{X}^n$ solves the SDE \eqref{def:stateprocesses-limitsection}.
Define the state process
\[
\bm{Y}^{n,k} = (Y^{n,k,1},\ldots,Y^{n,k,n}) := \bm{X}^n[(\alpha^{n,1},\ldots,\alpha^{n,k-1},\beta^{n,k},\alpha^{n,k+1},\ldots,\alpha^{n,n})]
\]
Note that $\bm{Y}^{n,k}$ follows the dynamics
\begin{align*}
dY^{n,k,k}_t &=  b (t,Y^{n,k,k}_t,\mu^{n,k}_t, \beta(t,\bm{Y}^{n,k,k}_t,\mu^{n,k}))dt + dW^k_t, \\
dY^{n,k,i}_t &= b(t,Y^{n,k,i}_t,\mu^{n,k}_t,\alpha^{n,i}(t,\bm{Y}^{n,k}))dt + dW^i_t, \quad i \neq k, \\
\mu^{n,k}_t &= \frac{1}{n}\sum_{j=1}^n\delta_{Y^{n,k,j}}.
\end{align*}
Assume that $\bm{X}^n$ is defined on some filtered probability space $(\Omega^n,\F^n,\FF^n,\PP^n)$, and we of course assume that the Brownian motions $W^k$ from \eqref{def:stateprocesses-limitsection} are in fact $\FF^n$-Brownian motions. Note $\bm{Y}^n$ may live on a different probability space which, to avoid complicating notation, we will not give a name.
Recall from the second paragraph of Section \ref{se:1playerdeviations} that we are working throughout this proof with a given (relabeled) subsequence along which $\L(\mu^n)=\PP^n \circ (\mu^n)^{-1}$ converges in $\P(\CP)$ to $\L(\mu)$.

{\ } \\
\noindent\textbf{Step 1.}
It is convenient in this proof to work on a suitable canonical space, and the first step is simply to set up notation. Define an equivalent probability measure $\QQ^{n,k}$ on $(\Omega^n,\F^n,\FF^n)$ by setting
\begin{align*}
\frac{d\QQ^{n,k}}{d\PP^n} &= \exp\Bigg( \int_0^T  \Big(b(t,X^{n,k}_t,\mu^n_t,\beta(t,X^{n,k}_t,\mu^n)) - b(t,X^{n,k}_t,\mu^n_t,\alpha^{n,k}(t,\bm{X}^n)) \Big) \cdot dW^k_t \\
	&\quad\quad\quad - \frac12\int_0^T\Big|b(t,X^{n,k}_t,\mu^n_t,\beta(t,X^{n,k}_t,\mu^n)) - b(t,X^{n,k}_t,\mu^n_t,\alpha^{n,k}(t,\bm{X}^n)) \Big|^2dt \Bigg).
\end{align*}
By Girsanov's theorem and uniqueness of the SDEs, we have $\QQ^{n,k} \circ (\bm{X}^n)^{-1} = \L(\bm{Y}^{n,k})$, and thus $\QQ^{n,k} \circ (\mu^n,\bm{X}^n)^{-1} = \L(\mu^{n,k},\bm{Y}^{n,k})$. Note also that we may write $d\QQ^{n,k} / d\PP^n = \zeta^{n,k}_T$, where we define $\zeta^{n,k}$ as the unique solution of the SDE
\begin{align}
d\zeta^{n,k}_t &= \zeta^{n,k}_t \Xi^{n,k}_t \cdot dW^k_t, \quad \ \zeta^{n,k}_0=1, \nonumber \\ 
\Xi^{n,k}_t &:= b(t,X^{n,k}_t,\mu^n_t,\beta(t,X^{n,k}_t,\mu^n)) - b(t,X^{n,k}_t,\mu^n_t,\alpha^{n,k}(t,\bm{X}^n)) \label{pf:def:Xink}
\end{align}
We note for future use that boundedness of $b$ easily yields the estimate
\begin{align}
\sup_{n \in \N}\max_{k=1,\ldots,n}\E\left[\left|d\QQ^{n,k} / d\PP^n \right|^p\right] = \sup_{n \in \N}\max_{k=1,\ldots,n}\E\left[
|\zeta^{n,k}_T|^p\right] < \infty, \label{pf:optimality-RNbound1}
\end{align}
for any $p > 0$. Moreover, we may write
\begin{align}
\frac{1}{n}\sum_{k=1}^n & J^{n}_k(\alpha^{n,1},\ldots,\alpha^{n,k-1},\beta^{n,k},\alpha^{n,k+1},\ldots,\alpha^{n,n}) \nonumber \\
	&= \frac{1}{n}\sum_{k=1}^n \E\left[ \int_0^T f(t,Y^{n,k,k}_t,\mu^{n,k}_t, \beta(t,Y^{n,k,k}_t,\mu^{n,k}))dt + g(Y^{n,k,k}_T,\mu^{n,k}_T)\right]  \label{pf:optimality-Jexpression1-0} \\
	&= \frac{1}{n}\sum_{k=1}^n \E^{\QQ^{n,k}}\left[ \int_0^T f(t,X^{n,k}_t,\mu^n_t, \beta(t,X^{n,k}_t,\mu^n))dt + g(X^{n,k}_T,\mu^n_T)\right]   \nonumber \\
	&= \E^{\PP^n}\left[ \frac{1}{n}\sum_{k=1}^n \zeta^{n,k}_T \left(\int_0^T f(t,X^{n,k}_t,\mu^n_t, \beta(t,X^{n,k}_t,\mu^n))dt + g(X^{n,k}_T,\mu^n_T) \right)\right]. \label{pf:optimality-Jexpression1}
\end{align}
We would like to show that the measure $\frac{1}{n}\sum_{k=1}^n\L(Y^{n,k,k},\mu^{n,k})$ converges to $\L(X[\beta],\mu)$, along the same subsequence for which $\L(\mu^n)$ converges to $\L(\mu)$. Indeed, we could then pass to the limit directly in \eqref{pf:optimality-Jexpression1-0}. The change of measure allows us to transform the expression into one involving the original $\mu^n$ and the particles $X^{n,k}$, as well as the new auxiliary particles $\zeta^{n,k}$.
We will ultimately analyze the limiting behavior of the empirical measure of $(X^{n,k},\zeta^{n,k},W^k)_{k=1}^n$, as it is convenient to include the Brownian motion $W^k$ as well.

Precisely, we proceed as follows. Define the $\P(A)$-valued processes $\Lambda^{n,k}_t =  \delta_{\alpha^{n,k}(t,\bm{X}^n)}$, and view $\Lambda^{n,k}$ as a $\V$-valued random variable.
Consider the extended empirical measure
\begin{align*}
\bm{R}^n := \frac{1}{n}\sum_{k=1}^n \delta_{(X^{n,k},\zeta^{n,k},W^k,\Lambda^{n,k})},
\end{align*}
viewed as a random variable with values in $\P(\overline\Omega)$, where $\overline\Omega := \C^d \times \C^1_+ \times \C^d \times \V$. Here, $\C^1_+ := C([0,T];\R_+)$ is the space of nonnegative one-dimensional continuous paths.

{\ } \\

\noindent\textbf{Step 2.} 
We first show that the sequence $\{\PP^n \circ (\bm{R}^n)^{-1} : n \in \N\} \subset \P(\P(\overline\Omega))$ is tight. According to \cite[(2.5)]{sznitman1991topics}, it suffices to show that the sequence $\{M_n : n \in \N\} \subset \P(\overline\Omega)$ of mean measures is tight, where the mean measure $M_n$ is defined on Borel sets $S \subset \overline\Omega$ by
\begin{align}
M_n(S) = \E^{\PP^n}[\bm{R}^n(S)] = \frac{1}{n}\sum_{k=1}^n\PP^n\big((X^{n,k},\zeta^{n,k},W^k,\Lambda^{n,k}) \in S\big). \label{pf:optimality-meanmeasures}
\end{align}
To do this, it suffices to show that each marginal sequence is tight. Since $\V$ is compact, the $ \V$-marginal sequence is clearly tight. The third marginal of $M_n$ is precisely Wiener measure; this sequence is constant and therefore tight.
We saw in the proof of Lemma \ref{le:tightness-main} that the sequence of first marginals
\[
\frac{1}{n}\sum_{k=1}^n\PP^n \circ (X^{n,k})^{-1}
\]
is tight. Finally, we must check that the second marginal sequence
\[
\frac{1}{n}\sum_{k=1}^n\PP^n \circ (\zeta^{n,k})^{-1}
\]
is tight.
This is accomplished using Aldous' criterion for tightness \cite[Lemma 16.12]{kallenberg-foundations}. First, note that the estimate \eqref{pf:optimality-RNbound1} implies by Doob's inequality
\begin{align}
\sup_{n \in \N}\max_{k=1,\ldots,n}\E^{\PP^n}\left[\sup_{t \in [0,T]}|\zeta^{n,k}_t|^p\right] < \infty, \label{pf:optimality-RNbound2}
\end{align}
Recalling that $b$ is uniformly bounded, we have $|\Xi^{n,k}| \le 2\|b\|_\infty$, where we recall the notation $\Xi^{n,k}$ from \eqref{pf:def:Xink}.
For any $\delta > 0$ and any $[0,T-\delta]$-valued stopping time, It\^o's isometry yields
\begin{align*}
\E^{\PP^n} \left[|\zeta^{n,k}_{\tau+\delta}-\zeta^{n,k}_{\tau}|^2\right] &= \E^{\PP^n}\left[\left|\int_{\tau}^{\tau+\delta}\zeta^{n,k}_t\Xi^{n,k}_t \cdot dW^k_t\right|^2\right] = \E^{\PP^n}\left[\int_{\tau}^{\tau+\delta}|\zeta^k_t|^2 |\Xi^{n,k}_t|^2dt\right] \\
	&\le 4\delta \|b\|_\infty^2\E^{\PP^n}\left[\sup_{t \in [0,T]}|\zeta^k_t|^2\right].
\end{align*}
This converges to zero as $\delta \rightarrow 0$, uniformly in $n$, $k$, and $\tau$. This is enough to apply Aldous' criterion and conclude that the second marginal sequence of $M_n$ is tight, thus completing the proof that $\bm{R}^n$ is a tight sequence of $\P(\overline\Omega)$-valued random variables.

{\ } \\
\noindent\textbf{Step 3.} 
As a first step toward identifying the limit points of $\bm{R}^n$, by first showing that all limit points are supported on the set of solutions of a certain martingale problem. For the moment, fix $n \in \N$ and $k \in \{1,\ldots,n\}$.
For any $\varphi=\varphi(x,y,w) \in C^\infty_c(\R^d \times \R_+ \times \R^d)$, It\^o's formula yields
\begin{align*}
d\varphi(X^{n,k}_t,\zeta^{n,k}_t,W^k_t) &= \nabla_x\varphi(X^{n,k}_t,\zeta^{n,k}_t,W^k_t) \cdot b(t,X^{n,k}_t,\mu^n_t,\alpha^{n,k}(s,\bm{X}^n))dt + \frac12\Delta_x\varphi(X^{n,k}_t,\zeta^{n,k}_t,W^k_t)dt \\
	&\quad + \frac12\partial_{yy}\varphi(X^{n,k}_t,\zeta^{n,k}_t,W^k_t)|\zeta^{n,k}_t|^2|\Xi^{n,k}_t|^2dt + \frac12\Delta_w\varphi(X^{n,k}_t,\zeta^{n,k}_t,W^k_t)dt \\
	&\quad + \zeta^{n,k}_t(\nabla_x+\nabla_w)\partial_y\varphi(X^{n,k}_t,\zeta^{n,k}_t,W^k_t)\cdot \Xi^{n,k}_t dt \\ 
	&\quad + (\nabla_w\cdot \nabla_x)\varphi(X^{n,k}_t,\zeta^{n,k}_t,W^k_t) dt + \partial_y\varphi(X^{n,k}_t,\zeta^{n,k}_t,W^k_t)\zeta^{n,k}_t\Xi^{n,k}_t \cdot dW^k_t \\
	&\quad + \nabla_x\varphi(X^{n,k}_t,\zeta^{n,k}_t,W^k_t) \cdot dW^k_t + \nabla_w\varphi(X^{n,k}_t,\zeta^{n,k}_t,W^k_t) \cdot dW^k_t
\end{align*}
Here we write $(\nabla_w\cdot \nabla_x)$ for the operator $\sum_{i=1}^d\partial_{w_i}\partial_{x_i}$.
For $m \in \CP$, $t \in [0,T]$, and $\varphi \in C^\infty_c(\R^d \times \R_+ \times \R^d)$, 
define a random variable $M_t[m,\varphi] : \overline\Omega \rightarrow \R$ by
\begin{align*}
M_t[m,\varphi](x,y,w,q) &= \varphi(x_t,y_t,w_t) - \int_0^t\int_A \widehat{M}[m,\varphi](u,x_u,y_u,w_u,a)q_u(da)du,
\end{align*}
where, for $(t,x,y,w,a) \in [0,T] \times \R^d \times \R_+ \times \R^d \times A$, we define
\begin{align*}
\widehat{M}[m,\varphi](t,x,y,w,a) &=  \nabla_x\varphi(x,y,w) \cdot b(t,x,m_t,a) + \frac12\Delta_x\varphi(x,y,w) \\
	&\quad + \frac12\partial_{yy}\varphi(x,y,w)|y|^2\left|b(t,x,m_t,\beta(t,x,m)) - b(t,x,m_t,a)\right|^2 \\
	&\quad + y(\nabla_x+\nabla_w)\partial_y\varphi(x,y,w)\cdot \left[b(t,x,m_t,\beta(t,x,m)) - b(t,x,m_t,a )\right] \\
	&\quad + \frac12\Delta_w\varphi(x,y,w) + (\nabla_w\cdot \nabla_x)\varphi(x,y,w) .
\end{align*}
Under $\PP^n$, the above calculation shows that the process
\begin{align*}
M^{n,k,\varphi}_t := M_t[\mu^n,\varphi](X^{n,k},\zeta^{n,k},W^k,\Lambda^{n,k})
\end{align*}
is a martingale. Moreover, the cross-variation $[M^{n,k,\varphi},M^{n,j,\varphi}]$ vanishes for $j \neq k$.

To completely specify a martingale problem, we equip $\overline{\Omega}$ with a canonical filtration $\overline\FF=(\overline\F_t)_{t \in [0,T]}$. Precisely, this is defined by letting $\overline\F_t$ be the $\sigma$-field generated by the maps $\overline{\Omega} \ni (x,y,w,q) \mapsto (x_s,y_s,w_s,\widehat{q}(s,\cdot)) \in \R^d \times \R_+ \times \R^d \times \P(A)$, for $s \le t$, where $\widehat{q}$ is the version of the canonical $\P(A)$-valued process on $\V$ described in \eqref{def:predictablecanonical-V}.

For $s < t$ and any continuous $\overline\F_s$-measurable function $h : \overline\Omega \rightarrow \R$ bounded in absolute value by $1$, define $F[h,\varphi,s,t] : \P(\overline\Omega) \rightarrow \R$ by
\[
F[h,\varphi,s,t](\bm{R}) = |\langle \bm{R}, (M_t[R^x,\varphi] - M_s[R^x,\varphi])h\rangle|^2,
\]
where, for $\bm{R} \in \P(\overline\Omega) = \P(\C^d \times \C^1_+ \times \C^d \times \V)$, we write $R^x$ to denote the induced measure flow $R^x=(R^x_t)_{t \in [0,T]} \in \CP$ induced by the first $\C^d$-marginal of $\bm{R}$. That is, $R^x_t = \bm{R} \circ [(x,y,w,q) \mapsto x_t]^{-1}$. Because $b$ and $\beta$ are continuous by assumption, the map
\begin{align*}
\CP \times \overline\Omega \ni (m,x,y,w,q) \mapsto \int_0^t\int_A \widehat{M}[m,\varphi](u,x_u,y_u,w_u,a)q_u(da)du
\end{align*}
is continuous for each $t$ and $\varphi$ (see, e.g., \cite[Appendix A]{lacker2015mean} for details). We would immediately deduce that $F[h,\varphi,s,t]$ is continuous on $\P(\overline\Omega)$, except that $\widehat{M}$ is unbounded due to the multiplication by $|y|^2$. 
To deal with this, abbreviate $F=F[h,\varphi,s,t]$, and define for $r > 0$
\[
F^r[h,\varphi,s,t](\bm{R}) = |\langle \bm{R}, (M^r_t[R^x,\varphi] - M^r_s[R^x,\varphi])h\rangle|^2,
\]
where $M^r_t[m,\varphi](x,y,w,q) := M_t[m,\varphi](x,y \wedge r,w,q)$. Then $M^r[m,\varphi]$ is uniformly bounded for each $r$ and $\varphi$. Using \eqref{pf:optimality-RNbound2}, it is straightforward to check that
\begin{align}
\lim_{r\rightarrow\infty}\sup_{n \in \N}\E^{\PP^n}\left[|F[h,\varphi,s,t](\bm{R}^n) - F^r[h,\varphi,s,t](\bm{R}^n)|^2\right] = 0. \label{pf:optimality-boundedness1}
\end{align}
Note that $F^r[h,\varphi,s,t]$ is bounded and continuous on $\P(\overline\Omega)$. 

Now, recalling that the sequence $\bm{R}^n$ is tight by Step 2, we may suppose that it converges in law (along a subsequence) to some $\P(\overline\Omega)$-valued random variable $\bm{R}$. Use \eqref{pf:optimality-boundedness1} to conclude that $\PP^n \circ (F[h,\varphi,s,t](\bm{R}^n))^{-1}$ converges to $\L(F[h,\varphi,s,t](\bm{R}))$. Then, using Fatou's lemma, the fact that $M^{n,k,\varphi}$ and $M^{n,j,\varphi}$ define orthogonal martingales, and $|h| \le 1$, we find (taking limits along the same subsequence)
\begin{align*}
\E\left[F[h,\varphi,s,t](\bm{R})\right] &\le \liminf_n\E^{\PP^n}\left[F[h,\varphi,s,t](\bm{R}^n)\right] \\
	&= \liminf_n\E^{\PP^n}\left[\left|\frac{1}{n}\sum_{k=1}^n(M^{n,k,\varphi}_t - M^{n,k,\varphi}_s)h(X^{n,k},\zeta^{n,k},W^k,\Lambda^{n,k})\right|^2 \right] \\
	&= \liminf_n\frac{1}{n^2}\sum_{k=1}^n\E^{\PP^n}\left[\left|M^{n,k,\varphi}_t - M^{n,k,\varphi}_s\right|^2 h^2(X^{n,k},\zeta^{n,k},W^k,\Lambda^{n,k}) \right] \\
	&\le \liminf_n\frac{1}{n^2}\sum_{k=1}^n\E^{\PP^n}\left[\left|M^{n,k,\varphi}_t - M^{n,k,\varphi}_s\right|^2 \right].
\end{align*}
Finally, noting that
\begin{align*}
&\E^{\PP^n}\left[\left|M^{n,k,\varphi}_t - M^{n,k,\varphi}_s\right|^2 \right] \\
	&= \E^{\PP^n}\int_s^t\left|(\nabla_w + \nabla_x)\varphi(X^{n,k}_u,\zeta^{n,k}_u,W^k_u) + \partial_y\varphi(X^{n,k}_u,\zeta^{n,k}_u,W^k_u)\zeta^{n,k}_u\Xi^k_u\right|^2du,
\end{align*}
we use \eqref{pf:optimality-RNbound2} and boundedness of $b$ to get
\begin{align*}
\sup_{n \in \N} \max_{k=1,\ldots,n}\E^{\PP^n}\left[\left|M^{n,k,\varphi}_t - M^{n,k,\varphi}_s\right|^2 \right] < \infty.
\end{align*}
Hence,
\begin{align*}
\E\left[F[h,\varphi,s,t](\bm{R})\right] = 0.
\end{align*}
In  particular, $F[h,\varphi,s,t](\bm{R})=0$ a.s.\ for each $h,\varphi,s,t$.

By working with a countably dense family (as in the end of Step 1 of the proof of Theorem \ref{th:limit-dynamics}), we may switch the order of quantifiers to conclude that it holds with probability $1$ that, for all $(h,\varphi,s,t)$, $F[h,\varphi,s,t](\bm{R})=0$. Recalling the definition of $F[h,\varphi,s,t]$, this means that $\bm{R}$ is supported on the set $\mathfrak{L} \subset \P(\overline\Omega)$ consisting of those probability measures $R$ such that:
\begin{itemize}
\item  $(M_t[R^x,\varphi])_{t \in [0,T]}$ is an $R$-martingale, for each $\varphi \in C^\infty_c(\R^d \times \R_+ \times \R^d)$, where $R^x=(R^x_t)_{t \in [0,T]} \in \CP$ denotes the measure flow associated with the first marginal.
\item $R \circ [(x,y,w,q) \mapsto (x_0,y_0,w_0)]^{-1} = \lambda \times \delta_1 \times \delta_0$.
\end{itemize}

{\ } \\
\noindent\textbf{Step 4.} 
We now establish a key identity satisfied by the measures $R \in \mathfrak{L}$ identified in the previous step. For $m \in \CP$ let $P^m \in \P(\C^d)$ denote the law of the unique solution $X^m$ of the SDE
\[
dX^m_t = b(t,X^m_t,m_t,\beta(t,X^m_t,m)) dt + dB_t, \quad X^m_0 \sim \lambda.
\]
This defines a universally measurable map $\CP \ni m \mapsto P^m \in \P(\C^d)$, by Lemma \ref{le:Pe-adapted}. We claim that every $R \in \mathfrak{L}$ satisfies
\begin{align}
\int_{\overline\Omega}h(x)y_T \, R(dx,dy,dw,dq) = \langle P^{R^x},h \rangle, \label{pf:optimality-keyidentity1}
\end{align}
for bounded measurable functions $h$ on $\C^d$.

Fix $R \in \mathfrak{L}$. We can construct, on some filtered probability space $(\Omega,\F,\FF,\PP)$, an $\overline\Omega$-valued random variable $(X,\zeta,W,\Lambda)$ with law $R$ such that the process
\begin{align*}
\varphi(X_t,\zeta_t,W_t) - \int_0^t \int_A &\Bigg[\nabla_x\varphi(X_s,\zeta_s,W_s) \cdot b(s,X_s,R^x_s,a) + \frac12\Delta_x\varphi(X_s,\zeta_s,W_s) \\
	&\quad + \frac12\partial_{yy}\varphi(X_s,\zeta_s,W_s)|y|^2\left|b(s,X_s,R^x_s,\beta(s,X_s,R^x)) - b(s,X_s,R^x_s,a)\right|^2 \\ 
	&\quad + y(\nabla_x+\nabla_w)\partial_y\varphi(X_s,\zeta_s,W_s)\cdot \left[b(t,X_s,R^x_s,\beta(s,X_s,R^x) ) - b(s,X_s,R^x_s,a)\right] \\
	&\quad + \frac12\Delta_w\varphi(X_s,\zeta_s,W_s) + (\nabla_w \cdot \nabla_x)\varphi(X_s,\zeta_s,W_s)\Bigg]\Lambda_s(da)ds
\end{align*}
is a martingale for each $\varphi \in C_c^\infty(\R^d \times \R_+ \times \R^d)$.
This is a bit different from the usual martingale problem framework because of the integration with respect to $\Lambda_s(da)$, so standard theory does not immediately tell us how to represent $(X,\zeta,W)$ as the solution of an SDE. But the work of  El Karoui and M\'el\'eard \cite{elkarouimeleard-martingalemeasure} covers this situation by making use of the notion of \emph{martingale measures}, in the sense of Walsh \cite{walsh1986introduction}, and the reader is referred to either reference for precise definitions. According to \cite[Theorem IV-2]{elkarouimeleard-martingalemeasure}, by extending the probability space if needed, we may find a vector $M=(M^1,\ldots,M^d)$ of orthogonal martingale measures $M^i=M^i(da,dt)$ on $A \times [0,T]$, each with intensity measure $\Lambda_t(da )dt$, such that the following hold, for $t \in [0,T]$:
\begin{align*}
dX_t &= \int_A b(t,X_t,R^x_t,a)\Lambda_t(da)dt + dW_t \\
dW_t &= \int_A M(da,dt), \quad\quad \text{i.e., } \quad\quad W_t = \int_{A \times [0,t]} M(da,ds) = M(A \times [0,t]), \\
d\zeta_t &= \zeta_t dN_t,
\end{align*}
where we define the martingale $N$ by
\begin{align*}
N_t &= \int_{A \times [0,t]} \big(b(s,X_s,R^x_s,\beta(s,X_s,R^x) ) - b(s,X_s,R^x_s,a)\big) \cdot M(da,ds).
\end{align*}
The only fact we need to know about martingale measures in the following: For any bounded jointly functions $\varphi,\psi : [0,T] \times A \times \Omega \rightarrow \R^d$ (using the Borel $\sigma$-field on $A$ and the $\FF$-progressive $\sigma$-field on $[0,T] \times \Omega$), the processes $t \mapsto \int_{A \times [0,t]}\varphi(s,a) \cdot M(da,ds)$ and $t \mapsto \int_{A \times [0,t]}\psi(s,a) \cdot M(da,ds)$ are orthogonal martingales with covariation process $\int_0^t\int_A\varphi(s,a)\cdot \psi(s,a)\Lambda_s(da)ds$. In particular, using this and L\'evy's characterization, we deduce that $W$ is a Brownian motion.

Continuing to work on the same probability space $(\Omega,\F,\FF,\PP)$, define a change of measure by
\begin{align*}
\frac{d\QQ}{d\PP} := \zeta_T = \exp(N_T - \tfrac12[N]_T).
\end{align*}
By Girsanov's theorem (e.g., in the general form of \cite[Theorem III.39]{protter2005stochastic}), the process $B = W - [W,N]$ is a $\QQ$-Brownian motion, and we compute
\begin{align*}
B_t &= W_t - \int_0^t\int_A \big(b(s,X_s,R^x_s,\beta(s,X_s,R^x)) - b(s,X_s,R^x_s,a) \big) \Lambda_s(da)ds.
\end{align*}
Substitute this into the equation for $X$ to get
\begin{align}
dX_t &= b(t,X_t,R^x_t,\beta(t,X_t,R^x)) dt + dB_t, \label{pf:limitanalysis-branchpoint}
\end{align}
still with initial distribution $\QQ \circ X_0^{-1} = \PP \circ X_0^{-1} = \lambda$.

The SDE \eqref{pf:limitanalysis-branchpoint} has a unique in law solution, and its law is precisely $\QQ \circ X^{-1} = P^{R^x}$, where $P^m$ was defined for $m \in \CP$ at the beginning of this step. It then holds, for any bounded measurable $h : \C^d \rightarrow \R$, that
\begin{align*}
\int_{\overline\Omega}h(x)y_T \, R(dx,dy,dw,dq) = \E^{\PP}[h(X)\zeta_T] = \E^{\QQ}[h(X)] = \langle P^{R^x},h \rangle,
\end{align*}
which establishes \eqref{pf:optimality-keyidentity1}.

{\ } \\
\noindent\textbf{Step 5.} 
We are finally ready to take limits. Recalling from Step 2 that $\bm{R}^n$ is a tight sequence, let $\bm{R}$ denote any weak limit. From Step 3 we know that $\bm{R}$ belongs almost surely to  $\mathfrak{L}$.
Recalling the identifications of Step 1, we may pass to the limit along the same subsequence along which $\bm{R}^n$ converges in law to $\bm{R}$ to get, using \eqref{pf:optimality-keyidentity1},
\begin{align*}
\lim\frac{1}{n}\sum_{k=1}^n\E^{\PP^n}[\zeta^{n,k}_T h(\mu^n,X^{n,k})] &= \lim\E^{\PP^n}\left[ \int_{\overline\Omega}y_Th(\mu^n,x) \, \bm{R}^n(dx,dy,dw,dq) \right] \\
	&= \E\left[\int_{\overline\Omega}y_Th(R^x,x) \, \bm{R}(dx,dy,dw,dq)\right] \\
	&= \E\left[ \langle P^{R^x},h(R^x,\cdot) \rangle \right],
\end{align*}
for any bounded continuous function $h$ on $\CP \times \C^d$.
Recall that $\mu^n$ converges in law to $\mu$, which implies that $R^x \stackrel{d}{=} \mu$. Hence, 
\begin{align*}
\lim_{n\rightarrow\infty}\frac{1}{n}\sum_{k=1}^n\E^{\PP^n}[\zeta^{n,k}_Th(\mu^n,X^{n,k})] &=  \E\left[\langle P^{\mu},h(\mu,\cdot) \rangle \right],
\end{align*}
Recalling the notation from before the statement the Proposition, the process $X[\beta]$ solves the SDE 
\[
dX_t[\beta] = b(t,X_t[\beta],\mu_t,\beta(t,X_t[\beta],\mu))dt + dW_t,
\]
where $X_0 \sim \lambda$, $W$, and $\mu$ are independent. 
Lemma \ref{le:ap:SDErandom-uniqueness} ensures that the conditional law of $X[\beta]$ given $\mu$ is precisely $P^\mu$.
In particular, $\E\left[\langle P^{\mu},h(\mu,\cdot) \rangle \right] = \E\left[h(\mu,X[\beta])\right]$, and we have
\begin{align*}
\lim_{n\rightarrow\infty}\frac{1}{n}\sum_{k=1}^n\E^{\PP^n}[\zeta^{n,k}_Th(\mu^n,X^{n,k})] &=  \E\left[h(\mu,X[\beta]) \right].
\end{align*}
Finally, recalling that $f$, $g$, and $\beta$ are continuous by assumption, we may finally return to \eqref{pf:optimality-Jexpression1} from Step 1 to complete the proof:
\begin{align*}
\frac{1}{n}\sum_{k=1}^n & J^{n}_k(\alpha^{n,1},\ldots,\alpha^{n,k-1},\beta^{n,k},\alpha^{n,k+1},\ldots,\alpha^{n,n}) \\
	&= \E^{\PP^n}\left[ \frac{1}{n}\sum_{k=1}^n \zeta^{n,k}_T \left(\int_0^T f(t,X^{n,k}_t,\mu^n_t, \beta(t,X^{n,k}_t,\mu^n))dt + g(X^{n,k}_T,\mu^n_T) \right)\right] \\
	&= \E\left[\int_0^T f(t,X_t[\beta],\mu_t, \beta(t,X_t[\beta],\mu))dt + g(X_T[\beta],\mu_T)\right] \\
	&= J(\beta).
\end{align*}
\hfill\qedsymbol

\section{Closed-loop versus open-loop} \label{se:closedloopvsopenloop-proofs}

This section compares our notion of weak semi-Markov RMFE (Definition \ref{def:MFEsemimarkov}) with the notion of \emph{weak MFG solution} of \cite[Definition 3.1]{lacker2016general}, which itself is a specialization of \cite[Definition 3.1]{carmona-delarue-lacker} to the case without common noise. The relevance of the latter definition is that it characterizes the limits of $n$-player approximate equilibria in open-loop regime \cite[Theorem 3.4]{lacker2016general}. Our goal is to show that these two definitions are largely equivalent. To state the definition of a weak MFG solution, we first need a bit of notation.

Recall from Section \ref{se:relaxedcontrols} the definition of the space $\V$ of relaxed controls. Define $\X := \C^d \times \V \times \C^d$, and equip this space with the filtration $\FF^\X=(\F^\X_t)_{t \in [0,T]}$, where $\F^\X_t$ is the $\sigma$-field generated by the maps $\X \ni (w,q,x) \mapsto (w_s,q(S),x_s) \in \R^d \times \R \times \R^d$, where $s \le t$ and $S$ is a Borel subset of $[0,t] \times A$. As usual, we identify a $\P(A)$-valued process $\Lambda=(\Lambda_t)_{t \in [0,T]}$ with the random element of $\V$ given by $dt\Lambda_t(da)$. For a measure $\bm{\widetilde{m}} \in \P(\X)$, we write $\widetilde{m}^x=(\widetilde{m}^x_t)_{t \in [0,T]} \in \CP$ for the measure flow associated with the third marginal, i.e., $\widetilde{m}^x_t = \bm{\widetilde{m}} \circ [(w,q,x) \mapsto x_t]^{-1}$.

\begin{definition} \label{def:weakMFGsolution}
A \emph{weak MFG solution} is a tuple $(\Omega,\F,\FF,\PP,W,\bm{\widetilde\mu},\Lambda,X)$, where:
\begin{enumerate}
\item $(\Omega,\F,\FF,\PP)$ is a complete filtered probability space. Also, $W$ is an $\FF$-Brownian motion of dimension $d$, $X$ is an $\FF$-adapted $d$-dimensional process with $\PP \circ X_0^{-1}=\lambda$, and $\Lambda$ is a $\P(A)$-valued $\FF$-progressively measurable process. Lastly, $\bm{\widetilde\mu}$ is a $\P(\X)$-valued random variable such that $\bm{\widetilde\mu}(S)$ is $\F_t$-measurable whenever $S \in \F^\X_t$ and $t \in [0,T]$.
\item $\bm{\widetilde\mu}$, $X_0$, and $W$ are independent.
\item The state equation holds,
\begin{align*}
dX_t = \int_Ab(t,X_t,\widetilde{\mu}^x_t,a)\Lambda_t(da)dt + dW_t.
\end{align*}
\item The control $\Lambda$ is \emph{compatible}, in the sense that $\sigma(\Lambda_s : s \le t)$ is conditionally independent of $\F^{X_0,W,\bm{\widetilde\mu}}_T$ given $\F^{X_0,W,\bm{\widetilde\mu}}_t$, for each $t \in [0,T]$, where
\[
\F^{X_0,W,\bm{\widetilde\mu}}_t := \sigma(X_0, W_s, \bm{\widetilde\mu}(S) : s \le t, \, S \in \F^\X_t).
\]
\item The control $\Lambda$ is optimal, in the sense that if $(\Omega',\F',\FF',\PP',W',\bm{\widetilde\mu}',\Lambda',X')$ satisfies (1-4) and $\PP' \circ (\bm{\widetilde\mu}')^{-1} = \PP \circ \bm{\widetilde\mu}^{-1}$, then we have
\begin{align*}
\E^{\PP}&\left[\int_0^T\int_Af(t,X_t,\widetilde{\mu}^x_t,a)\Lambda_t(da)dt + g(X_T,\widetilde{\mu}^x_T)\right] \\
	&\ge \E^{\PP'}\left[\int_0^T\int_Af(t,X'_t,\widetilde\mu'^{x}_t,a)\Lambda'_t(da)dt + g(X'_T,\widetilde\mu'^{x}_T)\right].
\end{align*}
\item The consistency condition holds:  $\bm{\widetilde\mu} = \PP((W,\Lambda,X) \in \cdot \, | \, \bm{\widetilde\mu})$ a.s.
\end{enumerate}
We may abuse notation somewhat by referring to $\bm{\widetilde\mu}$ itself as a weak MFG solution. This is reasonable because we can recover the full joint law of $(\bm{\widetilde\mu},W,\Lambda,X)$ from that of $\bm{\widetilde\mu}$ by using the consistency condition (6).
\end{definition}

\begin{theorem} \label{th:weaksemiMarkov-to-weakMFGsolution}
Suppose $(\Omega,\F,\FF,\PP,W,\Lambda^*,X,\mu)$ is a weak semi-Markov RMFE. Let $\Lambda_t = \Lambda^*(t,X_t,\mu)$, and set $\bm{\widetilde\mu} = \PP((W,\Lambda,X) \in \cdot \, | \, \mu)$. Then $(\Omega,\F,\FF,\PP,W,\bm{\widetilde\mu},\Lambda,X)$ is a weak MFG solution.
\end{theorem}
\begin{proof}
First, define $\FF^{\bm{\widetilde\mu}}=(\F^{\bm{\widetilde\mu}}_t)_{t \in [0,T]}$ as the filtration generated by $\bm{\widetilde\mu}$, namely, $\F^{\bm{\widetilde\mu}}_t = \sigma(\bm{\widetilde\mu}(S) : S \in \F^\X_t)$. As usual, let $\F^\mu_t = \sigma(\mu_s : s \le t)$. We claim first that $\F^{\bm{\widetilde\mu}}_t=\F^\mu_t$ for each $t$. Recall from Remark \ref{re:Xcompatiblewithmu} that $\L(X_t \, | \, \mu) = \L(X_t \, | \, \F^\mu_t) = \mu_t$ a.s.\ for each $t$.
It follows immediately that $\F^\mu_t \subset \F^{\bm{\widetilde\mu}}_t$, because
\[
\mu_t = \PP(X_t \in \cdot \, | \, \mu) = \widetilde{\mu}^x_t, \ a.s.
\]
For the reverse, fix a bounded $\F^\X_t$-measurable function $h : \X \rightarrow \R$.
Note that $X$ is necessarily $\FF^{X_0,W,\mu}$-adapted by Lemma \ref{le:ap:SDErandom-uniqueness}, and thus so is $\Lambda$, where $\FF^{X_0,W,\mu} = (\F^{X_0,W,\mu}_t)_{t \in [0,T]}$ is defined by $\F^{X_0,W,\mu}_t=\sigma(X_0,W_s,\mu_s : s \le t)$. Hence, we may find a bounded $\F^{X_0,W,\mu}_t$-measurable random variable $\psi(X_0,W,\mu)$ such that $h(W,\Lambda,X) = \psi(X_0,W,\mu)$ a.s. Then,
\begin{align*}
\langle \bm{\widetilde\mu}, \, h\rangle &= \E[h(W,\Lambda,X) \, | \, \mu] = \E[\psi(X_0,W,\mu) \, | \, \mu] \\
	&= \langle \lambda \times \W, \, \psi(\cdot,\cdot,\mu)\rangle,
\end{align*}
where $\W$ denotes Wiener measure on $\C^d$, and the last identity follows from the independence of $X_0$, $W$, and $\mu$. Because $\psi(X_0,W,\mu)$ is $\F^{X_0,W,\mu}_t$-measurable, this shows that $\langle \bm{\widetilde\mu}, \, h\rangle$ is $\F^\mu_t$-measurable. Hence, $\F^\mu_t \supset \F^{\bm{\widetilde\mu}}_t$.

Properties (1-3) and (6) of Definition \ref{def:weakMFGsolution} are straightforward to check now that we have shown $\F^{\bm{\widetilde\mu}}_t=\F^\mu_t$ for each $t$. The compatibility property (4) follows easily from the fact that $\Lambda$ is $\FF^{X_0,W,\mu}=\FF^{X_0,W,\bm{\widetilde\mu}}$-adapted.

It remains to check property the optimality property (5). According to \cite[Lemma 3.11]{carmona-delarue-lacker} (see also \cite[Lemma 4.7]{lacker2016general}), it suffices to check (5) only for alternative controls $\Lambda'$ which are adapted to the filtration $\FF^{X_0,W,\bm{\widetilde\mu}}$, because such controls are dense in a joint distributional sense. Precisely, (5) is equivalent to the following:
\begin{enumerate}
\item[(5')] For each $\FF^{X_0,W,\bm{\widetilde\mu}}$-progressively measurable $\P(A)$-valued process $\Lambda'=(\Lambda'_t)_{t \in [0,T]}$, we have
\begin{align}
\E&\left[\int_0^T\int_Af(t,X_t,\widetilde{\mu}^x_t,a)\Lambda_t(da)dt + g(X_T,\widetilde{\mu}^x_T)\right] \nonumber \\
	&\ge \E\left[\int_0^T\int_Af(t,X'_t,\widetilde\mu^{x}_t,a)\Lambda'_t(da)dt + g(X'_T,\widetilde\mu^{x}_T)\right], \label{pf:openvsclosed1}
\end{align}
where $X'$ is the unique strong solution of the SDE
\[
dX'_t = \int_A b(t,X'_t,\widetilde{\mu}^x_t,a)\Lambda'_t(da)dt + dW_t, \quad X'_0=X_0.
\]
\end{enumerate}
Let $(X',\Lambda')$ be as in (5').
Recall that $\mu_t=\widetilde{\mu}^x_t$ for all $t \in [0,T]$, and so 
\eqref{pf:openvsclosed1} is equivalent to
\begin{align}
\E&\left[\int_0^T\int_Af(t,X_t,\mu_t,a)\Lambda^*(t,X_t,\mu)(da)dt + g(X_T,\mu_T)\right] \nonumber \\
	&\ge \E\left[\int_0^T\int_Af(t,X'_t,\mu_t,a)\Lambda'_t(da)dt + g(X'_T,\mu_T)\right]. \label{pf:openvsclosed2}
\end{align}
We showed also that $\F^\mu_t = \F^{\bm{\widetilde\mu}}_t$ for each $t \in [0,T]$, and thus $\F^{X_0,W,\bm{\widetilde\mu}}_t = \F^{X_0,W,\mu}_t := \sigma(X_0,W_s,\mu_s : s \le t)$.
Then $\Lambda'$ is $\FF^{X_0,W,\mu}$-progressively measurable, and we may write $\Lambda'_t= \Lambda'(t,X_0,W,\mu)$.

Because $\mu$ is a weak RMFE, we know that $\Lambda^*$ is optimal when compared to alternative \emph{semi-Markov} controls. To check that it is optimal over $\FF^{X_0,W,\bm{\widetilde\mu}}$ controls, we proceed by a projection argument reminiscent of those of Section \ref{se:comparisonofequilibria}. For $m \in \CP$, let $\PP^m = \PP( \cdot \, | \, \mu=m)$ denote a version of the regular conditional law given $\mu$. The statements in the rest of this paragraph hold for $\PP \circ \mu^{-1}$-almost every $m \in \CP$. Since $X_0$, $W$, and $\mu$ are independent, we have $\PP^m \circ (X_0,W)^{-1} = \PP \circ (X_0,W)^{-1} = \lambda \times \W$, where $\W$ denotes Wiener measure. Moreover, under $\PP^m$, the SDE still holds, which we may write as
\begin{align*}
dX'_t = \int_A b(t,X'_t,m_t,a)\Lambda'(t,X_0,W,m)(da)dt + dW_t.
\end{align*}
We wish to apply Theorem \ref{th:markovprojection} under this measure $\PP^m$. To do so, we first find a Borel measurable function $\widehat\Lambda : [0,T] \times \R^d \times \CP \rightarrow \P(A)$ such that
\begin{align}
\widehat\Lambda(t,X'_t,m) = \E^{\PP^m}[\Lambda'(t,X_0,W,m) \, | \, X'_t], \ \ \PP^m-a.s., \ \ \forall t \in [0,T], \label{pf:closedvsopenloop-hatlambdadef}
\end{align}
where these expectations are in the sense of mean measure; see Lemma \ref{le:conditional-marginal-meanmeasure}. The point of this definition is that the unique strong solution $X^m$ (on $(\Omega,\F,\FF,\PP)$) of the SDE
\[
dX^m_t = \int_A b(t,X^m_t,m_t,a)\widehat\Lambda(t,X^m_t,m)(da)dt + dW_t, \quad X^m_0=X_0,
\]
satisfies $\PP \circ (X^m_t)^{-1} = \PP^m \circ (X'_t)^{-1}$ for each $t \in [0,T]$, by Theorem \ref{th:markovprojection}.

At this point we would like to re-introduce the random measure flow by replacing $m$ by $\mu$ and treating $\widehat\Lambda(t,X'_t,\mu)$ as a semi-Markov control. For this to work, we must check that $\widehat\Lambda$ is not merely Borel measurable but rather semi-Markov. Note that $X'$ is a strong solution, so it is $\FF^{X_0,W,\mu}$-adapted, and we can write $X'_t=X'(t,X_0,W,\mu)$. We may then write \eqref{pf:closedvsopenloop-hatlambdadef} as
\begin{align*}
\widehat\Lambda(t,X'_t,m) = \E^{\PP^m}[\Lambda'(t,X_0,W,m) \, | \, X'(t,X_0,W,m)].
\end{align*}
Recall that $\PP^m \circ (X_0,W)^{-1} = \lambda \times \W$ and that $\Lambda'$ and $X'$ are progressive, which implies in particular that $\Lambda'(t,X_0,W,m) = \Lambda'(t,X_0,W,\widetilde{m})$ and $X'(t,X_0,W,m) = X'(t,X_0,W,\widetilde{m})$ a.s., whenever $t \in [0,T]$ and $m_s=\widetilde{m}_s$ for $s \in [0,t]$. From these facts we deduce that $\widehat\Lambda(t,X'_t,m) = \widehat\Lambda(t,X'_t,\widetilde{m})$ a.s., whenever $m_s=\widetilde{m}_s$ for $s \in [0,t]$.

Finally, returning to the unconditional measure $\PP$, define $\widehat{X}$ to be the unique strong solution (on $(\Omega,\F,\FF,\PP)$) of the SDE
\begin{align}
d\widehat{X}_t = \int_A b(t,\widehat{X}_t,\mu_t,a)\widehat\Lambda(t,\widehat{X}_t,\mu)(da)dt + dW_t, \quad \widehat{X}_0=X_0, \label{pf:closedvsopenloop-sde1}
\end{align}
and note that $\widehat{X}$ is adapted to $\FF^{X_0,W,\mu}$. Indeed, see Lemma \ref{le:ap:SDErandom-uniqueness} and \ref{le:ap:SDErandom-existence} for well-posedness of this SDE, despite the fact that $\widehat\Lambda$ may be discontinuous.
In addition, as we check carefully in the same two lemmas, the conditional law of $\widehat{X}$ given $\mu$ is precisely $\PP \circ (X^m)^{-1}$. In particular, we find
\begin{align*}
\PP^m \circ \widehat{X}_t^{-1} = \PP \circ (X_t^m)^{-1} = \PP^m \circ (X'_t)^{-1},
\end{align*}
for almost every $m$ and for each $t$. Equivalently, plugging in the random $\mu$, we have $\PP^\mu = \PP^\mu \circ (X'_t)^{-1}$ a.s. for each $t$. Using this and the definition of $\widehat\Lambda$, we finally use Fubini's theorem and the tower property of conditional expectation to get
\begin{align*}
 &\E\left[\int_0^T\int_Af(t,X'_t,\mu_t,a)\Lambda'_t(da)dt + g(X'_T,\mu_T)\right] \\
 &= \E\left[\int_0^T\int_Af(t,X'_t,\mu_t,a)\Lambda'(t,W,\mu)(da)dt + g(X'_T,\mu_T)\right] \\
 &= \E\left[\int_0^T\int_Af(t,X'_t,\mu_t,a)\widehat\Lambda(t,X'_t,\mu)(da)dt + g(X'_T,\mu_T)\right] \\
 &= \E\left[\int_0^T\int_Af(t,\widehat{X}_t,\mu_t,a)\widehat\Lambda(t,\widehat{X}_t,\mu)(da)dt + g(\widehat{X}_T,\mu_T)\right].
\end{align*}
Recalling the form of the SDE \eqref{pf:closedvsopenloop-sde1} for $\widehat{X}$, we may finally use the defining property (5) of a weak RMFE (Definition \ref{def:MFEsemimarkov}) to conclude that this expectation is dominated by
\[
\E\left[\int_0^T\int_Af(t,X_t,\mu_t,a)\Lambda^*(t,X_t,\mu)(da)dt + g(X_T,\mu_T)\right],
\]
which proves \eqref{pf:openvsclosed2}.
\end{proof}

\begin{theorem} \label{th:weakMFGsolution-to-weaksemiMarkov}
Suppose $\bm{\widetilde\mu}$ is a weak MFG solution. Then there exists a weak semi-Markov RMFE $\mu$ such that $\mu \stackrel{d}{=} \widetilde{\mu}^x$.
\end{theorem}
\begin{proof}
Let $(\Omega,\F,\FF,\PP,W,\bm{\widetilde\mu},\Lambda,X)$ be a weak MFG solution.
Recalling that $\bm{\widetilde\mu}$ is a random measure on $\X= \C^d \times \V \times \C^d$, let $\bm{\overline\mu}$ denote the image under under the map $\C^d \times \V \times \C^d \ni (w,q,x) \mapsto (x,q) \in \C^d \times \V$.
It is straightforward to check using the properties of Definition \ref{def:weakMFGsolution} and It\^o's formula that $\bm{\overline\mu}$ satisfies the identity \eqref{ass:le:projection-semiMarkov}. Hence Lemma \ref{le:projection-semiMarkov} applies, in particular part (c), and (enlarging the probability space if necessary) we may define $\Lambda^*$ and $X^*$ as therein. It is immediate from Lemma \ref{le:projection-semiMarkov} to check that properties (1-4) and (6) of Definition \ref{def:MFEsemimarkov} are valid. It remains to check the optimality property (5).

First, from part (b) of Lemma \ref{le:projection-semiMarkov}, note that
\begin{align}
\E^{\PP}&\left[\int_0^T\int_Af(t,X_t,\widetilde{\mu}^x_t,a)\Lambda_t(da)dt + g(X_T,\widetilde{\mu}^x_T)\right] \nonumber \\
	&= \E^{\PP}\left[\int_0^T\int_Af(t,X^*_t,\mu_t,a)\Lambda^*(t,X^*_t,\mu)(da)dt + g(X^*_T,\mu_T)\right]. \label{pf:weakMFGsolution-to-weaksemiMarkov1}
\end{align}
Fix any semi-Markov function $\Lambda' : [0,T] \times \R^d \times \CP \rightarrow \P(A)$, and let $X'$ denote the unique strong solution (see Lemmas \ref{le:ap:SDErandom-uniqueness} and \ref{le:ap:SDErandom-existence}) of the SDE
\[
d X'_t = \int_A b(t, X'_t,\mu_t,a) \Lambda'(t, X'_t,\mu)(da)dt + dW_t, \quad X_0 \sim \lambda.
\]
Note that $X'$ is adapted to the complete filtration generated by the process $(X_0,W_t,\mu_t)_{t \in [0,T]}$. Define the $\P(A)$-valued process $\widetilde{\Lambda}_t = \Lambda'(t,X'_t,\mu)$. One checks easily that $(\Omega,\F,\FF,\PP,W,\bm{\widetilde\mu},\widetilde{\Lambda},X')$ satisfies properties (1-4) of Definition \ref{def:weakMFGsolution}. Hence, using property (5) therein along with \eqref{pf:weakMFGsolution-to-weaksemiMarkov1}, we find 
\begin{align*}
\E^{\PP}&\left[\int_0^T\int_Af(t,X^*_t,\mu_t,a)\Lambda^*(t,X^*_t,\mu)(da)dt + g(X^*_T,\mu_T)\right] \\
	&\ge \E^{\PP}\left[\int_0^T\int_Af(t,X'_t,\widetilde{\mu}^x_t,a)\widetilde{\Lambda}_t(da)dt + g(X'_T,\widetilde{\mu}^x_T)\right] \\
	&= \E^{\PP}\left[\int_0^T\int_Af(t,X'_t,\mu_t,a)\Lambda'(t,X_t,\mu)(da)dt + g(X'_T,\mu_T)\right].
\end{align*}
This is valid for any choice of $\Lambda'$, and we conclude that property (5) of Definition \ref{def:MFEsemimarkov} holds.
\end{proof}

We can now give a very concise proofs of Theorem \ref{th:openloopconverse} and Theorem \ref{th:uniqueness}, taking advantage of the two theorems above.
A direct and more illuminating proof of the latter is certainly possible, but the paper is already rather long.

\subsection{Proof of Theorem \ref{th:uniqueness}} \label{se:uniqueness-proof}

By \cite[Theorem 6.2]{carmona-delarue-lacker}, the assumptions of Theorem \ref{th:uniqueness} ensure uniqueness in law for weak MFG solutions in the sense of Definition \ref{def:weakMFGsolution}.
Because of Theorem \ref{th:weaksemiMarkov-to-weakMFGsolution}, this gives uniqueness in law for weak RMFE in the sense of Definition \ref{def:MFEsemimarkov}, and in particular uniqueness in law for weak MFE in the sense of Definition \ref{def:MFEsemimarkov-strict}. \hfill\qedsymbol

\subsection{Proof of Theorem \ref{th:openloopconverse}} 

It was shown in \cite[Theorem 3.4]{lacker2016general} that both claims are true if ``weak MFE" is replaced by ``weak MFG solution" in the statements. We saw in Theorem \ref{th:weakMFGsolution-to-weaksemiMarkov} that a weak MFG solution is a weak RMFE and in Proposition \ref{pr:RMFE-to-MFE} that a weak RMFE is a weak MFE under Assumptions \ref{assumption:A} and \ref{assumption:B}. \hfill\qedsymbol

\section{Constructing $n$-player equilibria from mean field equilibria} \label{se:examples}

This section continues the discussion of Section \ref{se:selectionoflimits} on the question of which weak MFE can arise as the limit of $n$-player (approximate) Nash equilibria. 
We begin in Section \ref{se:converse-strongMFE-proof} by proving Theorem \ref{th:converselimit-strongMFE-relaxed}, which states that every strong RMFE arises as the limit of $n$-player approximate equilibria.

The rest of the section is devoted to examples: We warm up in Section \ref{se:contrived} with some observations on the case where the game-theoretic aspect of the problem degenerates in the sense that $A$ is a singleton.
In this uncontrolled regime, we are simply left with the study of McKean-Vlasov limits,
which already reveals of some of the range of possible behaviors.

However, much richer behavior is possible when the game-theoretic aspect does not trivialize. Section \ref{se:ex:weakMFEconverse} discusses such an example, in which there exist weak MFE which are not mixtures of strong MFE.

\subsection{Proof of Theorem \ref{th:converselimit-strongMFE-relaxed}} \label{se:converse-strongMFE-proof}

Let $(m,\Lambda^*)$ be a strong RMFE, in the sense of Definition \ref{def:strongMFE}. Let $X^*$ denote the corresponding state process,
\begin{align}
dX^*_t = \int_Ab(t,X^*_t,m_t,a)\Lambda^*(t,X^*_t)(da)dt + dW_t, \quad X^*_0 \sim \lambda.  \label{pf:strongMFElimit0}
\end{align}
Now, for the $n$-player game, define $\Lambda^{n,i} \in \RCM_n$ by setting
\[
\Lambda^{n,i}(t,\bm{x}) = \Lambda^*(t,x_i), \quad \text{for} \quad \bm{x}=(x_1,\ldots,x_n) \in (\R^d)^n.
\]
Define
\begin{align*}
\epsilon_n := \sup_{\beta \in \RCM_n}J^n_1(\beta,\Lambda^{n,2},\ldots,\Lambda^{n,n}) - J^n_1(\Lambda^{n,1},\ldots,\Lambda^{n,n}).
\end{align*}
Note that $\epsilon_n \ge 0$, and by symmetry it holds for any $i \in \{1,\ldots,n\}$ that
\begin{align*}
\epsilon_n = \sup_{\beta \in \RCM_n}J^n_i(\Lambda^{n,1},\ldots,\Lambda^{n,i-1},\beta,\Lambda^{n,i+1},\ldots,\Lambda^{n,n}) - J^n_i(\Lambda^{n,1},\ldots,\Lambda^{n,n}).
\end{align*}
Hence, $\bm{\Lambda}^n=(\Lambda^{n,1},\ldots,\Lambda^{n,n})$ is an $\epsilon_n$-Nash equilibrium. Assumption \ref{assumption:C} lets us apply the result of \cite[Theorem 2.5(2)]{lacker2018strong} (or more specifically Remark 2.7 therein), a strong form of propagation of chaos, to conclude that $\mu^n \rightarrow m$ in law in $\CP$. Moreover, for any $t \in [0,T]$ and any bounded measurable (not necessarily continuous) function $\varphi : \R^d \rightarrow \R$, we have
\begin{align}
\int_{\R^d} \varphi\,d\mu^n_t[\bm{\Lambda}^n] \rightarrow \int_{\R^d} \varphi\,dm_t, \label{pf:strongMFElimit0-1}
\end{align}
in probability.

It remains to show that $\epsilon_n \rightarrow 0$. 
Fix arbitrarily a sequence $\beta^n \in \RCM_n$ such that 
\begin{align}
J^n_1(\beta^n,\Lambda^{n,2},\ldots,\Lambda^{n,n}) \ge \sup_{\beta \in \RCM_n}J^n_1(\beta,\Lambda^{n,2},\ldots,\Lambda^{n,n}) - \frac{1}{n}. \label{pf:strongMFElimit1}
\end{align}
Abbreviate $\bm{X}^n=(X^{n,1},\ldots,X^{n,n})=\bm{X}[\bm{\Lambda}^n]$ and $\mu^n=\mu^n[\bm{\Lambda}^n]$, as well as
\begin{align*}
\bm{Y}^n &= (Y^{n,1},\ldots,Y^{n,n}) = \bm{X}[(\beta^n,\Lambda^{n,2},\ldots,\Lambda^{n,n})], \\
\nu^n &= \mu^n[(\beta^n,\Lambda^{n,2},\ldots,\Lambda^{n,n})].
\end{align*}
In particular, the state process $\bm{X}^n$ follows the SDEs
\begin{align*}
dX^{n,i}_t &= \int_A b(t,X^{n,i}_t,\mu^n_t,a)\Lambda^*(t,X^{n,i}_t)(da)dt + dW^i_t, \quad \ \  \mu^n_t = \frac{1}{n}\sum_{k=1}^n\delta_{X^{n,k}_t},
\end{align*}
whereas $\bm{Y}^n$ follows the SDEs
\begin{align*}
dY^{n,1}_t &= \int_Ab(t,Y^{n,1}_t,\nu^{n}_t,a)\beta^n(t,\bm{Y}^{n})(da)dt + dW^1_t, \\
dY^{n,k}_t &= \int_Ab(t,Y^{n,k}_t,\nu^{n}_t,a)\Lambda^*(t,Y^{n,k}_t)(da)dt + dW^k_t, \quad i \neq 1, \\
\nu^n_t &= \frac{1}{n}\sum_{j=1}^n\delta_{Y^{n,j}_t}, 
\end{align*}
Suppose that $\bm{X}^n$ is defined on a filtered probability space $(\Omega^n,\F^n,\FF^n,\PP^n)$, where $W^k$ are of course assumed to be $\FF^n$-Brownian motions. (We will avoid giving a name to whatever probability space $\bm{Y}^n$ is defined on, which may be different.)
Define a probability measure $\QQ^n$ on $(\Omega^n,\F^n,\FF^n)$ by
\begin{align*}
\frac{d\QQ^n}{d\PP^n} &= \exp\Bigg( \int_0^T \int_Ab(t,X^{n,1}_t,\mu^n_t,a)(\beta^n(t,\bm{X}^n) - \Lambda^*(t,X^{n,1}_t))(da) dW^1_t \\
	&\quad\quad - \frac12\int_0^T\left|\int_Ab(t,X^{n,1}_t,\mu^n_t,a)(\beta^n(t,\bm{X}^n) - \Lambda^*(t,X^{n,1}_t))(da)\right|^2dt \Bigg),
\end{align*}
By Girsanov's theorem and uniqueness of the SDEs, we have $\QQ^n \circ (\bm{X}^n)^{-1} = \L(\bm{Y}^n)$.
Boundedness of $b$ implies that 
\begin{align*}
\sup_{n \in \N}\E^{\PP^n}\left[\left|\frac{d\QQ^n}{d\PP^n}\right|^p\right] < \infty,
\end{align*}
for all $p \ge 1$.  Hence, because $\mu^n$ converges in probability to $m$ under $\PP^n$ (in the sense that $\lim_{n\rightarrow\infty}\PP^n(\mu^n \notin U) = 0$ for any open neighborhood $U$ of $m$ in $\CP$), it also converges in probability to $m$ under $\QQ^n$. But $\QQ^n \circ (\mu^n)^{-1} = \L(\nu^n)$, and so $\nu^n \rightarrow m$ in probability.\footnote{For a metric space $(E,d)$, a point $e_0 \in E$, and a sequence $\xi_n$ of $E$-valued random variables, perhaps defined on different probability spaces, recall that $\L(\xi_n) \rightarrow \delta_e$ weakly if and only if $\xi_n \rightarrow e$ in probability, which means $\lim_{n\rightarrow\infty}\PP(d(\xi_n,e_0) > \epsilon) = 0$ for all $\epsilon > 0$.}

Now, view $(Y^{n,1},\beta^n(\cdot,\bm{Y^n}),W^1)$ as a random element of $\C^d \times \V \times \C^d$, where the space $\V$ of relaxed controls was defined in Section \ref{se:relaxedcontrols}. Recalling that $\V$ is compact, it is straightforward to check that this sequence is tight. Letting $(Y,\beta,W)$ denote any subsequential limit point, one readily checks using continuity of $b$ that $W$ is a Brownian motion with respect to the filtration $(\sigma(Y_s,\beta_s,W_s : s \le t))_{t \in [0,T]}$, that $Y_0 \sim \lambda$, and that the SDE holds,
\[
dY_t = \int_A b(t,Y_t,m_t,a)\beta_t(da)dt + dW_t.
\]
Use Lemma \ref{le:conditional-marginal-meanmeasure} to find a measurable function $\Lambda : [0,T] \times \R^d \rightarrow \P(A)$ such that
\begin{align*}
\Lambda(t,Y_t) = \E[\beta_t \, | \, Y_t], \ \ a.s., \ \ a.e. \ t,
\end{align*}
in the sense of mean measures. Apply Theorem \ref{th:markovprojection} to find that $Y_t\stackrel{d}{=} Z_t$ for all $t \in [0,T]$, where $Z$ is the unique strong solution of the SDE
\[
dZ_t = \int_A b(t,Z_t,m_t,a)\Lambda(t,Z_t)(da)dt + dW_t.
\]
Using the assumption that $f$ and $g$ are bounded and continuous, we conclude that, along the same convergent subsequence for which $(Y^{n,1},\beta^n(\cdot,\bm{Y^n}),W^1)$ converges to $(Y,\beta,W)$, we have
\begin{align*}
\lim_n \ &\E\left[\int_0^T\int_A f(t,Y^{n,1}_t,\nu^{n}_t,a)\beta^n(t,\bm{Y}^{n})(da)dt + g(Y^{n,1}_T,\nu^{n}_T)\right]  \\
	&= \E\left[\int_0^T\int_A f(t,Y_t,m_t,a)\beta_t(da)dt + g(Y_T,m_T)\right] \\
	&= \E\left[\int_0^T\int_A f(t,Y_t,m_t,a)\Lambda(t,Y_t)(da)dt + g(Y_T,m_T)\right] \\
	&= \E\left[\int_0^T\int_A f(t,Z_t,m_t,a)\Lambda(t,Z_t)(da)dt + g(Z_T,m_T)\right] \\
	&\le \E\left[\int_0^T\int_A f(t,X^*_t,m_t,a)\Lambda^*(t,X^*_t)(da)dt + g(X^*_T,m_T)\right],
\end{align*}
where the last inequality is from the optimality part of the assumption that $m$ is a strong MFE. This inequality holds for any convergent subsequence of the tight sequence $(Y^{n,1},\beta^n(t,\bm{Y^n}),W^1)$, and we conclude that 
\begin{align*}
\limsup_{n\rightarrow\infty}J^n_1(\beta^n,\Lambda^{n,2},\ldots,\Lambda^{n,n}) &= 
\limsup_{n\rightarrow\infty} \ \E\left[\int_0^T\int_A f(t,Y^{n,1}_t,\nu^{n}_t,a)\beta^n(t,\bm{Y}^{n})(da)dt + g(Y^{n,1}_T,\nu^{n}_T)\right]  \\
	&\le \E\left[\int_0^T\int_A f(t,X^*_t,m_t,a)\Lambda^*(t,X^*_t)(da)dt + g(X^*_T,m_T)\right].
\end{align*}

On the other hand, notice that the convergence $\mu^n \rightarrow m$ implies
\begin{align*}
\lim_{n\rightarrow\infty}J^n_1(\bm{\Lambda}^n) &= \lim_{n\rightarrow\infty}\E\left[\int_0^T\int_A f(t,X^{n,1}_t,\mu^n_t,a)\Lambda^*(t,X^{n,1}_t)(da)dt + g(X^{n,1}_T,\mu^n_T)\right] \\
	&= \lim_{n\rightarrow\infty}\E\int_0^T\int_{\R^d}\int_A f(t,x,\mu^n_t,a)\Lambda^*(t,x)(da)\mu^n_t(dx)dt + \E\int_{\R^d}g(x,\mu^n_T)\mu^n_T(dx) \\
	&= \int_0^T\int_{\R^d}\int_A f(t,x,m_t,a)\Lambda^*(t,x)(da)m_t(dx)dt + \int_{\R^d}g(x,m_T)m_T(dx) \\
	&= \E\left[\int_0^T\int_A f(t,X^*_t,m_t,a)\Lambda^*(t,X^*_t)(da)dt + g(X^*_T,m_T)\right],
\end{align*}
where the second line used symmetry and the third used \eqref{pf:strongMFElimit0-1} to deal with the fact that $\Lambda^*$ may be discontinuous.
Recalling the previous inequality and \eqref{pf:strongMFElimit1}, we conclude that $\epsilon_n \rightarrow 0$. \hfill\qedsymbol

\subsection{Uncontrolled models and ill-posed ODEs} \label{se:contrived}

Weak MFE are easy to construct by building degenerate control problems into ill-posed McKean-Vlasov equations or ODEs, as illustrated in this section. Suppose the drift function is the trivial
\begin{align*}
b(t,x,m,a) = B(\overline{m}),
\end{align*}
for some bounded continuous function $B : \R^d \rightarrow \R^d$, where we again denote by $\overline{m}$ the mean of a measure $m \in \P(\R^d)$, if it exists. The state process $(X^1,\ldots,X^n)$ of the $n$-player game are then un-controlled, and we do not even need to specify objective functions $(f,g)$ or an action space $A$.
The dimension $d$ is arbitrary. 
The state processes then evolve according to
\begin{align}
dX^i_t = B(\overline\mu^n_t)dt + dW^i_t, \quad \mu^n_t = \frac{1}{n}\sum_{k=1}^n\delta_{X^k_t}, \label{ex2:SDE}
\end{align}
with i.i.d.\ initial states given by $\lambda$.

This is the unique $n$-player equilibrium, and the above SDE system \eqref{ex2:SDE} is unique in law. But a broad range of $n\rightarrow\infty$ limiting behavior is possible here, and there are potentially multiple (weak) MFE. Averaging \eqref{ex2:SDE} over $i=1,\ldots,n$, the empirical mean is seen to follow
\begin{align*}
d\overline\mu^n_t = B(\overline\mu^n_t)dt + \frac{1}{\sqrt{n}}d\overline W_t,
\end{align*}
where $\overline W := \frac{1}{\sqrt{n}}\sum_{k=1}^nW^k$ is a Brownian motion. The sequence of real-valued processes $(\overline\mu^n_t)_{t \in [0,T]}$ is easily seen to be tight (using, e.g., Aldous' criterion for tightness \cite[Lemma 16.12]{kallenberg-foundations}), and it is straightforward to check that every weak limit is supported on the set $S_{\mathrm{ODE}} \subset C([0,T];\R)$ consisting of those functions $x=x(t)$ satisfying the integral equation
\begin{align}
x(t) = \overline{\lambda} + \int_0^tB(x(s))ds, \quad \forall t \in [0,T]. \label{ex2:ODE}
\end{align}
It can be checked that a $\P(\R^d)$-valued process $\mu=(\mu_t)_{t \in [0,T]}$ is a weak MFE if and only if $(\overline\mu_t)_{t \in [0,T]}$ belongs almost surely to $S_{\mathrm{ODE}}$ and $\mu_t$ is precisely 
\[
\mu_t = \mathcal{N}_d\left(\overline{\lambda} + \int_0^tB(\overline\mu_s)ds, \, tI\right),
\]
where $\mathcal{N}_d(m,\Sigma)$ denotes the $d$-dimensional Gaussian law with mean vector $m$ and covariance matrix $\Sigma$. In particular, weak MFE are parametrized by mixtures of solutions of the ODE \eqref{ex2:ODE}.

Of course, in some cases, such as if $B$ is Lipschitz, this ODE has a unique solution. In this case, there is a unique MFE, and the $n$-player equilibrium converges to it. But without uniqueness for \eqref{ex2:ODE}, anything could happen. The vanishing noise limit $n\rightarrow\infty$ may select one particular solution, or it may fail to converge at all. See \cite{bafico1982small,trevisan2013zero} for examples of this phenomenon.

\subsection{A game-theoretic example} \label{se:ex:weakMFEconverse}
We now turn to a more interesting example, in which the nonuniqueness of the MFE comes from the game-theoretic aspect rather than from ill-posed state process dynamics. 
In particular, this example admits many weak MFE which are not mixtures of strong MFE.
Consider the $d=1$-dimensional mean field game described by the coefficients
\begin{align*}
b(t,x,m,a) = a, \quad f \equiv 0, \quad g(x,m) = x\overline{m}, \quad A = [-1,1], \quad \lambda = \delta_0,
\end{align*}
where $\overline{m} = \int_\R y\,m(dy)$.
This example was analyzed in \cite[Section 3.3]{lacker2016general}. It was shown in Proposition 3.6 therein that there are precisely three strong MFE, $m^{-1}$, $m^0$, and $m^1$, defined by
\begin{align}
m^c_t = \L(ct + W_t), \quad\quad \text{for} \quad c \in \{-1,0,1\}. \label{ex:strongMFE-3}
\end{align}
On the other hand, there are infinitely many weak MFE, many of which are \emph{not mixtures} of these three strong MFE. In \cite[Proposition 3.7]{lacker2016general}, one such weak MFE was constructed explicitly, and we elaborate somewhat on this construction below. Note that \cite{lacker2016general} works with \emph{weak MFG solutions} in the sense of Definition \ref{def:weakMFGsolution} instead of our notion of \emph{weak semi-Markov MFE} (Definition \ref{def:MFEsemimarkov}), but we saw in Section \ref{se:closedloopvsopenloop-proofs} that the two are equivalent in a sense.

To construct a family of weak MFE, let $t_0 \in [0,T]$, and let $(\Omega,\F,\PP)$ be any probability space supporting a Brownian motion $W$ and an independent random variable $\gamma$ with $\PP(\gamma=1)=\PP(\gamma=-1)=1/2$.
Define a $\P(\R^d)$-valued process $\mu=(\mu_t)_{t \in [0,T]}$ by
\begin{align}
\mu_t = \begin{cases}
\L(W_t) &\text{if } t \le t_0 \\
\L(W_t + \gamma(t-t_0) \, | \, \gamma) &\text{if } t \in (t_0,T],
\end{cases}  \label{ex:mu-dynamics}
\end{align}
and note that the mean of $\mu_t$ is
\begin{align}
\overline{\mu}_t = \gamma(t - t_0)^+. \label{ex:overlinemu-dynamics}
\end{align}

Suppose $\FF=(\F_t)_{t \in [0,T]}$ is the complete filtration generated by the processes $(W,\mu)$. In particular, $\F_t = \sigma(W_s : s \le t)$ for $t \le t_0$, and $\F_t = \sigma(W_s,\gamma : s \le t)$ for $t \in (t_0,T]$. Define the state process
\[
dX^*_t = \gamma 1_{(t_0,T]}(t)dt + dW_t, \quad X^*_0=0,
\]
and define a control $\alpha^*_{t_0} : [0,T] \times \R \rightarrow \R$ by
\[
\alpha^*_{t_0}(t,x) = \mathrm{sgn}(x)1_{(t_0,T]}(t),
\]
where
\[
\mathrm{sgn}(x) := \begin{cases}
1 &\text{if } x > 0 \\
-1 &\text{if } x < 0 \\
0 &\text{if } x = 0.
\end{cases}
\]
Then $\gamma = \mathrm{sgn}(\overline\mu_t)$ for $t > t_0$, and we can rewrite the dynamics of $X^*$ as
\begin{align}
dX^*_t = \alpha^*_{t_0}(t,\overline{\mu}_t)dt + dW_t. \label{intro:example:MKVequation}
\end{align}

We claim that $(\Omega,\F,\FF,\PP,W,\alpha^*_{t_0},X^*,\mu)$ is a weak MFE in the sense of Definition \ref{def:MFEsemimarkov-strict}. 
To check that the consistency condition $\mu_t = \L(X^*_t \, | \, \F^\mu_t)$ holds, note first that the $\sigma$-field $\F^\mu_t := \sigma(\mu_s : s \le t)$ is trivial if $t \le t_0$ and is equal to $\sigma(\gamma)$ if $t \in (t_0,T]$. Hence,
\begin{align*}
\L(X^*_t \, | \, \F^\mu_t) &= \L\left( \left. W_t + \int_0^t\alpha^*_{t_0}(s,\overline{\mu}_s)ds \, \right| \, \F^\mu_t\right)  = \L\left( \left. W_t + \gamma(t-t_0)^+ \, \right| \, \F^\mu_t\right) = \mu_t.
\end{align*}
We must lastly check that the control $\alpha^*_{t_0}$ defined above is optimal.
Fix an alternative semi-Markov control $\alpha=\alpha(t,x,m)$, and define the state process 
\[
dX'_t = \alpha(t,X'_t,\mu)dt + dW_t, \quad X'_0 = 0.
\]
The corresponding reward, using the fact that $\mu$ and $W$ are independent, is
\begin{align*}
J(\alpha) &:= \E[X'_T\overline\mu_T] = \E\left[\int_0^T\alpha(t,X'_t,\mu)dt\right] = \E\left[\int_0^T\alpha(t,X'_t,\mu)\E[\overline\mu_T \, | \, \F_t]dt\right] \\
	&= \E\left[\int_0^T \alpha(t,X'_t,\mu)\mathrm{sgn}(\overline\mu_t)\,dt\right].
\end{align*}
Indeed, the last step follows from the independence of $W$ and $\mu$, which yields
\begin{align*}
\E[\overline\mu_T \, | \, \F_t] &= \begin{cases}
\overline\mu_T = \mathrm{sgn}(\overline\mu_T) = \mathrm{sgn}(\overline\mu_t) &\text{if } t \in (t_0,T] \\
\E[\overline\mu_T] = 0 &\text{if } t \le t_0
\end{cases} \\
	&= \alpha^*_{t_0}(t,\overline\mu_t).
\end{align*}
The optimizers of $J(\alpha)$ over $\alpha$ are precisely those $\alpha$ which satisfy
\[
\alpha(t,X'_t,\mu) = \alpha^*_{t_0}(t,\overline\mu_t), \quad \text{for } t \in (1,T].
\]
In particular, the control $\alpha^*_{t_0}$ itself above is optimal, and we conclude that $(\Omega,\F,\FF,\PP,W,\alpha^*_{t_0},X^*,\mu)$ is a weak semi-Markov MFE.

\begin{remark}
This example notably illustrates weak MFE which are not mixtures of strong MFE. Indeed, recall from \eqref{ex:strongMFE-3} that the three strong MFE are $m^{-1},m^0,m^1$. The weak MFE $\mu$ constructed above satisfies in particular $\PP(|\overline\mu_T| = T-t_0) = 1$. Hence, unless $t_0=0$ or $t_0=T$, this weak MFE is not a mixture of strong MFE.
\end{remark}

The McKean-Vlasov equation in \eqref{intro:example:MKVequation} is ill-posed (by design), which renders this example difficult to analyze.
Indeed, consider the set $S^*_{t_0}$ of $m \in \CP$ such that there exists a solution of
\[
dX_t = \alpha^*_{t_0}(t,\overline{m}_t)dt + dW_t, \quad X_0=0, \quad \L(X_t)=m_t, \  \ \forall t \in [0,T].
\]
Taking expectations, we find 
\begin{align}
d\overline{m}_t = \alpha^*_{t_0}(t,\overline{m}_t)dt, \quad \overline{m}_0=0. \label{intro:ODEillposed}
\end{align}
This is an ill-posed ODE, and its solutions (on the time interval $[0,T]$) are precisely the functions $\{H_s^\pm : s \in [t_0,T]\}$, where
\begin{align}
H_s^\pm(t) = \pm (t-s)^+, \label{def:Hpm-ODE}
\end{align}
noting that $H^{\pm}_T \equiv 0$. Note then that $S^*_{t_0}$ consists of precisely the measure flows of the form $(\L(W_t + H_s^{\pm}(t)))_{t \in [0,T]}$, for $s \in [t_0,T]$.

On the other hand, suppose we construct the natural $n$-particle system
\begin{align*}
dX^i_t = \alpha^*_{t_0}(t,\overline\mu^n_t)dt + dW^i_t, \quad X^i_0=0, \quad \mu^n_0 = \frac{1}{n}\sum_{k=1}^n\delta_{X^k_t}.
\end{align*}
Averaging over $i=1,\ldots,n$, we find that the empirical mean satisfies
\begin{align}
d\overline\mu^n_t = \alpha^*_{t_0}(t,\overline\mu^n_t)dt + \frac{1}{\sqrt{n}}d\overline{W}_t, \quad \overline\mu^n_t = 0, \label{ex:nparticle-meanflow}
\end{align}
where $\overline{W}_t = \frac{1}{\sqrt{n}}\sum_{k=1}^nW^k_t$ is a Brownian motion. One would expect that as $n\rightarrow\infty$ the limit points of $(\overline\mu^n_t)_{t \in [0,T]}$ are supported on solutions of the ODE \eqref{intro:ODEillposed}. But, in fact, this is a well understood example of the ``regularization by noise" phenomenon, and a particular mixture is picked out in the limit $n\rightarrow\infty$. Indeed, the law of $(\overline\mu^n_t)_{t \in [0,T]}$ converges to the mixture $\frac12 \delta_{H^+_{t_0}} + \frac12 \delta_{H^-_{t_0}}$; this was proven in \cite{trevisan2013zero} in the case $t_0=0$, and the extension to general $t_0$ is straightforward. In addition, one can deduce from this that the full measure flow $\mu^n$, not just its mean, converges in law in $\CP$ to $\mu$ defined in \eqref{ex:mu-dynamics}.

In light of this discussion, and after studying the proof of Theorem \ref{th:converselimit-strongMFE}, it is natural to guess that
\begin{align}
\alpha^{n,i}_{t_0}(t,\bm{x}) := \alpha^*_{t_0}\left(t,\frac{1}{n}\sum_{k=1}^nx_k\right) \label{def:ex:alpha-n,i}
\end{align}
defines an approximate (Markovian) Nash equilibrium for the $n$-player game, for any $t_0 \in [0,T]$. For $t_0=T$ this is true and follows from Theorem \ref{th:converselimit-strongMFE}, because the MFE $\mu_t=\L(W_t)$ is strong in this case. For general $t_0 \in [0,T)$ it is not as clear, and we have resolved only the $t_0=0$ case:

\begin{proposition} \label{pr:ex:weakMFEconverse}
Let $\bm{\alpha}^n=(\alpha^{n,1}_0,\ldots,\alpha^{n,n}_0)$, where $\alpha^{n,i}_0$ are defined as in \eqref{def:ex:alpha-n,i} with $t_0=0$. Then there exists $\epsilon_n \ge 0$ with $\epsilon_n \rightarrow 0$ such that $\bm{\alpha}^n$ is a Markovian $\epsilon_n$-Nash equilibrium for each $n$. Moreover, the law of the $\C^d$-valued random variable $(\overline{\mu}^n_t[\bm{\alpha}^n])_{t \in [0,T]}$ converges to $\frac12\delta_{H^+_0} + \frac12\delta_{H^-_0}$.
\end{proposition}

\begin{remark}
On the other hand, suppose instead that we take $\gamma$ to be $1$, $-1$, or $0$ with $\PP(\gamma =1)=\PP(\gamma=-1)=p < 1/2$ so that $\E\gamma =0$ and $\PP(\gamma =0) > 0$. Carrying out the exact same construction as above, we arrive at another weak MFE in which $(\mu,X^*)$ once again obeys the dynamics
\begin{align*}
dX^*_t &= \alpha^*_{t_0}(t,\overline\mu_t)dt + dW_t, \quad X^*_0 =0, \quad \mu_t = \L(X_t \, | \, \F^\mu_t), \ a.s., \ \forall t \in [0,T],
\end{align*}
and again with $\mu$ satisfying both \eqref{ex:overlinemu-dynamics} and \eqref{ex:mu-dynamics}. The point is that in this case the law of $(\overline\mu_t)_{t \in [0,T]}$ is given by the mixture $p \delta_{H^+_{t_0}} + p\delta_{H^-_{t_0}} + (1-2p) \delta_{0}$. This is not the mixture picked out in the limit from the $n$-particle system \eqref{ex:nparticle-meanflow}, in which we saw that the law of $\overline\mu^n$ converges to $\frac12 \delta_{H^+_{t_0}} + \frac12 \delta_{H^-_{t_0}}$. In this case, it is not clear if this particular weak MFE can arise as the limit of $n$-player approximate equilibria, but the naive construction certainly fails.
\end{remark}

\subsection{Proof of Proposition \ref{pr:ex:weakMFEconverse}}

Recall that our weak MFE $\mu$ satisfies $\overline\mu_t = \gamma t$, where $\PP(\gamma=1)=\PP(\gamma=-1)=1/2$.
The final claim of the Proposition, that the law of $\overline\mu^n[\bm{\alpha}^n]$ converges to $\frac12\delta_{H^+_0}+\frac12\delta_{H^-_0}$, was shown in \cite{trevisan2013zero} .

Define
\begin{align*}
\epsilon_n := \sup_{\beta \in \AM_n}J^n_1(\beta,\alpha_0^{n,2},\ldots,\alpha_0^{n,n}) - J^n_1(\bm{\alpha}^n).
\end{align*}
Note that $\epsilon_n \ge 0$, and by symmetry it holds for any $k \in \{1,\ldots,n\}$ that
\begin{align*}
\epsilon_n = \sup_{\beta \in \AM_n}J^n_k(\alpha_0^{n,1},\ldots,\alpha_0^{n,k-1},\beta,\alpha_0^{n,k+1},\ldots,\alpha_0^{n,n}) - J^n_k(\bm{\alpha}^n).
\end{align*}
Hence, $\bm{\alpha}^n=(\alpha_0^{n,1},\ldots,\alpha_0^{n,n})$ is an $\epsilon_n$-Nash equilibrium.
It remains to show that $\epsilon_n \rightarrow 0$.
A direct calculation, using symmetry and the fact that $|\overline\mu^n_t[\bm{\alpha}^n]| \rightarrow t$ in law, shows that 
\begin{align*}
J^n_1(\alpha_0^{n,1},\ldots,\alpha_0^{n,n}) &= \E[X_T^1[\bm{\alpha}^n]\overline\mu^n_T[\bm{\alpha}^n]] = \E[|\overline\mu^n_T[\bm{\alpha}^n]|^2] \rightarrow T^2
\end{align*}
as $n\rightarrow\infty$. Hence, to show that $\epsilon_n \rightarrow 0$, it suffices to show that
\begin{align}
\limsup_{n\rightarrow\infty} \sup_{\beta \in \AM_n}J^n_1(\beta,\alpha^{n,2},\ldots,\alpha^{n,n}) \le T^2. \label{pf:ex:weakMFEconverse1}
\end{align}

To this end, for each $n$ find $\beta^n \in \AM_n$ such that
\begin{align}
\sup_{\beta \in \AM_n}J^n_1(\beta,\alpha_0^{n,2},\ldots,\alpha_0^{n,n}) \le J^n_1(\beta^n,\alpha_0^{n,2},\ldots,\alpha_0^{n,n}) + \frac{1}{n}.\label{pf:ex:weakMFEconverse1-1}
\end{align}
Abbreviate $X^n = X^1[(\beta^n,\alpha_0^{n,2},\ldots,\alpha_0^{n,n})]$, $Y^n = \overline\mu^n[(\beta^n,\alpha_0^{n,2},\ldots,\alpha_0^{n,n})]$. Abuse notation by writing $\beta^n_t = \beta^n(t,\bm{X}_t[(\beta^n,\alpha_0^{n,2},\ldots,\alpha_0^{n,n})])$. Then 
\begin{align*}
dX^n_t &= \beta^n_tdt + dW^1_t, \\
dY^n_t &= \left(\frac{1}{n}\beta^n_t + \frac{n-1}{n}\mathrm{sgn}(Y^n_t) \right)dt + \frac{1}{n}\sum_{k=1}^ndW^k_t.
\end{align*}
Finally, we view $\beta^n$ as a random variable with values in $L^2_1 := L^2([0,T];[-1,1])$.
Equip $L^2_1$ with the subspace topology inherited from the weak topology of the Hilbert space $L^2([0,T];\R)$, and note that $L^2_1$ is then compact and metrizable.

\begin{lemma}
The sequence $(X^n,Y^n,W^1,\beta^n)$ of $\C \times \C \times \C \times L^2_1$-valued random variables is tight, and every weak limit $(X,Y,W,\beta)$ satisfies:
\begin{enumerate}[(i)]
\item $\L(Y) = \frac12\delta_{H^+_0} + \frac12\delta_{H^-_0}$, where $H^{\pm}_0$ are defined in \eqref{def:Hpm-ODE}.
\item The following equations hold, for $t \in [0,T]$:
\begin{align*}
X_t = \int_0^t\beta_sds + W_t, \quad\quad\quad\quad
Y_t = \int_0^t \mathrm{sgn}(Y_s) ds.
\end{align*}
\item $W$ is a Brownian motion with respect to the filtration $\FF=(\F_t)_{t \in [0,T]}$ defined by $\F_t = \sigma(X_s,Y_s,W_s,\beta_s : s \le t)$. 
\item $Y$ and $W$ are independent.
\end{enumerate}
\end{lemma}
\begin{proof}
Tightness follows from standard arguments. Let $(X,Y,W,\beta)$ denote any limit point. Clearly (iii) holds. We first check that (i) holds by showing that the law of $Y^n$ converges weakly to $\frac12\delta_{H^+_0} + \frac12\delta_{H^-_0}$. Suppose that $(X^n,Y^n,W^1,\ldots,^n,\beta^n)$ are defined on the filtered probability space probability space $(\Omega^n,\F^n,\FF^n,\PP^n)$. Define the Brownian motion $\overline{W}_t = \frac{1}{\sqrt{n}}\sum_{i=1}^nW^i_t$.
On this space, let $Z^n$ denote the unique strong solution of the SDE
\[
dZ^n_t = \mathrm{sgn}(Z^n_t) dt + \frac{1}{\sqrt{n}}d\overline{W}_t.
\]
We know from \cite{trevisan2013zero} that $\PP^n \circ (Z^n)^{-1} \rightarrow \frac12\delta_{H^+_0} + \frac12\delta_{H^-_0}$. 
Define an equivalent probability measure $\QQ^n$ by setting
\begin{align*}
\frac{d\QQ^n}{d\PP^n}  &= \exp\Bigg(\frac{1}{\sqrt{n}}\int_0^T\left(\beta^n_t - \mathrm{sgn}(Z^n_t) \right) d\overline{W}_t  - \frac{1}{2n}\int_0^T\left(\beta_t - \mathrm{sgn}(Z^n_t) \right)^2dt\Bigg).
\end{align*}
By Girsanov's theorem and uniqueness in law of the SDEs, we have $\QQ^n \circ (Z^n)^{-1} = \PP^n \circ (Y^n)^{-1}$. This yields the following bound on relative entropy:
\begin{align*}
\E^{\PP^n}\left[\frac{d\QQ^n}{d\PP^n} \log \frac{d\QQ^n}{d\PP^n}\right] &= -\E^{\QQ^n}\left[ \log \frac{d\PP^n}{d\QQ^n}\right] = \frac{1}{2n}\E^{\QQ^n}\int_0^T\left(\beta_t - \mathrm{sgn}(Z^n_t) \right)^2dt  \le \frac{2T}{n}.
\end{align*}
By Pinsker's inequality, the total variation norm of $\QQ^n - \PP^n$ converges to zero. Because $\PP^n \circ (Z^n)^{-1} \rightarrow \frac12\delta_{H^+_0} + \frac12\delta_{H^-_0}$, we conclude that also $\QQ^n \circ (Z^n)^{-1} \rightarrow \frac12\delta_{H^+_0} + \frac12\delta_{H^-_0}$. 
Recalling that $\QQ^n \circ (Z^n)^{-1} = \PP^n \circ (Y^n)^{-1}$, this completes the proof of (i).

With (i) now established, we prove (ii). It is clear that $X_t = \int_0^t\beta_sds + W_t$ holds, because $X^n_t = \int_0^t\beta^n_sds + W^n_t$ for each $n$ and because $L^2_1 \ni q \mapsto \int_0^t q_sds \in \R$ is (weakly) continuous for each $t$. Finally, note that (i) implies that $Y_t = \int_0^t\mathrm{sgn}(Y_s) ds$ for all $t$.

To check property (iii), note that the law of $W$ is clearly equal to Wiener measure, so we must only show that $W_t-W_s$ is independent of $\F_s$ for each $t > s \ge 0$. This argument is straightforward and thus omitted.

We finally show that (iv) follows from the other claims. Because $W$ is $\FF$-Brownian, it is also $\FF_+$-Brownian, where $\FF_+=(\F_{t+})_{t \in [0,T]}$ denotes the right-continuous augmentation, defined by $\F_t = \cap_{\epsilon > 0}\F_{t+\epsilon}$. In particular, $W$ is independent of $\F_{0+}$. Now, from (i) we may write $Y_t = t\,\mathrm{sgn}(Y_T)$ a.s., from which we conclude that the entire process $Y$ is a.s.-measurable with respect to $\F_{0+}$. Hence, $Y$ and $W$ are independent.
\end{proof}

With this Lemma in hand, we now complete the proof of Proposition \ref{pr:ex:weakMFEconverse}.
Working with a subsequence of $(X^n,Y^n,W^1,\beta^n)$ and its limit $(X,Y,W,\beta)$, we have
\begin{align*}
\lim_n  J^n_1(\beta^n,\alpha_0^{n,2},\ldots,\alpha_0^{n,n}) &= \lim_n \E[X^n_TY^n_T] \\
	&= \E[X_TY_T] = \E\left[Y_T\int_0^T\beta_tdt + Y_TW_T\right] = \E\left[Y_T\int_0^T\beta_tdt\right] \\
	&\le T\E|Y_T| = T^2,
\end{align*}
with the limit taken along the appropriate subsequence. Note that the second equality is valid in light of the simple estimate $\sup_{n \in \N}\E[|X^n_TY^n_T|^p] < \infty$ for any $p  > 1$, which provides the uniform integrability needed to pass to the limit. Finally, because this holds for each convergent subsequence, we conclude finally from
\begin{align*}
\limsup_{n\rightarrow\infty}J^n_1(\beta^n,\alpha_0^{n,2},\ldots,\alpha_0^{n,n}) \le T^2.
\end{align*}
Recalling  \eqref{pf:ex:weakMFEconverse1} and \eqref{pf:ex:weakMFEconverse1-1}, this completes the proof.\hfill\qedsymbol

\appendix

\section{SDEs with random coefficients} \label{ap:SDErandomcoeff}

This section develops some intuitively clear but somewhat delicate technical points regarding SDEs with random coefficients. 
It will be useful to write $\FF^E=(\F^E_t)_{t \in [0,T]}$ for the canonical filtration on the path space $C([0,T];E)$, defined for any Polish space $E$. 

For the rest of the section, fix a complete separable metric space $E$ (which in applications in this paper will be $E=\P(\R^d)$). As in Definition \ref{def:semiMarkovFunction}, let us say that a function $B : [0,T] \times \R^d \times \CE \rightarrow \R^d$ is semi-Markov if it is Borel measurable and satisfies $F(t,x,e)=F(t,x,e')$ whenever $(t,x) \in [0,T] \in \R^d$ and $e,e' \in \CE$ satisfy $e_s=e'_s$ for all $s \le t$.
Fix throughout the section one such semi-Markov function $B$, which we assume is bounded. Equip $\CE$ with the supremum distance.
We fix also a complete filtered probability space $(\Omega,\F,\FF,\PP)$ supporting a $d$-dimensional $\FF$-Brownian motion $W$ as well as an $\F_0$-measurable $\R^d$-valued random variable $\xi$ with law $\lambda$.

The goal of this section is to justify the following points:
\begin{enumerate}[(1)]
\item \textbf{Deterministic well-posedness:} For a deterministic $e \in \CE$, there is a unique strong solution of the SDE
\begin{align}
dX^e_t = B(t,X^e_t,e)dt + dW_t, \quad\quad X^e_0=\xi. \label{def:ap:SDE-deterministic}
\end{align}
Let $P^e \in \P(\C^d)$ denote its law. By ``strong solution" here we mean $X^e$ is adapted to the complete filtration generated by the process $(\xi,W_t)_{t \in [0,T]}$.
\item \textbf{Stochastic well-posedness:} If $\eta$ is a $\CE$-valued random variable with law $M$, independent of $(\xi,W)$, then there is a unique strong solution of the SDE
\begin{align}
dX_t = B(t,X_t,\eta)dt + dW_t, \quad\quad X_0=\xi. \label{def:ap:SDE-stochastic}
\end{align}
By ``strong solution" we mean $X$ is adapted to the complete filtration generated by the process $(\xi,W_t,\eta_t)_{t \in [0,T]}$.
\item \textbf{Consistency:} The map $\CE \ni e \mapsto P^e \in \P(\C^d)$ is universally measurable and, in the notation of part (2), provides a version of the conditional law of $X$ given $\eta$. That is, for each bounded measurable function $\varphi$ on $\CE \times \C^d$, we have
\[
\E[\varphi(\eta,X)] = \int_{\CE}M(de)\int_{\C^d}P^e(dx)\varphi(e,x).
\]
\item \textbf{Stability:} Given a uniformly bounded sequence of semi-Markov functions $B_n : [0,T] \times \R^d \times \CE \rightarrow \R^d$ satisfying $B_n(t,x,e) \rightarrow B(t,x,e)$ for $M$-a.e.\ $e$ and Lebesgue-a.e.\ $(t,x)$, we have
\[
\lim_{n\rightarrow\infty}\E[\varphi(\eta,X^n)] = \E[\varphi(\eta,X)]
\]
for each bounded measurable function $\varphi : \CE \times \C^d \rightarrow \R$, where $X^n$ is the unique strong solution of 
\begin{align}
dX^n_t = B_n(t,X^n_t,\eta)dt + dW_t, \quad\quad X^n_0=\xi. \label{def:ap:SDE-stochastic-n}
\end{align}
\item \textbf{Equivalence to forward equations:} Suppose a continuous $\P(\R^d)$-valued process $\mu=(\mu_t)_{t \in [0,T]}$ is a weak solution of the randomized Fokker-Planck equation associated to \eqref{def:ap:SDE-stochastic}. Precisely, suppose $\mu$ is adapted to the filtration generated by $\eta$, and it holds almost surely that, for all $t \in [0,T]$ and $\varphi \in C^\infty_c(\R^d)$,
\begin{align*}
\langle \mu_t,\varphi\rangle = \langle \lambda,\varphi\rangle + \int_0^t\left\langle \mu_s, \, B(s,\cdot,\eta) \cdot \nabla\varphi(\cdot) + \tfrac12\Delta\varphi(\cdot)\right\rangle ds.
\end{align*}
Then $\mu_t = \L(X_t \, | \, (\eta_s)_{s \le t})$ a.s., for each $t$, where $X$ is as in \eqref{def:ap:SDE-stochastic}.
\end{enumerate}
These results are applied in the text in the particular case $E=\P(\R^d)$, and with $\mu=\eta$ in step (5), but we find it clearer and perhaps useful on its own to work in this more general setting.

\subsection{Deterministic well-posedness} \label{se:ap:deterministic}
Part (1) of the program follows from the result of Veretennikov \cite{veretennikov1981strong} (see also \cite[Theorem 2.1]{krylov-rockner}). That is, for each $e \in \CE$ there exists a unique strong solution $X^e$ of the SDE \eqref{def:ap:SDE-deterministic}.
Let $P^e = \PP \circ (X^e)^{-1}$. In particular, pathwise uniqueness holds for this SDE, in the following sense: Suppose our probability space $(\Omega,\F,\FF,\PP)$ supports two continuous $\FF$-adapted processes $X^1,X^2$ which both satisfy
\[
dX^i_t = B(t,X^i_t,e)dt + dW_t, \quad X^i_0=\xi, \quad i=1,2,
\]
and also as usual the process $W$ is an $\FF$-Brownian motion independent of $\xi$. Then $X^1=X^2$ a.s., and the law of $X^1$ is precisely $P^e$.

We would like to be able to construct a version of $(t,\omega,e) \mapsto X^e_t(\omega)$ which is jointly measurable and which depends in an adapted fashion on $e$, but it is not clear how to do this. Uniqueness of the strong solution $X^e$ easily yields $\PP(X^e_s=X^{\tilde{e}}_s, \ \forall s \le t)=1$ whenever $t \in [0,T]$ and $e,\tilde{e} \in \CE$ satisfy $e_s=\tilde{e}_s$ for all $s \le t$. But the null set depends on $(t,e,\tilde{e})$, and we thus face a continuum of null sets. There is no continuity in $e$ to exploit, as we have made no continuity assumptions on $B$, and this is the main technical impediment to our program (1-5). Instead, we work with the law $P^e$ instead of the process $X^e$ itself.

In the following, let $\overline\FF^E=(\overline\F^E_t)_{t \in [0,T]}$ denote the universal completion of $\FF^E$. Precisely, if $\mathcal{N}^P$ denotes the set of $P$-null sets of the Borel $\sigma$-field on $\CE$, then
\[
\overline\F^E_t := \bigcap_{P \in \P(\CE)}\sigma(\F^E_t \cup \mathcal{N}^P).
\]

\begin{lemma} \label{le:Pe-adapted}
The map $\CE \ni e \mapsto P^e \in \P(\C^d)$ is universally measurable. Moreover, this map is \emph{adapted} in the sense that, for every $t \in [0,T]$ and every $S \in \F^{\R^d}_t$, the map $e \mapsto P^e(S)$ is $\overline\F^E_t$-measurable.
\end{lemma}
\begin{proof}
If $t \in [0,T]$, then uniqueness of the SDE ensures that if $e_s=\tilde{e}_s$ for $s \le t$ then $X^e_s = X^{\tilde{e}}_s$ for all $s \le t$, a.s. Hence, if $S \in \F^{\R^d}_t$, then $P^e(S) = \PP(X^e \in S) = \PP(X^{\tilde{e}} \in S) = P^{\tilde{e}}(S)$, and we deduce that the second claim will follow from the first.

First suppose that $B(t,x,e)$ is continuous in $e$ for each $(t,x)$. We claim that then $e \mapsto P^e$ is continuous. To see this, suppose $e^n \rightarrow e$ in $\CE$. It then holds for bounded continuous function $\varphi : [0,T] \times \R^d \rightarrow \R^d$ with compact support that
\begin{align*}
\lim_{n\rightarrow\infty}\int_0^T\int_{\R^d}B(t,x,e^n) \cdot \varphi(t,x)dxdt = \int_0^T\int_{\R^d}B(t,x,e) \cdot \varphi(t,x)dxdt.
\end{align*}
It follows from \cite[Theorem 11.3.3]{stroock-varadhan} that $P^{e^n} \rightarrow P^e$. 

We now address general $B$ by an approximation argument.
Fix a probability measure $M \in \P(\CE)$. Define the finite measure $Q$ on $[0,T] \times \R^d \times E$ by setting, for Borel sets $S$,
\[
Q(S) = \int_{\CE}\int_{\R^d}\int_0^T 1_S(t,x,e)\exp(-|x|^2)dtdxM(de).
\]
We may then find a sequence of continuous semi-Markov functions $B_n$ which converges $Q$-almost everywhere to $B$.
Define $P_n^e$ as the law of the corresponding SDE solution, i.e., $P^e_n = \PP \circ (X^{n,e})^{-1}$ where $X^{n,e}$ is given by
\[
dX^{n,e}_t = B_n(t,X^{n,e}_t,e)dt + dW_t, \quad X^{n,e}_0 = \xi.
\]
As argued in the previous paragraph, $e \mapsto P^e_n$ is continuous for each $n$. Moreover, it holds for $M$-almost every $e \in \CE$ that 
\begin{align*}
\lim_{n\rightarrow\infty}\int_0^T\int_{\R^d}B_n(t,x,e) \cdot \varphi(t,x)dxdt = \int_0^T\int_{\R^d}B(t,x,e) \cdot \varphi(t,x)dxdt
\end{align*}
for each bounded continuous function $\varphi : [0,T] \times \R^d \rightarrow \R^d$ with compact support.
It follows again from \cite[Theorem 11.3.3]{stroock-varadhan} that $P^e_n \rightarrow P^e$ for $M$-a.e.\ $e \in \CE$. Hence, the map $e \mapsto P^e$ agrees $M$-a.e.\ with a Borel measurable function, so it is measurable with respect to the $M$-completion of the Borel $\sigma$-field of $\CE$. As this holds for every choice of $M \in \P(\CE)$, the proof is complete.
\end{proof}

\subsection{Stochastic well-posedness}

We now turn to steps (2) and (3) of the program outlined at the beginning of the section, by proving weak existence and pathwise uniqueness for the SDE \eqref{def:ap:SDE-stochastic} and then identifying the law of the unique solution as $\L(\eta,X) = M(de)P^e(dx)$. As the SDE \eqref{def:ap:SDE-stochastic} has random coefficients, the original form of the Yamada-Watanabe theorem does not apply, and we instead use the generalization due to Jacod-M\'emin \cite{jacodmemin1981} to conclude, as usual, that weak existence and pathwise uniqueness are together equivalent to uniqueness in law and existence of a strong solution.
The first lemma checks that the SDE \eqref{def:ap:SDE-stochastic} is pathwise unique and identifies its law (rather, it satisfies \emph{very good pathwise uniqueness} in the language of \cite[Definition 2.24]{jacodmemin1981}).

\begin{lemma} \label{le:ap:SDErandom-uniqueness}
Let $M \in \P(\CE)$. Suppose our filtered probability space $(\Omega,\F,\FF,\PP)$ supports an $\FF$-adapted continuous $E$-valued process $\eta$ with law $M$, independent of $(\xi,W)$, as well as two $d$-dimensional $\FF$-adapted processes $(X^1,X^2)$ satisfying
\[
dX^i_t = B(t,X^i_t,\eta)dt + dW_t, \quad X^i_0 =\xi, \quad i=1,2.
\]
Define $\F^{\xi,\eta,W}_t := \sigma(\xi,\eta_s,W_s : s \le t)$, and assume that $(X^1_s,X^2_s)_{s \in [0,t]}$ is conditionally independent of $\F^{\xi,\eta,W}_T$ given $\F^{\xi,\eta,W}_t$, for each $t \in [0,T]$. 
Then $X^1=X^2$ a.s., and $\PP \circ (\eta,X^i)^{-1} = M(de)P^e(dx)$ for each $i=1,2$.
\end{lemma}
\begin{proof}
By assumption, $\PP \circ \eta^{-1} = M$. Let us show that $e \mapsto P^e$ is a version of the conditional law $\PP(X^i \in \cdot \, | \, \eta=e)$. Define the regular conditional law $\CE \ni e \mapsto \QQ_e = \PP(\cdot \, | \, \eta=e) \in \P(\Omega)$. Because $W$, $\xi$, and $\eta$ are independent, we have $\QQ_e \circ (\xi,W)^{-1} = \PP \circ (\xi,W)^{-1}$ for $M$-a.e.\ $e$. Moreover, the SDE
\[
X^i_t = \xi + \int_0^tB(s,X^i_s,e)ds + W_t, \ \ \forall t \in [0,T],
\]
holds almost surely under $\QQ_e$, for $M$-a.e.\ $e$. 

We would like to conclude from pathwise uniqueness (see the first paragraph of Section \ref{se:ap:deterministic}) that $\QQ_e(X^1=X^2)=1$ and $\QQ_e \circ (X^i)^{-1} = P^e$ for $M$-a.e.\ $e$. To do so we need only to show that $W$ is an $\widetilde\FF$-Brownian motion under $\QQ_e$, for $M$-a.e.\ $e$, where $\widetilde\FF=(\widetilde\F_t)_{t \in [0,T]}$ denotes the filtration (on $\Omega$) generated by $X^1$, $X^2$, and $W$, i.e., $\widetilde\F_t = \sigma(X_s,W_s : s \le t)$. This amounts to proving that for each $t \in [0,T]$, each $\sigma(X^1_s,X^2_s : s \le t)$-measurable random variable $\varphi_t(X)$, each $\sigma(W_s : s \le t)$-measurable random variable $h_t(W)$, and each $\sigma(W_s - W_t : s \in [t,T])$-measurable random variable $h_{t+}(W)$, we have
\[
\E^{\QQ_e}[\varphi_t(X)h_t(W)h_{t+}(W)] = \E^{\QQ_e}[\varphi_t(X)h_t(W)]\E^{\QQ_e}[h_{t+}(W)].
\]
To prove this, notice that if $\varphi : \CE \rightarrow \R$ is any bounded measurable function, then (taking expectations under $\PP$)
\begin{align*}
\E[\varphi(\eta)\varphi_t(X)h_t(W)h_{t+}(W)] &= \E[\E[\varphi_t(X)|\F^{\xi,\eta,W}_T]\varphi(\eta)h_t(W)h_{t+}(W)] \\
	&= \E[\E[\varphi_t(X)|\F^{\xi,\eta,W}_t]\varphi(\eta)h_t(W)h_{t+}(W)] \\
	&= \E[\E[\varphi_t(X)|\F^{\xi,\eta,W}_t]\varphi(\eta)h_t(W)]\E[h_{t+}(W)] \\
	&= \E[\varphi(\eta)\varphi_t(X)h_t(W)]\E[h_{t+}(W)]
\end{align*}
Indeed, the second and final lines follow from the assumed conditional independence of $(X^1_s,X^2_s)_{s \le t}$ and $\F^{\xi,\eta,W}_T$ given $\F^{\xi,\eta,W}_t$, while the second to last identity follows from the fact that $(\xi,\eta,(W_s)_{s \le t})$ and $(W_s-W_t)_{s \ge t}$ are independent, which is an easy consequence of the independence of $\xi$, $\eta$, and $W$. We conclude that
\begin{align*}
\E\left[\left. \varphi_t(X)h_t(W)h_{t+}(W) \,  \right| \,  \eta \, \right]= \E\left[\left. \varphi_t(X)h_t(W) \, \right| \, \eta \, \right]\E[h_{t+}(W)], \ \ a.s.,
\end{align*}
which completes the proof.
\end{proof}

Now that we have checked pathwise uniqueness, we turn to the problem of existence. The following lemma shows that we can construct $(\eta,X)$ with law $M(de)P^e(dx)$ so that the SDE \eqref{def:ap:SDE-stochastic} does indeed hold, as well as the conditional independence property of Lemma \ref{le:ap:SDErandom-uniqueness}. This will be enough to deduce strong existence, using a form of the Yamada-Watanabe theorem \cite[Theorem 2.25]{jacodmemin1981}.

\begin{lemma} \label{le:ap:SDErandom-existence}
Let $M \in \P(\CE)$. Suppose our filtered probability space $(\Omega,\F,\FF,\PP)$ supports an $\FF$-adapted continuous $E$-valued process $\eta$ with law $M$, independent of $(\xi,W)$. Then there exists a continuous $\FF$-adapted process $X$ solving
\[
dX_t = B(t,X_t,\eta)dt + dW_t, \quad X_0 =\xi,
\]
such that $\PP \circ (\eta,X)^{-1} = M(de)P^e(dx)$.
In particular, $X$ is adapted to the complete filtration generated by the process $(\xi,\eta_t,W_t)_{t \in [0,T]}$.
\end{lemma}
\begin{proof}
Following the strategy described above, we begin by building a weak solution. We work on the canonical space $\overline\Omega = \CE \times \C^d$. Let $(\eta,X)$ denote the canonical (coordinate) processes, and let $\overline\FF=(\overline\F_t)_{t \in [0,T]}$ denote the filtration they generate, which can be written as $\overline\F_t = \F^E_t \otimes \F^{\R^d}_t$.
Define $\overline\PP(de,dx) = M(de)P^e(dx)$.
For each $t \in [0,T]$, define $W_t : \overline\Omega \rightarrow \R$ by
\begin{align*}
W_t(e,x) = x_t - x_0 - \int_0^tB(s,x_s,e)ds,
\end{align*}
and define $\xi :=X_0$.
The process $W=(W_t)_{t \in [0,T]}$ is $\overline\FF$-progressively measurable with respect to the canonical filtration. Note that $W(e,\cdot)$ is a Brownian motion on $(\C^d,\FF^{\R^d},P^e)$, for each $e \in \CE$, by definition of $P^e$.  It follows easily that $W$ is an $\overline\FF$-Brownian motion under $\overline\PP$. Moreover, $\xi$, $\eta$, and $W$ are independent.
By construction, the SDE holds,
\[
dX_t = B(t,X_t,\eta)dt + dW_t, \quad\quad X_0=\xi.
\]

We will show that $(X_s)_{s \le t}$ is conditionally independent of $\F^{\xi,\eta,W}_T$ given $\F^{\xi,\eta,W}_t$ for each $t$, where $\F^{\xi,\eta,W}_t := \sigma(\xi,\eta_s,W_s : s \le t)$.
To prove this, fix $t \in [0,T]$ as well as random variables  $h_t(W)$, $h_{t+}(W)$, $\varphi_0(\xi)$, $\varphi_t(X)$, $\psi_t(\eta)$, and $\psi_T(\eta)$, measurable with respect to $(W_s)_{s \le t}$, $(W_s - W_t)_{s \in [t,T]}$, $\xi$, $(X_s)_{s \le t}$, $(\eta_s)_{s \le t}$, and $\eta$, respectively. Then, by definition of $\overline\PP$,
\begin{align}
\E^{\overline\PP} &\left[h_t(W) h_{t+}(W) \varphi_0(\xi)\varphi_t(X) \psi_t(\eta) \psi_T(\eta)\right] \nonumber \\
	&= \int_{\CE}M(de)\psi_t(e)\psi_T(e) \int_{\C^d}P^e(dx)\varphi_0(x_0)\varphi_t(x)h_t(W(e,x))h_{t+}(W(e,x)) \nonumber \\
	&= \langle \W, h_{t+}\rangle \int_{\CE}M(de)\psi_t(e)\psi_T(e) \int_{\C^d}P^e(dx)\varphi_0(x_0)\varphi_t(x)h_t(W(e,x)), \label{pf:ap:existence11}
\end{align} 
where $\W$ denotes Wiener measure on $\C^d$, and where the last line used the fact that $W(e,\cdot)$ is a Brownian motion on $(\C^d,\FF^{\R^d},P^e)$, mentioned above. Now, the independence of $\xi$, $\eta$, and $W$ easily implies
\begin{align*}
\E^{\overline\PP}[ \psi_T(\eta) \, | \, \F^{\xi,\eta,W}_t ] = \E^{\overline\PP}[ \psi_T(\eta) \, | \, \F^{\eta}_t ] =: \widetilde\psi_t(\eta).
\end{align*}
Moreover, because the function $\C^d \ni x \mapsto \varphi_0(x_0)\varphi_t(x)h_t(W(e,x))$ is $\F^{\R^d}_t$-measurable for each fixed $e$, the adaptedness of $e \mapsto P^e$ proven in Lemma \ref{le:Pe-adapted} implies that the function
\[
\CE \ni e \mapsto \int_{\C^d}P^e(dx)\varphi_0(x_0)\varphi_t(x)h_t(W(e,x))
\]
is $\overline\F^E_t$-measurable, and in particular it agrees $M$-a.e.\ with an $\F^E_t$-measurable function. Hence, we may conditionin on $\F^\eta_t$ on the right-hand side of \eqref{pf:ap:existence11} to get
\begin{align*}
\langle \W, h_{t+}\rangle \int_{\CE}M(de)\psi_t(e)\widetilde\psi_t(e) \int_{\C^d}P^e(dx)\varphi_0(x_0)\varphi_t(x)h_t(W(e,x)),
\end{align*} 
and we deduce from \eqref{pf:ap:existence11} that
\begin{align*}
\E^{\overline\PP}&\left[h_t(W) h_{t+}(W) \varphi_0(\xi)\varphi_t(X) \psi_t(\eta) \psi_T(\eta)\right] = \E^{\overline\PP}[h_{t+}(W)] \E^{\overline\PP}  \left[h_t(W) \varphi_0(\xi)\varphi_t(X) \psi_t(\eta) \widetilde\psi_t(\eta)\right].
\end{align*} 
Finally, recall that
\begin{align*}
\widetilde\psi_t(\eta)\E^{\overline\PP}[h_{t+}(W)] = \E^{\overline\PP}[ \psi_T(\eta) \, | \, \F^{\xi,\eta,W}_t ]\E^{\overline\PP}[ h_{t+}(W) \, | \, \F^{\xi,\eta,W}_t ],
\end{align*} 
which yields
\begin{align*}
\E^{\overline\PP} &\left[h_t(W) h_{t+}(W) \varphi_0(\xi)\varphi_t(X) \psi_t(\eta) \psi_T(\eta)\right] \\
	&= \E^{\overline\PP}  \left[h_t(W) \varphi_0(\xi)\varphi_t(X) \psi_t(\eta) \E^{\overline\PP}[ \psi_T(\eta) \, | \, \F^{\xi,\eta,W}_t ]\E^{\overline\PP}[ h_{t+}(W) \, | \, \F^{\xi,\eta,W}_t ]\right] \\
	&= \E^{\overline\PP}  \left[h_t(W)  \psi_t(\eta)\varphi_0(\xi) \E^{\overline\PP}[ \varphi_t(X) \, | \, \F^{\xi,\eta,W}_t]\E^{\overline\PP}[ \psi_T(\eta) \, | \, \F^{\xi,\eta,W}_t ]\E^{\overline\PP}[ h_{t+}(W) \, | \, \F^{\xi,\eta,W}_t ]\right].
\end{align*} 
This proves the desired conditional independence (and in fact a bit more).

Finally, we complete the proof in the manner announced before the statement of the lemma. It is well known (see, e.g., \cite[Theorem 3]{bremaud1978changes}) that the following are equivalent:
\begin{enumerate}
\item Every $\FF^{\xi,\eta,W}$-martingale is an $\FF^{\xi,\eta,W,X}$-martingale, where  $\FF^{\xi,\eta,W,X}=(\F^{\xi,\eta,W,X}_t)_{t \in [0,T]}$ is defined by $\F^{\xi,\eta,W,X}_t = \sigma(\xi,\eta_s,W_s,X_s : s \le t)$.
\item $(X_s)_{s \le t}$ is conditionally independent of $\F^{\xi,\eta,W}_T$ given $\F^{\xi,\eta,W}_t$ for each $t \in [0,T]$.
\end{enumerate}
This shows that our conditional independence property is in fact equivalent to the notion of \emph{very good solution measure} in \cite[Definition 1.7]{jacodmemin1981}. Thus, by \cite[Theorem 2.25]{jacodmemin1981}, we conclude that the solution measure is in fact strong, which means in our context that $X$ must be adapted with respect to the $\overline\PP$-completion of $\FF^{\xi,\eta,W}$.
\end{proof}

\subsection{Stability}

We turn next to part (4) of the outline from the beginning of the section. Suppose $B_n : [0,T] \times \R^d \times \CE \rightarrow \R^d$ is semi-Markov, for each $n$. Assume that $B_n$ are uniformly bounded and that $B_n(t,x,e) \rightarrow B(t,x,e)$, for $M$-a.e.\ $e \in \CE$ and Lebesgue-a.e.\ $(t,x) \in [0,T] \times \R^d$. Thanks to the work of the previous section, we may define $X^n$ as the unique strong solution of the SDE \eqref{def:ap:SDE-stochastic-n} corresponding to coefficient $B_n$. That is,
\[
dX^n_t = B_n(t,X^n_t,\eta)dt + dW_t, \quad X^n_0=\xi.
\]

\begin{lemma} \label{le:ap:approximation}
Let $M \in \P(\CE)$. Suppose our filtered probability space $(\Omega,\F,\FF,\PP)$ supports an $\FF$-adapted continuous $E$-valued process $\eta$ with law $M$, independent of $(\xi,W)$. Then, for every bounded measurable function $h : \CE \times \C^d \rightarrow \R$, we have
\begin{align}
\lim_{n\rightarrow\infty}\E[\varphi(\eta,X^n)] = \E[\varphi(\eta,X)] \label{def:ap:approximation}
\end{align}
In particular, $(\eta,X^n)$ converges in law to $(\eta,X)$.
\end{lemma}
\begin{proof}
From Lemmas \ref{le:ap:SDErandom-uniqueness} and \ref{le:ap:SDErandom-existence}, we know that $\PP \circ (\eta,X^n)^{-1} = M(de)P_n^e(dx)$, where we define $P_n^e := \PP \circ (X^{n,e})^{-1}$ as the law of the unique strong solution of the SDE
\[
dX^{n,e}_t = B_n(t,X^{n,e}_t,e)dt + dW_t, \quad\quad X^{n,e}_0=\xi.
\]
By assumption, for any bounded continuous function $\varphi : [0,T] \times \R^d \rightarrow \R^d$ with compact support, we have
\begin{align*}
\lim_{n\rightarrow\infty}\int_0^T\int_{\R^d}B_n(t,x,e) \cdot \varphi(t,x)dxdt = \int_0^T\int_{\R^d}B(t,x,e) \cdot \varphi(t,x)dxdt,
\end{align*}
for $M$-a.e.\ $e$. It follows from \cite[Theorem 11.3.3]{stroock-varadhan} that $P^e_n \rightarrow P^e$ for $M$-a.e.\ $e$. It follows immediately that $M(de)P^e_n(dx) \rightarrow M(de)P^e(dx)$ weakly. To prove that convergence holds for bounded measurable test functions, we need a bit more. Let $\W_\lambda \in \P(\C^d)$ denote  the law of a $d$ Brownian motion started from initial law $\lambda$. Then
\begin{align*}
\frac{dP^e_n}{d\W_\lambda}(w) = \exp\left( \int_0^T B_n(t,w_t,e) dw_t - \frac12\int_0^T|B_n(t,w_t,e)|^2dt\right).
\end{align*}
Because $B_n$ is uniformly bounded, it is straightforward to show that
\begin{align*}
\sup_{e \in \CE}\sup_{n \in \N}\int_{\C^d}\left|\frac{dP^e_n}{d\W_\lambda}\right|^2\,d\W_\lambda < \infty.
\end{align*}
This implies that the family $\{dP^e_n/d\W_\lambda : n \in \N, \, e \in \CE\}$ is precompact in $L^2(\W_\lambda)$ with the weak topology, and this is enough to let us upgrade the convergence. Indeed, we conclude that
\[
\lim_{n\rightarrow\infty}\int_{\C^d} h\,dP^e_n =\int_{\C^d} h\,dP^e, \quad \forall e \in \CE,
\]
not only for bounded continuous functions $h : \C^d \rightarrow \R$ but also for bounded measurable functions. Finally, if $h : \CE \times \C^d \rightarrow \R$ is bounded and measurable, we conclude from dominated convergence that
\[
\lim_{n\rightarrow\infty}\int_{\CE}\int_{\C^d} h(e,x)\,P^e_n(dx)\,M(de) = \int_{\CE}\int_{\C^d} h(e,x)\,P^e(dx)\,M(de).
\]
This is equivalent to the claimed \eqref{def:ap:approximation}.
\end{proof}

\subsection{Forward equations}
Let $E$ and $B$ be as in the previous section. Let $\lambda$ denote the law of the initial state $\xi$. Consider the problem of finding $(m_t)_{t \in [0,T]} \in \CP $ such that 
\begin{align}
\langle m_t,\varphi\rangle = \langle \lambda,\varphi\rangle + \int_0^t\langle m_s, \, B(s,\cdot,e) \cdot \nabla\varphi(\cdot) + \tfrac12\Delta\varphi(\cdot)\rangle ds, \label{def:ap:forwardEQ}
\end{align}
for all $t \in [0,T]$ and $\varphi \in C^\infty_c(\R^d)$.
This is nothing but the Fokker-Planck equation associated with the SDE \eqref{def:ap:SDE-deterministic}.
One solution is provided by the marginal flow $(P^e_t=\L(X^e_t))_{t \in [0,T]}$, and the following gives uniqueness.

\begin{lemma} \label{le:forwardeq}
Fix $e \in \CE$, and suppose $m \in \CP $ satisfies \eqref{def:ap:forwardEQ} for every $t \in [0,T]$ and $\varphi \in C^\infty_c(\R^d)$. Then $m_t = P^e_t$ for all $t \in [0,T]$.
\end{lemma}
\begin{proof}
It is well known that the solution of a Fokker-Planck equation, in very general settings, can be represented as the marginal laws of a solution of the corresponding martingale problem. See, e.g., \cite[Theorem 2.6]{figalli2008existence} or \cite[Theorem 2.5]{trevisan2016well}. In our context, the martingale problem has a unique solution given by $P^e$, and the claim follows.
\end{proof}

\begin{lemma} \label{le:forwardeq-random}
Suppose our filtered probability space $(\Omega,\F,\FF,\PP)$ supports an $\FF$-adapted continuous $E$-valued process $\eta$, independent of $(\xi,W)$, as well as a continuous $\P(\R^d)$-valued process $\mu$ which is adapted to the filtration generated by $\eta$. Suppose it holds almost surely that, for all $t \in [0,T]$ and $\varphi \in C^\infty_c(\R^d)$,
\begin{align}
\langle \mu_t,\varphi\rangle = \langle \lambda,\varphi\rangle + \int_0^t\left\langle \mu_s, \, B(s,\cdot,\eta) \cdot \nabla\varphi(\cdot) + \tfrac12\Delta\varphi(\cdot)\right\rangle ds, \label{def:ap:forwardEQ-random}
\end{align}
Then $\mu_t = P^\eta_t$ for all $t \in [0,T]$ a.s. Moreover, we may find a continuous process $X$, adapted to the complete filtration generated by the process $(\xi,W_t,\eta_t)_{t \in [0,T]}$, such that
\begin{align}
dX_t = B(t,X_t,\eta)dt + dW_t, \quad\quad X_0=\xi, \label{def:ap:forwardEQ-SDE-random}
\end{align}
and also $\L(X_t \, | \, \eta) = \L(X_t \, | \, (\eta_s)_{s \le t}) = \mu_t$ a.s.\ for each $t$.
\end{lemma}
\begin{proof}
Let $M \in \P(\CE)$ denote the law of $\eta$.
As we assumed $\mu$ is adapted to the filtration of $\eta$, we may write $\mu = \widehat{\mu}(\eta)$ a.s., where $\widehat{\mu} : \CE \rightarrow \CP$ is an adapted map in the sense that $\widehat{\mu}^{-1}(S) \in \F^{\P(\R^d)}_t$ for each $S \in \F^E_t$ and each $t \in [0,T]$. Then because of \eqref{def:ap:forwardEQ-random}, for $M$-a.e.\ $e \in \CE$ and every $\varphi \in C^\infty_c(\R^d)$ and $t \in [0,T]$ it holds that
\begin{align*}
\langle \widehat{\mu}_t(e),\varphi\rangle = \langle \lambda,\varphi\rangle + \int_0^t\left\langle \widehat{\mu}_s(e), \, B(s,\cdot,e) \cdot \nabla\varphi(\cdot) + \tfrac12\Delta\varphi(\cdot)\right\rangle ds.
\end{align*}
From Lemma \ref{le:forwardeq} we conclude that $\widehat{\mu}_t(e) = P^e_t$ for each $t \in [0,T]$. As this holds for almost every $e$, we deduce the first claim: $\mu_t = P^\eta_t$ for all $t$, a.s.

Now, using Lemmas \ref{le:ap:SDErandom-uniqueness} and \ref{le:ap:SDErandom-existence}, we may safely define $X$ to be the unique strong solution of the SDE \eqref{def:ap:forwardEQ-SDE-random}, and we know that $\PP \circ (\eta,X)^{-1} = M(de)P^e(dx)$. This last identity is equivalent to $\L(X \, | \, \eta) = P^\eta$ a.s. Marginalizing at time $t$ and using the conclusion of the previous paragraph, we find $\L(X_t \, | \, \eta) = P^\eta_t = \mu_t$ for all $t$, a.s.

Lastly, to deduce that $\L(X_t \, | \, \eta) = \L(X_t \, | \, (\eta_s)_{s \le t})$, note that for any bounded measurable $\varphi : \R^d \rightarrow \R$ we have $\E[\varphi(X_t) \, | \, \eta ] = \langle \mu_t,\varphi\rangle$. As $\mu_t$ is $(\eta_s)_{s \le t}$-measurable, we may condition on $(\eta_s)_{s \le t}$ to get $\E[\varphi(X_t) \, | \, (\eta_s)_{s \le t} ] = \langle \mu_t,\varphi\rangle$.
\end{proof}

We finally note how Lemma \ref{le:forwardeq-random} specializes in the most important situation for this paper, where $E=\P(\R^d)$ and the processes $\eta$ and $\mu$ are identical. Assume now that $B : [0,T] \times \R^d \times \CP \rightarrow \R^d$ is a given bounded semi-Markov function (in the sense of Definition \ref{def:semiMarkovFunction}).

\begin{corollary} \label{co:forwardeq-random}
Suppose our filtered probability space $(\Omega,\F,\FF,\PP)$ supports an $\FF$-adapted continuous $\P(\R^d)$-valued process $\mu$, independent of $(\xi,W)$. Suppose it holds almost surely that, for all $t \in [0,T]$ and $\varphi \in C^\infty_c(\R^d)$,
\begin{align}
\langle \mu_t,\varphi\rangle = \langle \lambda,\varphi\rangle + \int_0^t\left\langle \mu_s, \, B(s,\cdot,\mu) \cdot \nabla\varphi(\cdot) + \tfrac12\Delta\varphi(\cdot)\right\rangle ds, \label{def:ap:forwardEQ-random-MVcase}
\end{align}
Then we may find a continuous process $X$, adapted to the complete filtration generated by the process $(\xi,W_t,\mu_t)_{t \in [0,T]}$, such that
\begin{align*}
dX_t = B(t,X_t,\mu)dt + dW_t, \quad\quad X_0=\xi,
\end{align*}
and also $\L(X_t \, | \, \mu) = \L(X_t \, | \, \F^\mu_t) = \mu_t$ a.s.\ for each $t$, where $\F^\mu_t = \sigma(\mu_s : s \le t)$.
\end{corollary}

\section{Joint measurability of regular conditional laws} \label{se:ap:jointmeasurability}

This section provides the details of a technical point used in various places in the paper, notably in the proof of Lemma \ref{le:projection-semiMarkov}. Therein, we wanted to define a regular conditional law in a way that is jointly measurable with respect to the underlying probability law.
The first lemma in this direction is likely known, but we include a proof. Recall that for a Polish space $E$ we always equip $\P(E)$ with the topology of weak convergence and the corresponding Borel $\sigma$-field.

\begin{lemma} \label{le:ap:measurableversion}
Let $E$ and $E'$ be Polish spaces, and let $\pi : E \rightarrow E'$ be continuous. Then there exists a measurable map $\Gamma : \P(E) \times E' \rightarrow \P(E)$ such that 
\[
\int_E F(x)h(\pi(x))\,m(dx) = \int_E\left(\int_EF\,d\Gamma(m,\pi(x))\right)h(\pi(x))\,m(dx),
\]
for all bounded measurable $F : E \rightarrow \R$ and $h : E' \rightarrow \R$.
\end{lemma}
\begin{proof}
To write this in a more probabilistic notation, let $X : E \rightarrow E$ denote the identity map. What we must find is a version of the regular conditional law $m( X \in \cdot \, | \, \pi(X) = x')$ which is jointly measurable as a function of $(x',m) \in E' \times \P(E)$.
We borrow a construction of \cite[Lemma 3.1]{neufeld2014measurability}. Because $E'$ is Polish, we may find a refining sequence of finite Borel partitions $(A^n_1,\ldots,A^n_n)$ of $E'$ such that $\cup_n\sigma(A^n_1,\ldots,A^n_n)$ generates the Borel $\sigma$-field. For each $n$, define $\Gamma_n : \P(E) \times E' \rightarrow \P(E)$ by
\begin{align*}
\Gamma_n(m,x')(\cdot) = \sum_{k=1}^n\frac{m(\cdot \cap \pi^{-1}(A^n_k))}{m(\pi^{-1}(A^n_k)} 1_{A^n_k}(x'),
\end{align*}
where we adopt the convention $0/0 := 0$. As $E$ is Polish, we may find a countable sequence $(\varphi_k)$ of bounded continuous functions such that $\P(E) \ni m \mapsto (\langle m,\varphi_k\rangle )_{k \in \N} \in \R^\N$ is a homeomorphism to its image. Because the $\sigma$-algebras $\sigma(A^n_1,\ldots,A^n_n)$ increase in $n$ by design, the supermartingale convergence theorem ensures that for each $m \in \P(E)$ the $\lim_n \int\varphi_k \, d\Gamma_n(m,x')$ exists for $m \circ \pi^{-1}$-almost every $x'$ and is a version of the conditional expectation $\E^m[\varphi_k(X) \, | \, \pi(X)=x']$. Now, fixing $x_0 \in E$ arbitrarily, we may set
\[
\Gamma(m,x') := \begin{cases}
\lim_n \Gamma_n(m,x') &\text{if the limit exists} \\
\delta_{x_0} &\text{otherwise,}
\end{cases}
\]
where the limit is in the sense of weak convergence. Then, with the help of the sequence $(\varphi_k)$ from above, we deduce that for each $m \in \P(E)$ the map $x' \mapsto \Gamma(m,x')$ is a version of the regular conditional law $m(X \in \cdot \, | \, \pi(X)=x')$. As $\Gamma_n$ is jointly measurable for each $n$, so too is $\Gamma$. 
\end{proof}

We now turn to the real purpose of this section. In the following, let $\Omega$ be a Polish space.
Let $E$ be a complete and separable metric space, and let $X=(X_t)_{t \in [0,T]}$ be a measurable process. By \emph{measurable} here we mean that the function $X : [0,T] \times \Omega \rightarrow E$ is jointly Borel-measurable. We show next how to construct a version of $P(\cdot \, | \, X_t=x)$ which is jointly measurable in $t$, $x$, and the underlying probability measure $P$. Let $\E^P[\cdot]$ denote expectation with respect to a probability measure $P \in \P(\Omega)$.

\begin{lemma} \label{le:measurableversion-marginalprojection}
There exists a jointly measurable function $\Gamma : [0,T] \times E \times \P(\Omega) \rightarrow \P(\Omega)$ such that, for every bounded measurable function $F : [0,T] \times \Omega \rightarrow \R$ and each $P \in \P(\Omega)$, we have
\[
\E^P[F(t,\cdot) \, | \, X_t] = \int_\Omega F(t,\cdot) \, d\Gamma(t,X_t,P), \ \ P-a.s., \ \ a.e. \ t \in [0,T].
\]
\end{lemma}
\begin{proof}
Consider the measurable space $\overline{\Omega} = [0,T] \times \Omega$ and $\overline{E} = [0,T] \times E$, and define $\pi : \overline{\Omega} \rightarrow \overline{E}$ by $\pi(t,\omega) = (t,X_t(\omega))$. Apply Lemma \ref{le:ap:measurableversion} to find a measurable function $\overline\Gamma : \P(\overline{\Omega}) \times \overline{E} \rightarrow \P(\overline{\Omega})$ such that for each $\overline{P} \in \P(\overline{\Omega})$ it holds that $\overline\Gamma(\overline{P} ,\cdot)$ is a version of the conditional law $\overline{P} (\cdot \, | \, \pi)$. Let $U$ denote the uniform probability measure on $[0,T]$. Then, for $P \in \P(\Omega)$ and bounded measurable functions $F : \overline{\Omega}  \rightarrow \R$ and $h : \overline{E} \rightarrow \R$, we have
\begin{align*}
\frac{1}{T}&\E^P\int_0^TF(t,\cdot)h(t,X_t)dt  \\
	&= \int_{\Omega}\int_0^TF(t,\omega)h(\pi(t,\omega))U(dt)P(d\omega) \\
	&= \int_{[0,T] \times \Omega} h(\pi(t,\omega)) \left(\int_{\overline{\Omega}} F\,d\overline\Gamma(U \times P,\pi(t,\omega))\right) (U \times P)(dt,d\omega) \\
	&= \int_{[0,T] \times E} h(t,x) \left(\int_{\overline{\Omega}} F\,d\overline\Gamma(U \times P,(t,x))\right) (U \times P) \circ \pi^{-1}(dt,dx) \\
	&= \int_{[0,T] \times E} h(t,x) \left(\int_{\overline{\Omega}} F\,d\Gamma(P,t,x)\right) (U \times P) \circ \pi^{-1}(dt,dx),
\end{align*}
where we define $\widetilde\Gamma : \P(\Omega) \times \overline{E} \rightarrow \P(\overline{\Omega})$ by setting $\widetilde\Gamma(P,t,x) := \overline\Gamma(U \times P, (t,x))$. Note next that
\[
\int_{[0,T] \times E} g(t,x) (P \times U) \circ \pi^{-1}(dt,dx) = \frac{1}{T}\E^P\int_0^Tg(t,X_t)dt,
\]
for any bounded measurable $g : \overline{E} \rightarrow \R$. Hence, the above becomes
\begin{align*}
\frac{1}{T}\E^P\int_0^TF(t,\cdot)h(t,X_t)dt &= \frac{1}{T}\E^P\int_0^Th(t,X_t) \left(\int_{\overline{\Omega}} F\,d\widetilde\Gamma(P,t,X_t)\right) dt.
\end{align*}
For a measure $\overline{P} \in \P(\overline\Omega) = \P([0,T] \times \Omega)$, let $\overline{P}^T$ and $\overline{P}^\Omega$ denote the $[0,T]$ and $\Omega$ marginals, respectively.
By choosing $F$ depending only on $t$, we find that $\widetilde\Gamma(P,t,X_t)^T = \delta_t$ a.s.\ for a.e.\ $t \in [0,T]$.
Finally, define $\Gamma : [0,T] \times E \times \P(\Omega) \rightarrow \P(\Omega)$ by marginalizing, e.g., setting
\[
\Gamma(t,x,P) := \widetilde\Gamma(P,t,x)^\Omega.
\]
Then
\[
\widetilde\Gamma(P,t,X_t) = \Gamma(t,X_t,P) \times \delta_t, \ \ a.s., \ \ a.e. \  t\in [0,T],
\]
and we find
\begin{align*}
\frac{1}{T}\E^P\int_0^TF(t,\cdot)h(t,X_t)dt &= \frac{1}{T}\E^P\int_0^Th(t,X_t) \left(\int_{\Omega} F(t,\cdot)\,d\Gamma(t,X_t,P)\right) dt.
\end{align*}
This is enough to complete the proof (see \cite[Lemma 5.2]{brunick2013mimicking}).
\end{proof}

\section{Conditional means of random measures} \label{se:ap:meanmeasures}

This section gives some details regarding one additional technical point, relevant in various applications of the Markovian projection Theorem \ref{th:markovprojection} in settings involving relaxed controls, which is to construct a measurable version of the conditional mean of a random measure.
This is formalized in the following lemma, stated in a setting abstract enough to allow for the various applications we have in mind.

\begin{lemma} \label{le:meanmeasures}
Suppose $\Gamma : E \mapsto \P(\Omega)$ is a measurable map. Suppose also that $K : \Omega \rightarrow \P(A)$ is measurable. Then there exists a measurable function $\Lambda : E \rightarrow \P(A)$ such that, for every bounded measurable function $\varphi : E \times A \rightarrow \R$ and every $x \in E$, we have
\begin{align*}
\int_A \varphi(x,a)\,\Lambda(x)(da) = \int_\Omega \left(\int_A \varphi(x,a)\,dK(\omega)(da)\right)\Gamma(x)(d\omega).
\end{align*}
\end{lemma}
\begin{proof}
We use the following well known fact: For any Polish space $E$, the Borel $\sigma$-field on $\P(E)$ coincides with the $\sigma$-field generated by the collection of maps $\P(E) \ni m \mapsto \int \varphi \,dm \in \R$, where $\varphi$ ranges over bounded Borel-measurable real-valued functions of $E$. (See \cite[Corollary 7.29.1]{bertsekasshreve}.)
Define $\Lambda(x)(S)$ for Borel sets $S \subset A$ and $x \in E$ by setting
\begin{align*}
\Lambda(x)(S) = \int_\Omega  K(\omega)(S) \,\Gamma(x)(d\omega).
\end{align*}
For each $x \in E$, it is clear that $\Lambda(x)(\cdot)$ is a probability measure on $A$. On the other hand, for each Borel set $S \subset A$, the map $\omega \mapsto K(\omega)(S)$ is Borel measurable in light of the above fact, and thus so is  $x \mapsto \Lambda(x)(B)$. We conclude that $\Lambda$ defines a measurable map from $E$ to $\P(A)$, and the claimed identity holds whenever $\varphi(x,a) = 1_S(a)$ for a Borel set $S \subset A$. It is straightforward to extend this to any $\varphi$ of the form $\varphi(x,a)=\psi(a)$, for $\psi : A \rightarrow \R$ bounded and measurable. Because the identity holds pointwise, for each $x \in E$, we can then extend to general $\varphi=\varphi(x,a)$.
\end{proof}

One of the main purposes of the abstract considerations of Sections \ref{se:ap:jointmeasurability} and \ref{se:ap:meanmeasures} is the following:

\begin{lemma} \label{le:conditional-marginal-meanmeasure}
Suppose we are given, on some Polish probability space $(\Omega,\F,\PP)$, a measurable $\P(A)$-valued process $\beta=(\beta_t)_{t \in [0,T]}$ as well as a measurable $E$-valued process $X=(X_t)_{t \in [0,T]}$.
Then there exists a jointly measurable function $\widehat\beta : [0,T] \times E \rightarrow \P(A)$ such that, for each bounded measurable function $\varphi : [0,T] \times E \times A \rightarrow \R$, we have
\[
\int_A \varphi(t,X_t,a) \, \widehat\beta(t,X_t)(da) = \E\left[\left. \int_A \varphi(t,X_t,a) \, \beta_t(da)  \, \right| \, X_t \right], \ \ a.s., \ \ a.e. \ t \in [0,T].
\]
We may write this with the suggestive \emph{mean measure} notation,
\[
\widehat\beta(t,X_t) = \E[\beta_t \, | \, X_t].
\]
\end{lemma}
\begin{proof}
By Lemma \ref{le:measurableversion-marginalprojection}, a version of the regular conditional law $\Gamma(t,x) = \PP(\cdot \, | \, X_t=x) \in \P(\Omega)$ can be constructed which is jointly measurable in $(t,x)$. Then simply apply Lemma \ref{le:meanmeasures}.
\end{proof}

Lastly, in the proof of Lemma \ref{le:projection-semiMarkov}, we need the following:

\begin{lemma} \label{le:conditional-marginal-meanmeasure-2}
Suppose we are given, on some Polish probability space $(\Omega,\F,\PP)$, a jointly measurable function $\beta : [0,T] \times \R^d \times \Omega \rightarrow \P(A)$, as well as a continuous $E$-valued process $\eta=(\eta_t)_{t \in [0,T]}$. Let $\F^\eta_t=\sigma(\eta_s : s \le t)$. Then there exists a semi-Markov function $\Lambda : [0,T] \times \R^d \times \CE \rightarrow \P(A)$ such that
\begin{align*}
\Lambda(t,X_t,\eta) = \E[ \beta(t,X_t,\cdot) \, | \, \F^\eta_t], \ \ a.s., \text{ for each } t \in [0,T].
\end{align*}
That is, for each bounded measurable function $\varphi : [0,T] \times \R^d \times \CE \times A \rightarrow \R$, we have
\begin{align*}
\int_A \varphi(t,X_t,\eta,\cdot) \,d\Lambda(t,X_t,\eta) = \E\left[\left. \int_A\varphi(t,X_t,\eta,\cdot) \,d\beta(t,X_t,\cdot) \, \right| \, \F^\eta_t\right], \ \ a.s., \text{ for each } t \in [0,T].
\end{align*}
\end{lemma}
\begin{proof}
This will follow immediately from Lemma \ref{le:meanmeasures} once we can construct a version of the conditional law $\PP(\cdot \, | \, \F^\eta_t)$ which is progressively measurable in $(t,\eta)$. In fact, noting that $\F^\eta_t = \sigma(\eta_{\cdot \wedge t})$, this follows from Lemma \ref{le:measurableversion-marginalprojection} applied with $X=(X_t)_{t \in [0,T]}$ therein given by the stopped process $(\eta_{\cdot \wedge t})_{t \in [0,T]}$.
\end{proof}

\bibliographystyle{amsplain}
\bibliography{closed-loop-bib}

\end{document}